\documentclass[11pt,reqno]{amsart}

\usepackage{amssymb}
\usepackage{amsmath}
\usepackage{amsthm}
\usepackage[all]{xy}
\usepackage{enumerate}
\catcode`\@=11

\long\def\@savemarbox#1#2{\global\setbox#1\vtop{\hsize\marginparwidth 
  \@parboxrestore\tiny\raggedright #2}}
\newcommand\lref[1]{\ref{#1}%
\@ifundefined{r@DisplaY #1}{}{ (#1)}}% Prints label as well as
\newcommand\fakelabel[2]{\@bsphack\if@filesw {\let\thepage\relax
   \newcommand\protect{\noexpand\noexpand\noexpand}%
\xdef\@gtempa{\write\@auxout{\string
      \newlabel{#1}{{#2}{\thepage}}}}}\@gtempa
   \if@nobreak \ifvmode\nobreak\fi\fi\fi\@esphack}

\def\SL@margintext#1{{\showlabelsetlabel{\tiny\{\SL@prlabelname{#1}\}}}}
\catcode`\@=12

\theoremstyle{plain}
\newtheorem{theorem}{Theorem}[section]
\newtheorem{corollary}[theorem]{Corollary}
\newtheorem{lemma}[theorem]{Lemma}
\newtheorem{proposition}[theorem]{Proposition}

\newtheorem*{theorem*}{Theorem}
\newtheorem*{lemma*}{Lemma}
\newtheorem*{claim*}{Claim}

\theoremstyle{definition}

\numberwithin{equation}{section}

\newcommand{\BC}{\mathbb C}
\newcommand{\BH}{\mathbb H}
\newcommand{\C}{\BC}
\newcommand{\BR}{\mathbb R}
\newcommand{\R}{\BR}

\newcommand{\BM}{\mathbb M}

\newcommand{\CA}{\mathcal A}		
\newcommand{\CC}{\mathcal C}		
\newcommand{\CE}{\mathcal E}		\newcommand{\CF}{\mathcal F}

\newcommand{\CM}{\mathcal M}		
\newcommand{\MM}{\CM}
\newcommand{\CO}{\mathcal O}		
		
\newcommand{\CS}{\mathcal S}		\newcommand{\CT}{\mathcal T}

\newcommand{\actson}{\curvearrowright}
\newcommand{\D}{\partial}
\newcommand{\cover}{\widetilde}
\newcommand{\closure}{\overline}
\newcommand{\ep}{\epsilon}
\newcommand{\hhat}{\widehat}
\newcommand{\boundary}{\D}
\newcommand{\union}{\cup}

\newcommand{\intersect}{\cap}
\newcommand{\KK}{{\mathcal K}}
\newcommand{\XX}{{\mathcal X}}

\DeclareMathOperator{\PSL}{PSL}		%	Spezielle lineare Gruppe
\DeclareMathOperator{\Isom}{Isom}	%	Isometrien einer Mf
\DeclareMathOperator{\ML}{\mathcal{ML}}
\DeclareMathOperator{\PML}{\mathcal{PML}}
\DeclareMathOperator{\EL}{\mathcal{EL}}

\DeclareMathOperator{\Mod}{\mathrm{Mod}}

\DeclareMathOperator{\diam}{\mathrm{diam}}

\DeclareMathOperator{\base}{base}
\newcommand\Hyp{{\mathbb H}}

\title[Bounded combinatorics and uniform models]{Bounded combinatorics and uniform models for hyperbolic 3-manifolds}
\date{\today}
\author{Jeffrey Brock, Yair Minsky, Hossein Namazi, Juan Souto}
\thanks{The first three authors have been partially supported by NSF grants DMS-1207572, DMS-1005973, and DMS-0852418. Juan Souto has been partially supported by NSERC Discovery and Accelerator Supplement grants.}
\begin{document}

\maketitle
%%\setcounter{tocdepth}{1}
%%\tableofcontents

\begin{abstract}
Bounded-type 3-manifolds arise as combinatorially bounded gluings of irreducible 3-manifolds chosen from a finite list.  We prove effective hyperbolization and effective rigidity for a broad class of 3-manifolds of bounded type and large gluing heights.  Specifically, we show the existence and uniqueness of hyperbolic metrics on 3-manifolds of bounded type and large heights, and prove existence of a bilipschitz diffeomorphism to a combinatorial model described explicitly in terms of the list of irreducible manifolds, the topology of the identification, and the combinatorics of the gluing maps.
\end{abstract}

\section{Introduction} \label{sec: intro}

Taken together with Mostow's Rigidity Theorem, 
Perelman's celebrated proof of Thurston's {\em Geometrization Conjecture}
settles the existence and uniqueness of finite-volume hyperbolic structures on
3-manifolds in terms of simple topological criteria.
Still missing from this picture, however, is a complete, explicit and effective
means of relating geometric features of a hyperbolic 3-manifold to its
topological description and vice versa. 

Of course such means are available in many particular settings,
notably:
\begin{enumerate}
\item Thurston's Dehn-Filling 
Theorem, which describes infinite families of hyperbolic 3-manifolds
in terms of the geometry of a single cusped one and topological
filling data,
\item Gromov's Volume Theorem, which
relates a homological invariant to hyperbolic volume, and
\item the Ending Lamination Theorem, which provides a bilipschitz model
  for fibered 3-manifolds $M_\psi$ in terms of the stable and
  unstable laminations for the monodromy $\psi$;
\end{enumerate} 
and there are many others.

In this paper we seek to extend the last of these, studying families
of hyperbolic 3-manifolds obtained by gluing pieces of a predetermined
type using boundary identifications with certain topological
constraints.  For such families we describe {\em uniform models}:
explicit metrics on the manifolds that are guaranteed to be uniformly
bilipschitz to the hyperbolic metrics.

To be more specific, consider a finite collection $\MM$ of {\em
  decorated 3-manifolds}, which are compact, oriented, irreducible,
atoroidal 3-manifolds with no torus or sphere boundary components,
equipped with complete boundary markings. An {\em $\MM$-gluing} is a
3-manifold $X$ obtained from copies of elements of $\MM$ by gluing
paired boundary components. We define a notion of {\em height} for
each boundary pairing, namely, the distance between the given boundary
markings in the curve complex of the boundary components to be
identified.  We further define, for a positive number $R$, a notion of
``$R$-bounded combinatorics,'' a restriction on the
complexity of the gluing maps analogous to a bound on
continued fraction coefficients of real numbers.

In Section \ref{sec:bounded combinatorics} we make these definitions
precise and describe for each $\MM$-gluing $X$ with $R$-bounded
combinatorics a {\em model metric}, denoted $\BM_X$.  In this metric
$X$ is built from a fixed metric on each element of $\MM$, with
$I$-bundles interpolating between corresponding boundary components.
Here, each such $I$-bundle is equipped with a metric that closely
resembles the universal curve over an appropriate (thick)
Teichm\"uller geodesic segment.

Our main theorem shows such models are uniform:

\medskip\par\noindent
{\bf Theorem \ref{bilipschitz models for bounded type manifolds}.}
{\em	Let $\CM$ be a finite collection of decorated manifolds and
        fix $R>0$. There exist $D$ and $K$ such that, for any
        $\CM$-gluing $X$ with $R$-bounded combinatorics and all heights greater than $D$, $X$
        admits a unique hyperbolic metric $\sigma$. Moreover, there exists
        a $K$-bilipschitz homeomorphism from the model $\BM_X$ to
        $(X,\sigma)$ in the correct isotopy class, whose image is the
        complement of the rank 2 cusps in $X$. 
}
\par\medskip

We note that while $\MM$ is finite we do not limit the number of copies
of each element of $\MM$ that can be used in a gluing; in fact we
allow the consideration of manifolds with infinitely generated
fundamental groups. It is in these cases that the uniqueness of the
hyperbolic metric $\sigma$ is not a direct consequence of Mostow rigidity.

We remark also that the model $\BM_X$ maps to the complement of the
cusps merely because we wish to make $\BM_X$ out of compact
pieces. It is a simple matter to extend the models to include regions
that correspond to the cusps.

In a followup paper \cite{BMNS:Heegaard} we use this theorem to
topologically characterize fixed genus Heegaard splittings of
hyperbolic 3-manifolds with a lower bound on injectivity radius.
We use gluings with $R$-bounded combinatorics to define a 
property of a Heegaard splitting which we call {\em 
$R$-bounded compressions}.
Indeed, all such splittings are obtained as finite $\MM$-gluings with
$R$-bounded combinatorics where $\MM$ is a finite
collection of decorated manifolds determined depending on
the genus of the splitting. We prove the 
following theorem, where the first part is a simple
corollary of the above theorem.

\begin{theorem*}\cite{BMNS:Heegaard}
  Given $g$ and $R$, there exists $\ep>0$ so that all but finitely
  many closed 3-manifolds with a genus $g$ Heegaard splitting with
  $R$-bounded compressions admit a hyperbolic metric with injectivity
  radius at least $\ep$.

Conversely, given $g$ and $\ep>0$, there exists $R$ so that for all
but finitely many closed hyperbolic 3-manifolds with injectivity radius
at least $\ep$, every genus $g$ Heegaard splitting has $R$-bounded
compressions.
\end{theorem*}

These results are applied in \cite{OS} to analyze subgroups
of mapping class groups generated by pairs of handlebody groups
associated to Heegaard splittings. When these splittings satisfy
$R$-bounded combinatorics and large height conditions,
Ohshika-Sakuma show among other things that such groups act
with non-empty domain of discontinuity on the sphere of projective
measured laminations.

A further application of \cite{BMNS:Heegaard} described in
\cite[Remark 2.8]{BD} employs models for Heegaard splittings to show
the existence of families of knot complements in $S^3$ whose
$(1,n)$-Dehn fillings are integer homology spheres with larger and
larger injectivity radius on a larger and larger portion of their
volume.  Such examples give rise to integer homology spheres that {\em
  Benjamini-Schramm converge} to $\mathbb{H}^3$, answering a question
of Bergeron in the negative.

A crucial feature of our discussion is the possible presence of {\em
  compressible} boundary components in the elements of $\MM$.  When
elements of $\MM$ are constrained to have incompressible boundary, the
structure of a gluing, its fundamental group, and its hyperbolic
structure are considerably more accessible.  Indeed, the existence of
hyperbolic structures in this case follows from Thurston's original
hyperbolization theorem.  
Uniform models in the incompressible setting
are available even without the $R$-bounded combinatorics condition,
using the technology developed in \cite{MM2,ELC1,ELC2}.

The compressibility of boundary components presents immediate
difficulties even in determining whether a gluing $X$ satisfies the
topological conditions of the geometrization theorem, since the
fundamental groups of the pieces of $\MM$ may fail to inject into that
of $X$. In Namazi \cite{Na05} and Namazi-Souto \cite{NS09}, this issue
is addressed in the special case of gluings of handlebodies with
restrictions related to but stronger than the ones we impose here. Our
work here extends many of the ideas from those two papers.

One should also mention the work of Lackenby \cite{La02}, in which
conditions are given on a gluing of a handlebody to a manifold with
incompressible, acyclindrical boundary which guarantee that the result
of the gluing admits (given the geometrization theorem) a hyperbolic
structure. Lackenby's conditions, while considerably more permissive
than ours, provide only topological conclusions without further
implications for the geometry of the hyperbolic structure.

We note that the $R$-bounded combinatorics condition is quite
restrictive and is by no means the final word in understanding
hyperbolic structures on $\MM$-gluings.  Without a bound on the
combinatorics the appearance of a non-trivial thick-thin decomposition
of the manifold leads to considerable combinatorial
difficulties. While such intricacies are completely understood in the
setting of incompressible boundary, the full picture for
compressible boundary remains incomplete.

The other substantial restriction on these results is the finiteness
of the set $\MM$, on which the uniformity of our models strongly
depends.  When elements of $\MM$ arise from compression bodies in a
generalized Heegaard splitting, the finiteness of $\MM$ corresponds to
a bound on the genus of the splitting.  Clearly, a ``generic'' family
of 3-manifolds would not arise as gluings of a fixed set $\MM$, but
our techniques do not presently suggest a way to handle the most
general situation.  We consider this to be a fascinating and
challenging direction for further research.

\subsection{Outline of the proof}
Compactness theorems in representation spaces play a significant role
in our proof -- in fact a double role.

In Section \ref{sec:convergence theorems} we prove Theorem
\ref{eventually faithful convergence}, which gives conditions under
which a sequence of discrete representations $\rho_n:\pi_1(M) \to
\PSL_2(\C)$ has a convergent subsequence, where $M$ is a decorated
3-manifold.  In this theorem, we assume that for each boundary
component of $M$ there is a sequence of curve systems whose lengths
under $\rho_n$ stay bounded, whose heights grow without bound, and
whose combinatorics satisfy the ``$R$-bounded combinatorics''
condition detailed in Section \ref{subsec: bounded combinatorics}.
The theorem elaborates on the convergence and compactness theorems of
Thurston (the Double Limit Theorem, and the compactness theorem for
representations of fundamental groups of incompressible-boundary
manifolds) and the work Kleineidam-Souto (for representations of
fundamental groups of handlebodies and compression bodies).  It allows
weaker hypotheses in one essential respect -- the representations
$\rho_n$ are discrete but {\em not necessarily faithful}. We require
only {\em eventual faithfulness}, that is that every nontrivial
element of $\pi_1(M)$ is eventually not in the kernel of $\rho_n$.

This limiting result is used in Section \ref{sec:uniform immersions}
to find uniform immersions of pieces of $\MM$ equipped with model
metrics into hyperbolic manifolds in homotopy classes that satisfy
suitable conditions, specifically $R$-bounded combinatorics and {\em
  almost-injectivity}. Theorem \ref{bilipschitz embedding of models}
is the main result of this type. Its somewhat intricate hypothesis
involves a discrete and ``sufficiently'' injective representation in
which some marking data on the boundary of a decorated manifold admits
length upper bounds, large heights and bounded combinatorics, and its
conclusion is a uniformly locally bilipschitz immersion of the model
manifold into the quotient hyperbolic manifold.  Uniformity is
obtained from a contradiction argument, involving sequences of
examples for which Theorem \ref{eventually faithful convergence}
provides the limit that yields a contradiction.

The first use of Theorem \ref{bilipschitz embedding of models}, in
Section \ref{sec:bilipschitz models for convex cocompact}, is to build
and model hyperbolic metrics on the pieces of $\MM$ with prescribed
structure in the ends. That is, we consider convex-cocompact
representations with suitably chosen conformal boundaries, and use
Theorem \ref{bilipschitz embedding of models} to obtain bilipschitz
models for them.  In Section \ref{sec:nearly hyperbolic gluings}, we
use these structures to construct {\em negatively curved} (but not
necessarily hyperbolic) structures on our $\MM$-gluings, by arranging
for the gluing maps to be nearly isometric in suitable regions of the
ends of each piece and then interpolating between the hyperbolic
metrics via convex combination.  The outcome, Theorem \ref{existence
  of nearly hyperbolic metrics}, states that each gluing $X$ with
$R$-bounded combinatorics and sufficiently large heights possesses a
metric of negative curvature (pinched close to $-1$) that is 
hyperbolic outside of bounded regions associated to the end-gluings
and nearly hyperbolic in those regions.

It is worth pausing to explain why the proof is not essentially
complete at this point in the argument: while it appears that the
constructed {\em nearly hyperbolic} metric on $X$ should be close to
an actual hyperbolic metric, there is no tool for perturbing the
nearly hyperbolic metric to an actual hyperbolic metric in our
setting.  A theorem of Tian \cite{tian} provides such a perturbation
provided given small $L^\infty$ bounds on the deviation from constant
sectional curvature (which we do have), as well as small $L^2$ bounds
on the traceless Ricci curvature.  But the latter criterion is not
satisfied here as we assume no bound on the number of pieces in $X$,
and hence on the number of gluing regions, each of which contributes
to the $L^2$ norm.  We must therefore employ a somewhat indirect
argument to compare the model metrics to the true hyperbolic metrics.

Theorem \ref{existence of nearly hyperbolic metrics} gives us two
pieces of information. First, it implies 
that such gluings $X$ (in the closed case) indeed satisfy the topological
conditions for the existence of a hyperbolic
structure; hence, by Perelman's theorem, $X$ admits such a structure.
Second, it implies by a simple geometric argument that the
inclusion of each piece of $\MM$ into $X$ is {\em almost injective}
on $\pi_1$ -- that is, as heights go to infinity, any given element is
eventually not in the kernel.

Our next step is to use Theorem \ref{bilipschitz embedding of models}
to obtain uniform estimates on the inclusions
of the pieces of $\MM$ into the hyperbolic structures on gluings with
sufficiently large heights. In addition to almost-injectivity,
however, we also need some {\em a priori} control: namely, for each gluing
surface we must obtain upper bounds on the lengths, in the hyperbolic
metric on $X$, of certain curves coming from the model pieces on the
two sides of the surface. Such bounds are actually easy to obtain in
when the surface is compressible (we can use compressible curves whose
length is 0), but in general we must return to some of the compactness
theorems of Thurston, with a certain variation.

Thurston's {\em only windows break} theorem gives such {\em a priori} bounds on
the lengths of certain curves in the boundary of a 3-manifold, for
every discrete faithful representation of its fundamental group. The
curves are identified using the JSJ decomposition of the manifold.
In our setting we have representations that are not faithful, so to
recover the bounds we re-organize the gluings to obtain configurations
in $X$ restricted to which the representations are injective, and for
which we can control the JSJ decompositions. This is done in Section
\ref{sec: stability of JSJ decomposition}
by means of a stability theorem for JSJ decompositions of certain
families of gluings.

Finally, in Section \ref{sec:bilipschitz models for bounded type
  manifolds} we put these ingredients together, getting 
bilipschitz immersions of the pieces, which combine to give the
final bilipschitz equivalence between the model and the hyperbolic
metric. 

In an appendix, we review Thurston's ``only windows break'' theorem and give an
alternate proof of it, paying particular attention to the fact that
the constants in this theorem depend only on the boundary genus and
not the 3-manifold -- a crucial fact in the stability and {\em a priori}
bounds results of Section \ref{sec: stability of JSJ decomposition}.

We should also remark on the use of Perelman's hyperbolization theorem
in the above outline: when the gluing $X$ has sufficiently many pieces
$X$ must be Haken, so Perelman's Theorem may be replaced with
Thurston's original hyperbolization theorem.  This point is discussed
in more detail at the end of the proof of Theorem \ref{bilipschitz
  models for bounded type manifolds}.

Moreover, in the case that $X$ is composed of infinitely many pieces,
neither of the geometrization theorems, nor Mostow rigidity, can be
directly applied. Existence of the hyperbolic structures in this case
is obtained by a limiting process using the existence of our
bilipschitz models to maintain control, and uniqueness is obtained by
an appeal to McMullen's rigidity theorem, where again the models
provide the needed geometric hypotheses.

\subsection*{Acknowledgements}
The authors greatly benefitted with conversations from many people 
during the course of this work, particularly Dick Canary, Ken Bromberg, 
Ian Biringer, and Cyril Lecuire. Hossein Namazi would like also to 
acknowledge the hospitality of Yale University's mathematics department 
for their hospitality while this work was being completed.

\section{Bounded Combinatorics and Decorated Manifolds} \label{sec:bounded combinatorics}

In this section we recall some background associated to surfaces,
develop the definitions and notation for {\em decorated 3-manifolds},
and make explicit our notion of gluings with {\em bounded combinatorics}.

\subsection{Curves and Markings}\label{subsec: markings}

Given a compact orientable surface $F$, various complexes are now
commonly used to encode and organize the isotopy classes of simple
closed curves and multicurves on $F$.  We briefly introduce these and
refer the reader to \cite{MM1, MM2} for detailed discussion.  The {\em
  complex of curves} of $F$, denoted $\CC(F)$, is a simplicial complex
whose vertices are associated bijectively with isotopy classes of
homotopically nontrivial and nonperipheral, or {\em essential}, simple
closed curves on $F$.  In particular, with the natural path metric in
which each simplex is standard, $\CC(F)$ is {\em Gromov hyperbolic}
\cite{MM1} when $\CC(F)$ is infinite. We use $d_F(\cdot,\cdot)$ to
represent this natural path metric.

When $Y \subset F$ is an essential subsurface of $F$ (its boundary
consists of essential loops or boundary components of $F$)
the {\em subsurface projection}
$$\pi_Y: \CC(F)\to \mathcal{P} (\CC(Y))\cup\{\emptyset\}$$ associates to each
simplex in $\CC(F)$ a uniformly bounded diameter subset of vertices in
$\CC(Y)$ -- if $\alpha \in \CC(F)$ fails to intersect $Y$, then
$\pi_Y(\alpha) = \emptyset$.  The subsurface projection behaves
similarly to a nearest point projection.  Specifically, when $Y$ is
not an annulus and $\alpha$ is a vertex of $\CC(F)$, $\pi_Y(\alpha)$
is obtained via a surgery from the components of the essential
intersection of $\alpha$ with $Y$.  Note that when $Y$ is an annulus
the definition $\CC(Y)$ and $\pi_Y$ are slightly different; details
are to be found in \cite{MM2}.

Two subsurfaces are considered the same if they are isotopic to each
other. Given subsets $\alpha$ and $\beta$ of $\CC(F)$, 
because the subsurface projection is coarsely well-defined, the
definition
$$d_Y(\alpha,\beta) = \diam_{\CC(Y)}(\pi_Y(\alpha)\cup\pi_Y(\beta))$$
plays the role of a coarse distance `relative to $Y$' when both projections are
nonempty.

A {\em multicurve} is a collection of pairwise disjoint and
non-parallel essential simple loops on $F$. A {\em marking} $\mu$ on
$F$ is a set $\mu$ of essential simple loops on $F$ consisting of a
multicurve $\base(\mu)$, called the {\em base} of $\mu$, and at most
one {\em transversal} associated to each element of the base. A {\em
  transversal} for $\alpha \in \base(\mu)$ is an essential simple loop
$\beta$ which does not intersect any other element of $\base(\mu)$,
and such that either $\beta$ intersects $\alpha$ exactly once or
$\alpha$ and $\beta$ intersect twice and the regular neighborhood of
$\alpha\cup\beta$ is an essential subsurface of $F$ homeomorphic to a
4-holed sphere.  A {\em complete marking} is one that is maximal: the
base of a complete marking is a pants decomposition and there is a
transversal for every element of the base.  

We frequently refer to a marking as a subset of the set of vertices of 
the curve complex and a multi-curve as a special (partial) marking. 
Note that the diameter of a marking as a subset of $\CC(F)$ is at 
most 2. The distance between two markings $\mu_1, \mu_2$ on $F$ 
is the minimum distance in $\CC(F)$ between an element of $\mu_1$ 
and an element of $\mu_2$.

%%%%%%%%%%%%%%%%%

\subsection{Laminations}\label{subsec: laminations}
We will use the notion of {\em geodesic laminations}, equipped with
{\em transverse measures} and the associated notion of {\em geometric
  intersection} on surfaces; for more precise definitions we refer the
reader to \cite{FLP79, Thu79,CB88}.  In particular, we will refer to
the spaces $\ML(F), \PML(F)$ of {\em measured laminations} and {\em
  projective measured laminations} on $F$ respectively.  The set $\ML$
is equipped with an intersection pairing $i(.,.)$ that continuously
extends the notion of weighted geometric intersection of simple closed
curves equipped with positive real weights. Two elements of $\PML$
{\em intersect} if their projective classes are represented in $\ML$
by elements with nonzero intersection.  Recall that $\PML(F)$ is
homeomorphic to a sphere, and is in particular compact.  The weighted
simple closed curves are dense in $\ML(F)$, and since a simple closed
curve carries a unique transverse measure up to scale, the sphere
$\PML(F)$ naturally contains the set of isotopy classes of essential
simple loops of $F$ as a dense subset.

A lamination is 
{\em filling} if it intersects every essential simple loop on $F$. 
The space $\EL(F)$, the {\em space of ending laminations} of $F$, 
is obtained as a quotient of the set of filling measured laminations 
by forgetting the measure. A theorem of Klarreich \cite{Kla} 
shows that $\EL(F)$ is naturally identified with the Gromov boundary 
of $\CC(F)$. In particular, a sequence of markings 
$(\mu_n)$, as subsets of $\CC(F)$, converges to 
$\lambda\in\EL(F)$ if and only if every $\PML$-limit of a subsequence 
$(\alpha_{n_i})$, with $\alpha_{n_i}\subset \mu_{n_i}$ a component 
of $\mu_{n_i}$, is a projective measured lamination which projects 
to $\lambda$. In this article, we often do not distinguish between 
an element $\lambda\in\EL(F)$ and elements of its pre-image in 
$\PML(F)$ or $\ML(F)$.

%%%%%%%%%%%%%%%%%%%%%%

\subsection{Bounded Combinatorics}\label{subsec: bounded combinatorics}
  Given markings $\mu_1, \mu_2$ on $F$ and constant $R>0$, we say the pair
  $(\mu_1,\mu_2)$ has {\em $R$-bounded combinatorics} if
  $$d_Y(\mu_1,\mu_2)\le R$$ 
for every proper essential subsurface
  $Y\subset F$. This notion and its applications appear in a number of
places, for example  \cite{Min01, Rafi,
 Na05}. The following lemma shows how bounded
  combinatorics controls local behavior in the (locally infinite)
  curve complex.

\begin{lemma}\label{limits of bounded combinatorics fill}
  Given $R>0$, suppose $(\mu_n)$ is a sequence of markings on $F$ so
  that the pair $(\mu_0,\mu_n)$ has $R$-bounded combinatorics for each
  $n$.  Then there is a subsequence  $(\mu_{n_k})$ that is either
  constant or converges to an element in
  the Gromov boundary of $\CC(F)$.
\end{lemma}

\begin{proof}
  Consider the set $\{ \alpha \in \CC(F) \vert \alpha \in \mu_n \}$ of
  simple loops which appear as components of $\mu_n$ for some $n$.  If
  this set is finite, then there is a subsequence of $(\mu_n)$ that is
  constant.  Assume this set is infinite. Since $\PML(F)$ is compact,
  a subsequence $(\mu_{n_k})$ and loops $\alpha_k\subset \mu_{n_k}$
  have the property that $(\alpha_k)$ converges in $\lambda$ in the
  space $\PML(F)$.  Since $i(\alpha_k,\beta_k)\le 2$ for every other
  component $\beta_k\subset\mu_{n_k}$, the continuity of the
  intersection pairing $i(\cdot,\cdot)$ on $\ML(F)$, guarantees that
  each limit $\lambda'$ of $(\beta_k)$ satisfies $i(\lambda,\lambda') = 0$.
  Klarreich's characterization of the Gromov boundary of $\CC(F)$,
  then, reduces the problem to showing that $\lambda$ is filling.

  If $\lambda$ is not filling then either $\lambda$ is supported on a
  multicurve, or it has a component that fills a proper essential
  subsurface $W$.  In the latter case, the projections
  $\pi_W(\mu_{n_k})$ to $\CC(W)$ converge to $\lambda$, in the sense
  of Klarreich's theorem, from which it follows that
  $d_W(\mu_{n_k},\mu_0)\to\infty$ as $k\to\infty$.  This contradicts
  the assumption that $d_W(\mu_{n_k},\mu_0)<R$.  If $\lambda$ is a
  multi-curve, then convergence of the infinite collection
  $(\alpha_k)$ to $\lambda$ guarantees that there is a component
  $\lambda_0$ of $\lambda$ about which $\alpha_k$ `spirals': geodesic
  representatives of $\alpha_k$ intersect a collar neighborhood $W$ of
  the geodesic representative of $\lambda_0$ in longer and longer
  segments.  This implies that $d_W(\mu_{n_k},\mu_0)\to\infty$ as
  $k\to\infty$, which is again a contradiction.
\end{proof}

%%%%%%%%%%%%%%%%%

\subsection{Pared manifolds} \label{subsec: pared manifolds} A {\em
  pared manifold} is a pair $(M,P)$ where $M$ is a compact orientable
atoroidal 3-manifold and $P\subset \D M$ is a collection of essential
annuli and tori containing every toroidal boundary component of $M$
with the property that if $A$ is a compact annulus, every homotopically
nontrivial embedding of pairs $(A,\D A)\to (M,P)$ is homotopic through
maps of pairs to a map into $P$.  We use the notation $\D_0 (M,P)$ to
denote the closure of $\D M\setminus P$, and a component of $\D_0
(M,P)$ is called a {\em free side}. We also use the notation $\D P$ to
denote the boundary components of the annular components of $P$.  By
an {\em essential} annulus or disk in $(M,P)$, we mean a
$\pi_1$-injective immersion $f:(A,\D A) \to (M,\D M\setminus \D P)$
with $A$ either a compact annulus or disk that cannot be homotoped
through maps $(A,\D A)\to (M,\D M\setminus \D P)$ to a map with image
in $\D M$.  By the boundary of an essential annulus or disk $f:(A,\D
A)\to (M,\D M\setminus \D P)$, we mean the restriction of $f$ to 
$\D A$.  The pared manifold $(M,P)$ is {\em acylindrical} if it admits 
no essential disk or annulus.

As a special case, an interval bundle $M$ over a compact surface $F$
(possibly with boundary) naturally has the structure of a pared 
manifold $(M,P)$, where $P$ is the union of the annuli that project 
to $\D F$. In this case, we call $\D_0 M$ the {\em horizontal boundary} 
of $M$ and $P$ the {\em vertical boundary.}

Considerations in this article will be constrained to pared manifolds
$(M,P)$ for which
\begin{itemize}
\item the free sides are incompressible, or
\item $P$ consists of the toroidal components of $\D M$ and has no 
annular component. 
\end{itemize}
Given $M$ with unspecified pared locus we take $M$ to have the
structure of a pared manifold with the pared locus the union of the
toroidal components of $\D M$; then $\D_0 M$ is the set of
non-toroidal components of $\D M$.

%%%%%%%%%%%%%%%%%%

\subsection{JSJ-decomposition.}\label{subsec: jsj decomposition}

Let $(M,P)$ be a pared manifold with incompressible free sides. 
The JSJ decomposition \cite{Jaco-Shalen, Johannson} of $(M,P)$ is 
a splitting of $(M,P)$ along a collection $\CA$ of disjoint 
properly embedded essential annuli in $(M,P)$ such that;
  \begin{itemize}
    \item[(a)] if $U$ is a component obtained after cutting $M$ 
			along $ \CA$ then
      \begin{itemize}
        \item either $U$ is a solid torus and $(\CA\cup P)\cap U$ 
				is a collection of parallel non-meridional annuli in $\D U$,
        \item $(U, (\CA\cup P)\cap U)$ is an interval bundle, or
        \item $(U,(\CA\cup P)\cap U)$ is acylindrical, and
      \end{itemize}
    \item[(b)] any essential annulus in $(M,P)$ can be properly 
    isotoped into one of the components of $M\setminus\CA$.
  \end{itemize}
  Moreover if $\CA$ is chosen to be minimal with respect to these
  properties, then $\CA$ is unique up to what is called an
  {\em admissible isotopy}. We refer the reader to \cite{CMc} for
  more on this decomposition.
  We follow Thurston \cite{ThuIII} and
  identify in a JSJ decomposition of the pared manifold $(M,P)$ a
  subset of the characteristic submanifold called the {\em window}.
  The window contains all the interval bundles in the above
  decomposition. %(where a solid torus component with exactly two
  Moreover, we thicken each component of $\CA$ which cannot be properly
  homotoped into an interval bundle, 
  obtaining an interval bundle with annulus base, and include these
  in the window. 

The union of the horizontal boundaries of the components of the 
window is a subsurface of boundary of $M$. The components of the 
boundary of this subsurface that  are not isotopic into 
components of $P$ are called {\em window frames}.

%%%%%%%%%%%%%%%%%%%

\subsection{Compression bodies}\label{subsec: compression bodies}

A {\em compression body} is a compact orientable 3-manifold with 
a boundary component, called the {\em exterior boundary} and 
denoted by $\D_e C$, so that the homomorphism 
$\pi_1(\D_e C) \to \pi_1(C)$ induced by the inclusion is surjective. 
Every other component of $\D C$ is called an {\em interior boundary 
component}. The compression body $C$ is {\em nontrivial} if the 
homomorphism $\pi_1(\D_e C) \to \pi_1(C)$ is not injective and 
$C$ is not a solid torus. 

Given a compact orientable irreducible 3-manifold $M$ and a 
component $E$ of $\D M$, following Bonahon \cite{Bo83}, we 
define $C_E$, {\em the relative compression body} associated to 
$E$, to be an embedded compression body with $E$ as the exterior 
boundary whose inclusion into $M$ is $\pi_1$-injective.
and is nontrivial when $M$ is not a solid torus and $E$ is
compressible. After isotopy, we assume the nontrivial relative 
compression bodies are mutually disjoint. The complement of the 
union of all nontrivial relative compression bodies is a submanifold 
with incompressible boundary; the union of the components of this 
complement, which are not trivial $I$-bundles, is referred to 
as the {\em incompressible core} of $M$.

%%%%%%%%%%%%%%%%%%%%

\subsection{Binding laminations and the Masur Domain}\label{subsec: masur domain}

Given a compression body $C$, we say a homotopically nontrivial 
simple loop on the exterior boundary $\D_e C$ is a {\em meridian} 
if it bounds a disk in $C$. The set of meridians on $\D_e C$ is 
denoted by $\Delta(C)$ and is considered as a subset of the 
vertices of $\CC(\D_e C)$. Masur-Minsky \cite{MM3} showed that 
$\Delta(C)$ is $K$-quasi-convex in $\CC(\D_e C)$, with $K$ 
depending only on the topology of $\D_e C$.
Masur \cite{Mas86} and Otal \cite{Ota88} studied a subset 
$\CO(C)$ of $\ML(\D_e C)$ called the {\em Masur domain}. 
We refer the reader to \cite{Mas86, Ota88, Le06b} for more on 
this set in general. 
We only state the following lemma which for us serves as the 
definition of the set of filling laminations in the Masur 
domain. The lemma is a simple consequence of the definition 
of the Masur domain. 
  \begin{lemma}\label{filling and masur domain}
    A filling lamination $\lambda$ on the exterior boundary 
$\D_e C$ of a non-trivial compression body $C$ is in the Masur 
domain if and only if $\lambda$ (as a point of the Gromov boundary of 
$\CC(\D_e C)$) is not a limit of any sequence of meridians. \hfill \qed
  \end{lemma}

It is immediate from the definition of the Masur domain and 
in the case of filling laminations, the above lemma, that 
belonging to the Masur domain depends only on the supporting 
geodesic lamination. Hence it is legitimate to refer to the 
projections of the Masur domain to $\PML(\D_e C)$ and 
$\EL(\D_e C)$ as Masur domain.

Assume $(M,P)$ is a pared manifold. Recall from our convention 
mentioned in \S\ref{subsec: pared manifolds} that either $(M,P)$ 
has incompressible free sides or $P$ is the set of toroidal 
components of $\D M$. In the latter case if $E$ is a compressible 
component of $\D_0 M = \D_0(M,P)$, we use $\CO(E)$ to denote the 
Masur domain of $C_E$, the relative compression body of $M$ 
associated to $E$. We say a lamination $\lambda$ on $\D_0 (M,P)$ 
is {\em binding} if  
  \begin{enumerate}
\item every essential disk or annulus in $(M,P)$ has a boundary 
component on a component $E$ of $\D_0 M$ with $\lambda|_E$, 
the restriction of $\lambda$ to $E$, a filling lamination,
\item if $E$ is a compressible component of $\D_0 M$, then 
$\lambda|_E \in\CO(E)$, and finally
    \item if $(M,P)$ is an $I$-bundle over a compact surface 
$(F,\D F)$, the projection of $\lambda$ to $F$ has transverse 
self intersection.
  \end{enumerate}

Note that by the Loop Theorem and Annulus Theorem, we can 
consider only embedded essential disks or annuli in the 
first condition above. 
We say $\lambda$ is {\em full} if $\lambda|_E$ is a filling lamination for every component $E$ of $\D_0 (M,P)$.

%%%%%%%%%%%%%%%%%%%%%%%

\subsection{Decorated manifolds} \label{subsec:decorated manifolds}

A {\em decorated manifold} is a compact, oriented, irreducible,
atoroidal 3-manifold $M$ with nonabelian fundamental group, and equipped with
a complete marking $\mu$ on the non-toroidal part of the boundary, i.e. the {\em decoration}. A decorated manifold is always considered as a pared manifold with pared locus the union of the toroidal boundary components and $\D_0 M$ the union of the non-toroidal components. Given a decorated manifold $M$, $\mu(M)$ denotes the decoration and if $E$ is a component of $\D_0 M$, we use the notation $\mu(M,E)$, or $\mu_E$ when there is no ambiguity, to denote the restriction of $\mu(M)$ to $E$. 

%%%%%%%%%%%%%%%%%%%%%%%

\subsection{Special decorations on trivial $I$-bundles.}\label{subsec: special decorations}
Trivial $I$-bundles play an important role in our construction and in
a few places have to be treated separately. Suppose $F\times[0,1]$ is
a trivial interval bundle over the closed orientable surface $F$ and
$\mu_0$ and $\mu_1$ are complete markings on $F$. Identifying
$F\times\{0\}$ and $F\times\{1\}$ with $F$,
we obtain decorations $\mu_i$ on $F\times\{i\}$, and 
we let $I_F[\mu_0,\mu_1]$ denote the resulting decorated manifold. 
It will also be convenient to let $I_F[\mu]$ be shorthand for $I_F[\mu,\mu]$.

%%%%%%%%%%%%%%%%%%%%%%

\subsection{Bounded Combinatorics for decorated manifolds}\label{subsec: bounded combinatorics for decorated manifolds}
If $M$ is a decorated manifold and $E\subset \D_0 M$ is a
non-toroidal component of the boundary, we say a marking (possibly a
partial marking) $\ep_E$ on $E$ has {\em $R$-bounded combinatorics}
if the pair  $(\ep_E, \mu_E)$ has $R$-bounded combinatorics and,  when $E$
is a compressible component of $\D_0 M$, we have
\begin{equation}\label{eqn: mu nearest}
 d_E(\mu_E,\ep_E) \le d_E(\Delta(E),\ep_E) + R
\end{equation}
where $\Delta(E)$ is the set of all meridians on $E$. The last
property can be interpreted to mean that $\mu_E$ is $R$-close to a
``projection'' of $\ep_E$ to $\Delta(E)$.

If $\ep$ is a (possibly partial) marking on $\D_0 M$, we say $\ep$ has
{\em $R$-bounded combinatorics with respect to $M$} if each nonempty
restriction $\ep|_E$ of $\ep$ to a component of $\D_0 M$ has
$R$-bounded combinatorics, and in addition:
  \begin{enumerate}
    \item Every essential disk or annulus in $M$ has at least one boundary component on a component $E$ of $\D_0 M$ with $\ep|_E$ nonempty.
    \item If $M$ is a decorated $I$-bundle $I_F[\mu_0,\mu_1]$, we require $\ep$ to have a
      component $\ep_i$ on each boundary component $F\times\{i\}$ of
      $M$, such that  $\mu_0$ and $\mu_1$ lie within
distance $R$ in $\CC(F)$ of a $\CC(F)$-geodesic connecting $\ep_0$
and $\ep_1$.

\item If $M$ is a decorated twisted $I$-bundle, it is double-covered by a product
$I$-bundle, and we require that the lifts of $\ep$ and $\mu(M)$
to this cover satisfy condition (2).

  \end{enumerate}

For every component $E \subset \D_0 M$, we call $d_E(\mu_E, \ep|_E)$ the {\em height of $\ep$ on $E$.} The next lemma explains how binding laminations and markings with bounded combinatorics are related.
  
\begin{lemma}\label{limits of bounded combinatorics bind}
  Suppose $M$ is a decorated manifold and $(\ep_n)$ is a sequence of
  markings on $\D_0 M$ with $R$-bounded combinatorics and heights
  tending to infinity as $n\to\infty$. Then after passing to a subsequence, the sequence $(\ep_n)$
  converges to a binding lamination $\lambda$. Moreover if $\ep_n|_E$ is
  nonempty for every component $E$ of $\D_0 M$ and every $n$, then
  $\lambda$ is full. 
\end{lemma}

\begin{proof}
By lemma \ref{limits of bounded combinatorics fill}, we can pass to a
subsequence and assume for every component $E$ of $\D_0 M$ that
either $(\ep_n|_E)$ converges to a filling lamination $\lambda_E$ or
$\ep_n|_E$ is empty for every $n$. We let
$\lambda=\bigcup_{E\subset\D_0 M}\lambda_E$. The definition of
$R$-bounded combinatorics with respect to $M$ shows that every
essential disk or annulus in $M$ has a boundary component on a
component $E$ of $\D_0 M$ with $\lambda_E$ non-empty and therefore a
filling lamination. To show that $\lambda $ is binding, 
it remains to show that $\lambda_E$ is in the
Masur domain if $E$ is compressible, and when $M$ is an interval bundle
over a compact surface $F$, to prove that the projection of $\lambda$
to $F$ has nontrivial self-intersection. 

Suppose $E$ is compressible.  To prove that $\lambda_E\in\CO(E)$, it
suffices by lemma \ref{filling and masur domain}
to show $\lambda_E$ is not a limit of a sequence
$(m_n)\subset\Delta(E)$ of meridians on $E$. On the contrary, assume
$\lambda_E$ is such a limit.
The hyperbolicity of $\CC(E)$ then implies that a $\CC(E)$-geodesic
$[m_0, \ep_n|_E]$
and a $\CC(E)$-geodesic $[m_0,m_n]$
have initial segments of length $T(n) \to \infty$ that stay within 
uniformly bounded distance of each other.
However, the quasiconvexity of $\Delta(E)$
\cite{MM3} implies that 
$[m_0,m_n]$ stays uniformly close to $\Delta(E)$ for its whole
length, which means that the distance from
$\ep_n|_E$ to $\Delta(E)$ is much smaller (roughly by $T(n)$) than
its distance to $m_0$. This immediately implies that the distance from
$\ep_n|_E$ to $\Delta(E)$ is much smaller (roughly by
$T(n)-d_E(\mu(M,E),m_0)$) than its distance to $\mu(M,E)$.  
This contradicts  property (\ref{eqn: mu nearest}) of  $R$-bounded
combinatorics, and hence we conclude that $\lambda_E\in\CO(E)$.

Now suppose $M=F\times[0,1]$ is a product with boundary components
$F_0, F_1$, both identified with $F$ via the product structure.  All
we need to show in this case is that $\lambda_{F_0}$ and
$\lambda_{F_1}$ intersect as laminations on $F$. Otherwise they
represent the same point in the Gromov boundary of $\CC(F)$ as they
are filling laminations.  Hyperbolicity of $\CC(F)$ then guarantees
that for $n$ sufficiently large every $\CC(F)$-geodesic connecting
$\ep_n|_{F_0}$ and $\ep_n|_{F_1}$ must be far from any fixed base
point of $\CC(F)$.  This contradicts condition (2) in the definition
of $R$-bounded combinatorics with respect to $M$.

The case of a twisted $I$-bundle follows similarly, by lifting to the
untwisted double cover.
Hence in all cases we have established that $\lambda$ is binding.
The last part of the statement of the lemma follows immediately.
\end{proof}

%%%%%%%%%%%%%%%%%%%%%%%%%%%%%

\subsection{Gluings}\label{subsec: gluings}
Suppose $\CM$ is a fixed collection (usually finite) of decorated
manifolds and $\Xi$ is a disjoint union of {\em copies} 
(possibly multiple) of elements of $\CM$. 
A {\em boundary identification} within $\Xi$ is
a fixed-point-free orientation reversing homeomorphism $\psi : E\to
E'$, where $E\subset \D_0 M$,
$E'\subset\D_0 M'$ for $M,M'\in \Xi$. (We allow $M=M'$ and $E=E'$.)
A {\em gluing map} is an involution $\Psi$, defined on a subset of  
\[ \D_0\Xi \equiv \bigcup_{M\in\Xi} \D_0 M\]
which is a union of boundary identifications. 
When $\Psi|_E : E\to E'$ is one of these identifications, we say 
that the corresponding components $M$ and $M'$ of $\Xi$ are 
{\em adjacent} along $\Psi|_E$ and $E$ (or $E'$) is the {\em
gluing surface}.
Given a gluing map $\Psi$, we obtain an oriented 3-manifold
$X$ as the identification space $\Xi / \Psi$. We say $X$ is an {\em
  $\CM$-gluing} if it is connected and every component $M$ of $Xi$
is a {\em piece} of $X$. 
When $X$ has infinitely many pieces, we also require that $X$ does not include an
{\em infinite row of trivial $I$-bundles}, i.e. a submanifold
homeomorphic to $F\times[0,\infty)$ composed of infinitely many pieces
  homeomorphic to $F\times[0,1]$. 

We use the notation $\D_0
X$ to denote the set of non-toroidal components of $\D X$. 
Given a piece $M$ of
$X$, we say a component $E$ of $\D_0 M$ is {\em buried} if it is in
the domain of the gluing map; otherwise $E$ is a component of $\D_0 X$,
and we say $E$ is {\em unburied}.
 
%%%%%%%%%%%%%%%%%%%%

\subsection{Gluings with bounded combinatorics}\label{subsec: gluings with bounded combinatorics}

Given an $\CM$-gluing $X$, built from a
disjoint union $\Xi$ with the gluing involution $\Psi$, we let $\mu_X$ 
be the marking on $\boundary_0 \Xi$ obtained from the disjoint union 
of $\mu(M,E)$ over all pieces $M$ in $\Xi$ and their boundary
components. Let 
\begin{equation}\label{def:nu}
\nu_X = \Psi(\mu_X).
\end{equation}
More generally given an $\CM$-gluing $X$ possibly with boundary and a
marking $\lambda$ on $\D_0 X$, we define $\nu_{(X,\lambda)}$ to be
the union of $\Psi(\mu_X)$ and $\lambda$. 

Writing $\nu \equiv \nu_{(X,\lambda)}$ for short, we let $\nu(M)$
denote the restriction of $\nu$ to $\D_0 M$, and $\nu(M,E)$ denote
the restriction to a boundary component $E\subset \D_0 M$. 
In particular we have
\begin{itemize}
\item [-] if $E$ is buried,
$$
\nu(M,E)\equiv \Psi|_{E'}(\mu(M',E'))
$$
where the restriction $\Psi|_{E'}:E'\to E $ of $\Psi$ identifies $E'\subset\D_0 M'$ with $E$, and
\item[-] if $E$ is unburied, $\nu(M,E)$ is the restriction $\lambda|_E$.
\end{itemize}

We say $(X,\lambda)$ has {\em $R$-bounded combinatorics}, or
$(X,\lambda)$ is an {\em $(\CM,R)$-gluing}, if for every piece $M\subset\Xi$ 
the marking $\nu(M)$ has $R$-bounded combinatorics with respect to $M$. 
The {\em heights} of $(X,\lambda)$ are the heights of $\nu(M,E)$ for
each component $E$ of $\D_0 M$, as defined in
\S\ref{subsec: bounded combinatorics for decorated manifolds}.

%%%%%%%%%

\subsection{Compressions.}\label{subsec: compressions}

Here we describe a special type of gluings that will be important in what comes later.
Suppose $M$ is a decorated manifold and let $\CM$ denote a collection
of decorated manifolds consisting of  $M$ and a finite number of
nontrivial compression bodies. We say $X$ is a {\em compression of $M$
  with respect to $\CM$} or simply a {\em compression of $M$} if $X =
M$ or if $X$ is obtained inductively from a compression $Y$ of $M$ 
by gluing a nontrivial compression body
$C\in\CM$ to $Y$ along the exterior
boundary of $C$. Note that
this inductive definition allows only for a bounded number of pieces
depending on the Euler characteristic of $\D_0 M$. 

%%%%%%%%%%%%%%%%%%%%%%%%

\subsection{Geodesics in the thick part of the Teichm\"uller 
space}\label{subsec: teichmuller}

Given a closed surface $F$ of genus at least 2, we use 
$\CT(F)$ to denote the Teichm\"uller space of $F$ equipped 
with the Teichm\"uller metric. We consider an element
$\sigma$ of $\CT(F)$ as a conformal structure on $F$ and 
at the same time as the complete hyperbolic surface
in that conformal class and unless explicitly stated, 
lengths and measurements are in the hyperbolic metric.

It is well known that there is a constant $L$ 
depending only on the topology of $F$ that for every 
complete marking $\nu$ on $F$, there is a hyperbolic metric 
on $F$ where the total length of $\nu$ is at most $L$. We make 
a fixed choice $\sigma_\nu\in\CT(F)$ of such a choice. (The 
distance between any two candidates is bounded from above 
independently of $\nu$.) In the opposite direction, given a 
point $\sigma\in\CT(F)$, we make a choice of a complete marking 
$\nu_\sigma$ which has the smallest length among all complete 
markings. (If $\sigma$ is in the thick part of $\CT(F)$, the 
distance in the marking graph between any two candidates is 
bounded from above independently of $\sigma$; in particular 
there are a bounded number of candidates with the bound 
depending only on topology of $F$ and the injectivity radius of
$\sigma$.) The length of $\nu_\sigma$ in $\sigma$ is bounded 
from above by a function of the topology of $F$ and the 
injectivity radius of $\sigma$. We assume the functions 
$\sigma\to\nu_\sigma$ and $\nu\to\sigma_\nu$ are equivariant 
with respect to the natural action of $\Mod(F)$, the mapping 
class group of $F$, on $\CT(F)$ and the set of markings on $F$. 

It is a consequence of work of Minsky \cite{ELC0} and Rafi
\cite{Rafi} that given $R$, there exists a constant 
$\ep_0>0$ so that if $\mu$ and $\nu$ have $R$-bounded combinatorics 
then the Teichm\"uller geodesic, $[\sigma_{\mu},\sigma_{\nu}]$, 
connecting them, is contained in the $\epsilon_0$-thick part of 
$\CT(F)$. In the opposite direction, there exists $R_0\ge R$ so that 
if the Teichm\"uller geodesic connecting $\sigma$ and $\tau$ is 
in the $\epsilon_0$-thick part of $\CT(F)$, then $\nu_{\sigma}$ 
and $\nu_{\tau}$ have $R_0$-bounded combinatorics. In addition 
there is a positive constants $c$ depending only on $R$ and the 
topology of $F$ so that
\[ \frac1{c}d_\CC(\mu,\nu)-c \le d_\CT(\sigma_{\mu},\sigma_{\nu})
	\le c d_\CC(\mu,\nu) + c \]
and
\[ \frac1{c}d_\CT(\sigma,\tau)-c \le d_\CC(\nu_{\sigma},\nu_{\tau})
	\le c d_\CT(\sigma,\tau) + c, \]
where $d_\CC$ and $d_\CT$ denote distances in $\CC(F)$ and $\CT(F)$ 
respectively. Obviously the choices of $\ep_0, R_0,$ and $c$ depend on the 
topology of the surface $F$ but to simplify the explanations we keep
this dependence implicit and do not write the dependence on the
topology of $F$. 
In fact, in most of our discussions, $R$ is fixed and we are dealing 
only with finitely many topological types of surfaces; in particular 
we assume the choices of $\ep_0, R_0,$ and $c$ are made uniformly.

We also use the discussion above and if $\sigma_F$ is 
a hyperbolic structure on the component $F\subset\D_0 M$ of the 
decorated manifold $M$, we say $\sigma_F$ has {\em $R$-bounded 
combinatorics} if $\sigma_F$ is $\epsilon_0$-thick and the
marking $\nu_{\sigma_F}$ has $R$-bounded combinatorics. The 
{\em height} of $\sigma_F$ is defined to be equal to the
height of $\nu_{\sigma_F}$. Similarly if $\sigma$ is a 
hyperbolic structure on $\D_0 M$, we say $\sigma$ has 
{\em $R$-bounded combinatorics (with respect to $M$)}, if
it is $\ep_0$-thick and 
 \[ \nu_\sigma=\bigcup_{F\subset\D_0 M}\nu_{\sigma|F} \]
has $R$-bounded combinatorics as a complete marking on
$\D_0 M$.

%%%%%%%%%%%%

\subsection{Models}\label{subsec: models}
We will now define a class of {\em model metrics} for our decorated
manifolds and gluings. Given an $(\CM,R)$-gluing $(X,\lambda)$ we will
let $\BM_X[\lambda]$ (or just $\BM_X$ when $\lambda=\emptyset$)
denote  $X$ equipped with this model metric.

Given a closed surface $F$ of genus at least 2, the {\em universal
  curve} over $\CT=\CT(F)$, the Teichm\"uller space 
of $F$, is a smooth bundle $\CS\to \CT$ whose fiber $\CS_\sigma$
over $\sigma\in\CT$ is a hyperbolic surface representing $\sigma$. 
The action of the mapping class group of $F$ lifts to an action
on $\CS$ which is isometric on fibers. Let $T_v\CS$, the vertical
tangent bundle, denote the
sub-bundle of $T\CS$ of $\CS$ whose fibers are
the kernels of the derivative of the fibration $\CS\to\CT$. It is
not hard to show there exists a smooth mapping class group invariant 
sub-bundle $T_h\CS$ of $T\CS$, which is complementary to $T_v\CS$, 
i.e. $T\CS= T_v\CS \oplus T_h\CS$. (See \cite{FM02}.) 
We fix such a choice once and for all.

Given a geodesic path $\gamma:I\to\CT(F)$, 
let $\CS_\gamma \to I$ denote the pullback bundle and
$T_v\CS$ its vertical tangent bundle. 
There is a complementary 1-dimensional smooth sub-bundle
$T_h\CS_\gamma\subset T\CS_\gamma$ defined as those vectors whose
image in $T\CS$ lies in $T_h\CS$ and projects to the 
image of $d\gamma$.
This uniquely determines
a vector field $V$ on $\CS_\gamma$ which projects to the
positive unit vector field in $I$.
Hence we can extend the Riemannian metric
on fibers of $\CS_\gamma$ to a Riemannian metric on $\CS_\gamma$ by declaring
$V$ to be of length 1 and orthogonal to $T_v\CS_\gamma$.
We should point out that given a cocompact subset of $\CT(F)$,
for example the $\ep_0$-thick part of $\CT(F)$, there exists $K$
so that for every geodesic $\gamma$ in this cocompact subset,
different choices for the sub-bundle $T_h\CS$ result in 
$K$-bilipschitz metrics on $\CS_\gamma$. In particular, up to
$K$-bilipschitz diffeomorphisms, the construction of the models 
that follows will be independent of the choice of the equivariant
sub-bundle $T_h\CS$, where $K$ will depend on $F$ and the constant
$R$.

Given distinct points $\sigma,\tau\in\CT(F)$, we use
$\BM_F[\sigma,\tau]$ to denote 
$\CS_\gamma$ equipped with the above metric, where $\gamma$ is the Teichm\"uller
geodesic starting from $\sigma$ and ending at $\tau$.
When $\mu$ and $\nu$ are complete 
markings on $F$ with $R$-bounded combinatorics and 
$\sigma_\mu\neq\sigma_\nu$, we use the notation $\BM_F[\mu,\nu]$
to denote $\BM_F[\sigma_\mu,\sigma_\nu]$. In case 
$\sigma_\mu = \sigma_\nu$, to avoid degeneracy we identify 
$\BM_F[\mu,\nu]$ with the product $\sigma_\mu\times [0,1]$ with
the product metric. Also we maintain 
the shorthand $\BM_F[\mu] = \BM_F[\mu,\mu]$. See \cite{ELC0}, 
and also \cite{ELC1} for a more general construction of such models 
in the case of surface groups. 

For every decorated manifold $M$, 
we endow $M$ with a fixed complete metric whose restriction 
to every component $E$ of $\D_0 M$ is $\sigma_{\mu_E}$. 
Such metrics can be easily constructed by a modification 
of the metric on the convex core of a geometrically finite 
hyperbolic structure on $M$, after removing standard cusp 
neighborhoods. We let $\BM_M$ denote $M$ endowed with this 
metric. In the case of $I$-bundles over a closed surface, 
we assume the choices of the model metrics are made so
that they are equivariant with respect to the action of
the mapping class group of the surface.

Now suppose $(X,\lambda)$ is an $(\CM,R)$-gluing with complete 
marking $\lambda$ on $\D_0 X$ and the gluing involution $\Psi$. 
We define $\BM_X[\lambda]$ as follows: We begin with the disjoint 
union of models  $\BM_M$ for all the pieces $M$ of $X$. Suppose 
$\Psi|_{E}:E\to E'$ is a restriction of $\Psi$ for
boundary components $E$ and $E'$ of pieces $M$ and $M'$ of
$X$ respectively. If $E\neq E'$, we attach
\[ \BM_{E}[\mu(M,E),\Psi(\mu(M',E')] \] 
to $E$ by the identity
and to $E'$ by $\Psi|_E$, on the appropriate side. If $E=E'$,
$\Psi$ restricts to an involution of $E$. The construction
of the metric on $\BM_E[\mu_E,\Psi(\mu_E)]$ is such that
$\Psi$ induces a free orientation preserving isometry of 
$\BM_E[\mu_E,\Psi(\mu_E)]$ which interchanges the boundary
components. The quotient of $\BM_E[\mu_E,\Psi(\mu_E)]$ is
therefore a twisted $I$-bundle and is attached to $E$ by
the identity map. Finally to each boundary component 
$E\subset \D_0 M$ which is unburied, we attach
\[ \BM_{E}[\mu_E,\lambda|_E] \]
on the appropriate side.
The resulting manifold is the desired model $\BM_X[\lambda]$, which we
note admits a natural isotopy class of identifications with $X$.
As an especial case, when $\D_0 X =\emptyset$ we denote this by $\BM_X$. 

Note that when $\ep$ is a complete marking on $\D_0 M$ with $R$-bounded 
combinatorics for a decorated manifold $M$, the model 
$\BM_M[\ep]$ is a special case, viewed as a gluing with 
the single piece $M$. In this case and when $\tau$ is a conformal
structure on $\D_0 M$ with $R$-bounded combinatorics, we sometimes 
use the notation $\BM_M[\tau]$ to denote the model obtained as
above, except that instead we use the interval bundle 
$\BM_E[\sigma_{\mu_E},\tau|_E]$ for each component $E$ of $\D_0 M$,
and glue that to the boundary $E$ of $\BM_M$ with the identity map.

Also note that the metric on the model $\BM_X[\lambda]$ constructed above 
is Riemannian on the gluing regions. This is a property that will be
used later in the proof of Theorem \ref{bilipschitz embedding of models}.

%%%%%%%%%%%%%%%%%%%%%%%%%%%%%

\subsection{Collapsing $I$-bundles.}\label{subsec: looking through}
It will be convenient in several parts of the proof to work in the
case where a gluing has no $I$-bundles, and we discuss the details of
that here. 

Let $X = \Xi/\Psi$ be an $\MM$-gluing.
When $X$ has at least one piece that is not an $I$-bundle, we can
identify each fibre of each $I$-bundle to a point, obtaining a
quotient map 
$$
\beta : X \to X^*
$$
where $X^*$ is still a 3-manifold, and is obtained as a gluing of
the non-$I$-bundle pieces  $\Xi^*\subset \Xi$ by a gluing map $\Psi^*$ which is
induced by $\Psi$ and the quotient map. 
If $X$ is entirely composed of $I$-bundles we can do almost the same thing,
but must leave one $I$-bundle
uncollapsed, thus exhibiting $X$ (and $X^*$) as a fibration.
Denote this the {\em fibered gluing case.}  We call this process
``collapsing $I$-bundles''. 

The reduction to the case without $I$-bundles is made possible by this lemma:
\begin{lemma}\label{looking through}
  Given a finite set $\CM$ of decorated manifolds and $R>0$ there 
exist positive $c$, $h$ and $R'$, so that if $X$ is an
$(\CM,R)$-gluing  with heights at least $h$ 
 and $X^* = \beta(X)$ is obtained by 
collapsing $I$-bundles on $X$, then
\begin{enumerate}
\item $X^*$ is an $(\CM,R')$-gluing.
\item The height  
of each gluing surface $E$ in $X^*$ is at least 
$cH-R'$, where $H$ is the sum of the heights in $X$ of
the gluing surfaces that map to $E$.
\item $\beta$ is homotopic to an $R'$-bilipschitz map
$$\beta^* : \BM_X \to \BM_{X^*}$$
which is the identity on each non-$I$-bundle piece of $X$.
\end{enumerate}
\end{lemma}

\begin{proof}
Let us first consider what happens when we collapse a single $I$-bundle $B$ 
which is not adjacent to any other $I$-bundles. 

Let $G = \Psi(\boundary B)$, the part of $\boundary_0 \Xi$
glued to $\boundary B$ -- this is composed of two components $E_0$ and
$E_1$ when $B$ is untwisted, and one component $E_0=E_1=G$ in the twisted
case. Let $M_i$ be the adjacent piece of $\Xi$ containing $E_i$. Let
$\phi:\boundary B \to \boundary B$ be the involution 
exchanging the endpoints of each interval fiber of $B$. Then the
induced involution $\Psi^*$ on $G$ is given by
$$
\Psi^*|_G = \Psi \circ \phi \circ \Psi.
$$

If $B$ is untwisted it has the form $I_F[\mu_0,\mu_1]$, and if it is
twisted its double cover $\hhat B$ has this form. 
We identify $F$ with $F\times\{0\}$ in $\boundary I_F[\mu_0,\mu_1]$.
As in \S\ref{subsec: gluings with bounded combinatorics}, 
the marking $\nu(B)$ is obtained from the decorations $\mu(M_0,E_0)$ and
$\mu(M_1,E_1)$ of the
neighbors of $B$ by applying $\Psi$. We then use $\phi$ to compare these
on one copy of $F$, obtaining
$\nu_0=\Psi(\mu(M_0,E_0))$ and
$\nu_1=\phi(\Psi(\mu(M_1,E_1)))$.
The $R$-bounded combinatorics
assumption on $B$ (as in 
\S\ref{subsec: bounded combinatorics for decorated manifolds})
specifies first that
\begin{equation}\label{eq:W bound}
d_W(\mu_i,\nu_i) \le R, \ \ i=0,1, 
\end{equation}
for any proper subsurface $W\subset F$. 
Moreover the additional condition for an $I$-bundle is that
\begin{equation}\label{eq:sigma close to line}
d_F(\mu_j,[\nu_0,\nu_1]) \le R, \ \ j=0,1
\end{equation}
where $[\nu_0,\nu_1]$ is a $\CC(F)$-geodesic connecting 
$\nu_0$ and $\nu_1$.
The heights of the gluing surfaces are given by $d_F(\mu_i,\nu_i)$,
and our hypothesis states that 
\begin{equation}\label{eq:sigma height}
d_F(\mu_i,\nu_i) \ge h.
\end{equation}
Notice also that $d_F(\mu_0,\mu_1)$ has some upper bound depending on
the finite set $\CM$.

Now the markings $\nu(M_i)$ are similarly obtained as $\Psi(\mu_i)$
(with $\Psi$ restricted to $F\times\{i\}$ for $i=0,1$). $R$-bounded
combinatorics as viewed from $M_i$ again requires (\ref{eq:W bound})
to hold, and if $E_i$ is compressible then we also have (as in
(\ref{eqn: mu nearest})) that
\begin{equation}\label{eq:Ei disks}
d_{E_i}(\Psi(\mu_i),\mu(M_i,E_i)) \le d_{E_i}(\Psi(\mu_i),\Delta(E_i)) + R.
\end{equation}

To establish the bounded-combinatorics and height conditions for the
new gluing map $\Psi^*$, let us consider $E_0$ (the case of $E_1$ is
similar). The new marking $\nu^*(M_0,E_0)$ is given by
$\Psi^*(\mu(M_1,E_1)) = \Psi(\phi(\Psi(\mu(M_1,E_1)))) = \Psi(\nu_1)$,
and we note that
$\mu(M_0,E_0) = \Psi(\nu_0)$.

The height of the new gluing is therefore 
$d_{E_0}(\Psi(\nu_0),\Psi(\nu_1))$ which equals $d_F(\nu_0,\nu_1)$,
and by the triangle inequality  and (\ref{eq:sigma close to line})
must satisfy
$$
d_F(\nu_0,\nu_1) \ge d_F(\mu_0,\nu_0) + d_F(\mu_0,\nu_1) - 2R.
$$
Using the upper bound on $d_F(\mu_0,\mu_1)$, the second term becomes
$d_F(\mu_1,\nu_1)$ up to an error depending on $\CM$. Hence we
obtain conclusion (2) on heights for this gluing surface. 

From (\ref{eq:W bound}) we obtain the subsurface projection bound
$d_W(\nu_0,\nu_1) \le 2R$. 
When $E_0$ is compressible, (\ref{eq:Ei disks}) holds for $\Psi(\mu_0)$
and we need to establish the corresponding inequality for
$\Psi(\nu_1)$. 

This follows from the following easy exercise in
$\delta$-hyperbolicity, whose proof we leave to the reader: 

\begin{lemma}\label{qconvex stability}
Suppose $Y$ is a $\delta$-hyperbolic space and
$C$ an $a$-quasiconvex subset. Let $x,y,z$ be such that $y\in C$,
$d(x,y) \le d(x,C) + r$, and $x\in[y,z]$. There are constants $r',h_0$
depending only on $\delta, a$ and $r$ such that, if $d(x,y) > h_0$, then
$d(z,y) \le d(z,C) + r'$.
\end{lemma}

In our setting, we can apply this with $C = \Delta(E_0)$, which is
quasiconvex by Masur-Minsky \cite{MM3}, with $y$ a point in
$\Delta(E_0)$ closest to $\mu(M_0,E_0)$, $z = \Psi(\nu_1)$ and $x$ the
nearest point in $[y,z]$ to $\Psi(\mu_0)$. Then (\ref{eq:sigma close to
line}), (\ref{eq:sigma height}) with $h$ chosen sufficiently large,
  and  (\ref{eq:Ei disks}) allow us to use the lemma. This establishes
  the inequality (\ref{eqn: mu nearest}) for $E_0$, and shows that we
  have $R'$-bounded combinatorics (for suitable $R'$) at this surface.

We next consider what happens when we collapse a stack of $I$-bundles
to obtain a single $I$-bundle; the point here is to show that
the height and bounded-combinatorics estimates for the combined bundle do not
depend on the number of $I$-bundles in the stack. We give the argument
in the case where all the $I$-bundles are untwisted; the case
involving a twisted bundle is easily handled by passage to a double
cover, and we leave this to the reader. 

For notational simplicity we also consider bundles of the form
$I_F[\mu]$. The case where the two decorations are not equal can be
obtained from this case at the price of a bounded change in markings
and bilipschitz constants (since we are working with a finite set $\CM$).

Thus consider a sequence $B_0,\ldots,B_n$ where $B_i = I_F[\mu_i]$,
and write $\boundary B_i$ as $\boundary_0 B_i \union \boundary_1 B_i$,
where $\boundary_j B_i$ denotes the surface identified with $F\times\{j\}$.
The gluing involution $\Psi$ takes $\boundary_1 B_i$ to $\boundary_0
B_{i+1}$, for $i=0,\ldots,n-1$.

If we collapse all but the first bundle, the identification sends each 
$\mu_i$ to 
$$
\nu_i = (\phi\circ \Psi)^i(\mu_i)
$$
where $\mu_i$ is understood to be on $\boundary_0 B_i$, and
$\phi|_{\boundary B_i}$ is the endpoint-exchanging involution on each $B_i$.
The resulting sequence $\nu_0,\ldots,\nu_n$, which can be considered
as markings in $\boundary_0 B_0$ which we identify with $F$, 
satisfies the following conditions: 

\begin{enumerate}[\qquad\em (i)]
\item $d_F(\nu_i,\nu_{i+1}) > h$
\item For all $W\subsetneq F$, $d_W(\nu_i,\nu_{i+1}) < R$.
\item $d_F(\nu_i,[\nu_{i-1},\nu_{i+1}]) < R$
\end{enumerate}
The first condition comes from the height lower bound for each
gluing. The second is the $R$-bounded combinatorics condition on
subsurface projections, 
and the third is the additional $R$-bounded
combinatorics condition for  $I$-bundles in 
\S\ref{subsec: bounded combinatorics for decorated manifolds}).

Our goal then is to show, for suitable $c,R'$ and $h$ sufficiently
large and independent of $n$, that the pair
$(\nu_0,\nu_n)$ has $R'$-bdd combinatorics, and that $d_F(\nu_0,\nu_n)
> c \sum d_F(\nu_i,\nu_{i+1}) - R'$.

Hyperbolicity implies that a local quasi-geodesic condition implies a
global one; namely: Given $K$ there exists an $K'$ and $L$ such that, if a path $\gamma$
satisfies the property that every subpath of length $L$ is
$K$-quasigeodesic, then $\gamma$ is $K'$-quasigeodesic \cite{Bow91}.
Applying this to the concatenation $\gamma = [\nu_0,\nu_1]*\cdots*[\nu_{n-1},\nu_n]$,
we see from conditions {\em (i)} and {\em (ii)} that if $h$ is sufficiently large
(independently of $n$) then $\gamma$ is a $K'$-quasigeodesic. In
particular we have $d(\nu_0,\nu_{n}) > 1/K' \sum d(\nu_i,\nu_{i+1}) -
K'$, and  for some $R_2$ we find that $\gamma$ lies in the
$R_2$-neighborhood of the geodesic $[\nu_0,\nu_n]$.

Now consider $W$. If $d_F(\boundary W,\gamma) > R_2 + 2$ then 
$d_F(\boundary W,[\nu_0,\nu_n])>2$, and 
the bounded geodesic projection lemma 
\cite{MM2}
gives a uniform upper bound on
$d_W(\nu_0,\nu_n)$. 

Now suppose $d_F(\boundary W,\gamma) \le R_2 + 2$. If $\nu_i$ is the
closest point to $\boundary W$ among $\nu_0,\ldots,\nu_n$ 
then $d_F(\boundary
W,[\nu_{i-1},\nu_i]\union[\nu_i,\nu_{i+1}]) \le R_2+2$, while
$d_F(\boundary W,\nu_{i\pm 1}) > h/2-R_2-2$. 

Now we notice that $d_W(\nu_{i-1},\nu_{i+1}) \le 2R$ by the triangle
inequality. Moreover (if $h$ is sufficiently large, depending on $K'$) then
$d(\boundary W,[\nu_{i+1},\nu_n])>2$, and similarly for
$[\nu_{0},\nu_{i-1}]$. The bounded geodesic projection lemma then
bounds $d_W(\nu_{i+1},\nu_n)$ and $d_W(\nu_{0},\nu_{i-1})$, and this
gives the desired bound on $d_W(\nu_0,\nu_n)$, establishing the bounded
combinatorics condition for the stack of $I$-bundles.

We may now obtain conclusions (1) and (2) of Lemma \ref{looking through} in two steps:
First we combine each maximal stack of $I$-bundles into a single $I$-bundle,
obtaining bounded combinatorics as above. If $X$ has any
non-$I$-bundle pieces we now collapse using the first case, to obtain
$\Xi^*/\Phi^*$ with $R'$-bounded combinatorics, and height estimates
as in (2). If $X$ is made up only of $I$-bundles then what we now have
is a single $I$-bundle identified to give a fibration, and (1) and (2)
follow easily. 

It remains to discuss the bilipschitz map $\beta^*$. For a stack of
$I$-bundles as above, the model described by the gluing is equivalent
to the union $\BM_F[\nu_0,\nu_1] \union\cdots\union
\BM_F[\nu_{n-1},\nu_n]$ along successive boundary components, whereas
the model for the single resulting $I$-bundle is
$\BM_F[\nu_0,\nu_n]$.  The bilipschitz correspondence between these
comes from the comparison of the Teichm\"uller geodesic
$[\sigma_{\nu_0},\sigma_{\nu_n}]$ with the broken path composed of
segments
$[\sigma_{\nu_i},\sigma_{\nu_{i+1}}]$. These are fellow travelers,
with uniform constants, because of the $R$-bounded combinatorics
condition and the fellow-traveller property of the corresponding paths
in $\CC(F)$.

After we have combined stacks of $I$-bundles into single $I$-bundles,
in the last step we see directly that the models match. When gluing
$M_0$ and $M_1$ along $B$, as in the beginning, we obtain the model
from $\BM_{M_0}$, $\BM_{M_1}$, and $\BM_B$, attaching
them using the $I$-bundle models $\BM_F[\mu(M_0,E_0),\Psi(\nu_0)]$ and 
$\BM_F[\Psi(\nu_1),\phi(\Psi(\mu(M_1,E_1)))]$. After collapsing $B$, we
have just $\BM_{M_0}$ and $\BM_{M_1}$, attached along 
$\BM_F[\mu(M_0,E_0),\Psi^*(\mu(M_1,E_1))]$. But this bundle is uniformly
bilipschitz equivalent to the gluing of the three previous $I$-bundles
$\BM_F[\mu(M_0,E_0),\Psi(\nu_0)]$,  
$\BM_F[\Psi(\nu_1),\phi(\Psi(\mu(M_1,E_1)))]$ and $\BM_B$, 
using the same appeal to Teichm\"uller geodesics
as in the paragraph above.  
\end{proof}
\section{Hyperbolic Structures and Convergence Theorems}\label{sec:convergence theorems}

Suppose $\Gamma$ is a discrete subgroup of $\Isom(\BH^3)=\PSL_2(\BC)$.
We let $\chi(\Gamma)$ denote the character variety of representations
of $\Gamma$ into $\PSL_2(\BC)$ as the GIT quotient of the variety
of representations of $\Gamma$. In this article we do not 
distinguish between a representation and its character and 
use them interchangeably. We refer to the induced topology on
$\chi(\Gamma)$ as the {\em algebraic topology}; in this topology
a sequence of characters $(\rho_n)$ converges to $\rho$ if 
representations $\rho_n'$ and $\rho'$ can be chosen that project
to $\rho_n$ and $\rho$ respectively and for every
$\gamma\in\Gamma$, $(\rho_n(\gamma))$ converges to 
$\rho(\gamma)$. 

An element of $\chi(\Gamma)$ is {\em discrete} if the
corresponding representations have discrete images,
and it is {\em faithful} if the representations are injective.
We say a sequence of homomorphisms $\rho_n:G\to H_n$ is 
{\em eventually faithful (or eventually injective)} if for each 
$g\in G$, $g \ne 1$, there exists $N$ such that $\rho_n(g)\ne 1$ for all 
$n>N$. We similarly define an eventually faithful sequence in 
$\chi(\Gamma)$.

In this section we gather a number of results on convergence
of sequences of characters in $\chi(M)=\chi(\pi_1(M))$, where 
$(M,P)$ is a pared manifold, and we try to understand the geometry 
of the limit when it exists. In particular the next theorem is
used whenever we need to show a sequence of representations
has a convergent subsequence.

\begin{theorem}\label{eventually faithful convergence}
Let $(M,P)$ be a pared manifold with either incompressible free
sides or pared locus consisting only of tori. Let $(\rho_n)$ be an eventually 
faithful sequence of discrete elements of $\chi(M)$ satisfying
the following: 
\begin{enumerate}
	\item For every element $\gamma\in\pi_1(M)$ representing 
the core of an annular component of $P$, $\ell(\rho_n(\gamma))$ is
bounded for all $n$.
	\item The sequence $(\ep_n)$ of markings on $\D_{0}(M,P)$
converges to a binding lamination $\lambda$ on $\D_0(M,P)$ and the 
translation lengths $\ell(\rho_n(\ep_n))$ are bounded.
\end{enumerate}  
Then a subsequence of $(\rho_n)$ converges to a
discrete and faithful element of $\chi(M)$.
\end{theorem}

{\bf Remarks:} The restriction to a special class of pared manifolds,
as in \S\ref{subsec: pared manifolds}, simplifies our discussion and
suits the needs of this paper. However we remark that it is not difficult
to extend the definitions and proof to the general setting.

\begin{proof}
The proof is a fairly standard corollary of what we refer to as 
Morgan-Shalen Theory and its generalization known as the Rips Machine. 
We we outline the proof briefly; for more on Morgan-Shalen Theory
and ideas used here we refer the reader to \cite{MS1,MS2,MS3,Ota96,
Kap01}. 

By Morgan-Shalen Theory \cite{MS1}, if $(\rho_n)$ has no convergent
subsequence in $\chi(M)$ then we can pass to a subsequence and choose constants 
$\delta_n\to 0$ so that, after rescaling $\Hyp^3$ by 
$\delta_n$, the sequence of actions converges to a nontrivial action
of $\pi_1(M)$ on an $\BR$-tree $T$. Here, convergence means that for every
$\alpha\in\pi_1(M)$, the translation length $\ell_T(\alpha)$ of the
action of $\alpha$ on $T$ is equal to $\lim_n
\delta_n\ell(\rho_n(\alpha))$. Nontriviality guarantees that the action has
no global fixed point.
It was shown by Morgan-Shalen that when $(\rho_n)$ are faithful
and discrete, the resulting action on $T$ is {\em small}, meaning that the
stabilizers of non-degenerate arcs are cyclic. By an observation in
\cite{NS09}, the same holds when the sequence $(\rho_n)$ is
discrete and {\em eventually} faithful. 

For $n$ sufficiently large, it follows easily that the $\rho_n$
image of the fundamental group of every toroidal boundary
of $M$ is a parabolic subgroup and as a result in 
the limit, the action of this subgroup on $T$ is elliptic. In
addition, it follows easily from assumption (1) that for every 
element $\alpha$ representing the core of an annular component 
of $P$, we have $\ell_T(\alpha)=0$; we say this
is a {\em relatively elliptic} action of $(M,P)$ on $T$.

A closed curve $\alpha$ in $(M,P)$ is {\em realizable} in $T$ if
elements of the conjugacy class in $\pi_1(M)$ represented by $\alpha$
have positive translation lengths in $T$. This notion of realizability
can be extended to laminations on $\D_0 M$; we refer to \cite{Ota96}
for the precise definition. We will use known results to prove at
least a component of $\lambda$ is realized in $T$. Once this is
proved, it is a consequence of properties of realizability (see
\cite{Ota96}) that $\delta_n \ell(\rho_n(\epsilon_n)) >c>0$ for a
constant $c>0$ and $n$ sufficiently large. This contradicts assumption
(2) and proves the sequence $(\rho_n)$ must have a convergent
subsequence. The limit is faithful because of the eventual
faithfulness of the sequence and discrete by Chuckrow's Theorem
\cite{Chu68} on Kleinian groups.  (See \cite[Prop. 3.1]{ThuI}.)

To prove the Theorem, it remains to show that in a small nontrivial
relatively elliptic action of $(M,P)$ on $T$, at least one component
of a binding lamination, such as $\lambda$, has to be realized.

In the case $(M,P)$ is an interval bundle over a compact surface $F$,
we obtain a small minimal nontrivial action of $\pi_1(F)$ on $T$ with
the property that $\ell_T(\alpha)=0$ if $\alpha$ is represented by a
component of $\D F$. It follows from Skora's theorem \cite{Sko96} that
such an action is dual to a measured lamination $\mu_T$ on $F$. As a
consequence, every measured lamination which intersects $\mu_T$ is
realized. The binding condition implies that the projection of
$\lambda$ to $F$ has nonzero intersection with $\mu_T$ and therefore
at least one component of $\lambda$ is realizable.

Now assume $(M,P)$ is not an interval bundle.  We suppose by
contradiction that we have a small relatively elliptic action of
$\pi_1(M)$ on $T$ for which no component of $\lambda$ is realizable,
and prove that the action is trivial.

We first claim that there is a union $\CF$ of components of
$\D_0(M,P)$ that meets every essential disk and annulus in $(M,P)$,
and such that the action restricted to $\pi_1(F)$ for each component
$F$ of $\CF$ is trivial.

If $F$ is a compressible component of $\D_0(M,P)$ then it must be
included in $\CF$, and by definition of a binding 
lamination $\lambda|_F$ is
a filling Masur domain lamination for the relative compression
body $C_F$. By results of Kleineidam-Souto \cite{KS02}, a filling
Masur domain lamination is always realizable for a nontrivial small
action of the fundamental group of a compression body on an $\BR$-tree;
so with the assumption that $\lambda|_F$ is not realizable, we 
conclude $\pi_1(C_F)$ acts on $T$ with a global fixed point, i.e. trivially.

Now consider any essential annulus $A$. If $A$ meets no compressible
boundary component then it must meet a boundary component $F$ of the
incompressible core such that (by the definition of binding)
$\lambda|_F$ is filling. Include this $F$ in $\CF$.
Note that $A$ cannot meet any of the compression bodies
and hence is contained in a single component $(M',P')$ of the
incompressible core. In this case, we may assume $(M',P')$ is not an
$I$-bundle, since then it would be all of $(M,P)$, and we are in the
previous case.
Hence, $F$ cannot be completely inside the window part of the JSJ 
decomposition of $(M',P')$ and therefore contains a component $\gamma$ 
of the window frame of the JSJ decomposition of $(M',P')$. 
It is a consequence of the results of Morgan-Shalen \cite{MS3} that in a 
small nontrivial relatively elliptic action of $(M',P')$ on an 
$\BR$-tree, the conjugacy classes associated to window frames act 
as elliptic isometries, hence $\ell_T(\gamma)=0$. But $\lambda|_F$ 
is not realizable in $T$, and another application of Skora's Theorem 
shows that the action of $\pi_1(F)$ on $T$ has to be trivial as 
claimed. 

This gives the desired set $\CF$, and
the proof is now completed by the following lemma:

\begin{lemma}\label{triviality of an action}
	Assume $(M,P)$ is a pared manifold and $\pi_1(M)\actson T$ is a small
relatively elliptic action on an $\BR$-tree $T$. Also assume $\CF$ is
the union of a collection of free sides of $(M,P)$ with the property that
\begin{enumerate}
	\item every essential disk or annulus in $(M,P)$ has a boundary
	component on $\CF$, and
	\item for every component $F$ of $\CF$, the induced action of $\pi_1(F)$
on $T$ is trivial.
\end{enumerate}
Then the action $\pi_1(M)\actson T$ is trivial. 
\end{lemma}
\begin{proof}
Suppose first that $(M,P)$ has incompressible free sides.

Let $G=(V,E)$ be a graph whose vertex set $V$ corresponds bijectively
to components of the JSJ decomposition of $(M,P)$ that are not solid
tori, as well as components of $\CF$.  We label each vertex by the
corresponding component, and we connect two vertices $X$ and $Y$ by an
edge in $E$ for each component $Z$ of their intersection in $M$ that has
non-elementary fundamental group, labeling the edge by $Z$.  We claim:
\begin{enumerate}[\qquad (i)]
\item for each $X\in V$, $\pi_1(X)$ acts with a (unique) fixed point
  in $T$,
\item whenever $X$ and $Y$ are connected by an edge associated to a
  component $Z$ of $X\intersect Y$, the amalgamation $\pi_1(X)*_{\pi_1(Z)}\pi_1(Y)$
  has a unique fixed point, and
\item the graph $G$ is connected.
\end{enumerate}
Note that each $\pi_1(X)$ is determined as a subgroup of $\pi_1(M)$
only up to conjugation, so that (i) is a statement about each of these
conjugates separately, and to make sense of  (ii) 
we consider $\pi_1(X)$, $\pi_1(Y)$ and $\pi_1(Z)$ with respect to 
a base point in $Z$.

Statements (i-iii) together imply the conclusion of the lemma, as
follows. The graph $G$, being connected, gives a graph-of-groups
decomposition of a group $\Gamma$ which surjects to $\pi_1(M)$. Let
$\widetilde G$ be the associated tree. The map that takes a vertex of
$\widetilde G$ to the unique fixed point in the $\R$-tree of the
associated stabilizer in $\Gamma$ (defined by virtue of (i)) is
equivariant; then (ii) and (iii) imply that this map is
constant.
Hence $\Gamma$ fixes
a point in the $\R$-tree, and so does $\pi_1(M)$.

To see (i), note first that a fixed point exists by hypothesis for
each component 
of $\CF$.  The Morgan-Shalen theory implies, as we remarked before,
that every conjugacy class associated to a window frame acts
trivially, and hence the induced action for each 
component $(U,Q)$ of the JSJ decomposition is relatively elliptic. The
Morgan-Shalen theory then further provides a fixed point for each
$(U,Q)$ which is not an $I$-bundle (i.e. is relatively
acylindrical). For an untwisted $I$-bundle component, we note by hypothesis (1)
that one of its free sides must be in $\CF$, and hence its fundamental
group fixes a point. For a twisted $I$-bundle the same holds for an
index 2 subgroup, and hence for the whole group.
Uniqueness of each of these fixed points follows from smallness of the
action, since the fundamental group of each $X\in V$ is
non-elementary.

To see (ii), if $Z$ is a component of $X\intersect Y$ with $\pi_1(Z)$
nonelementary, then $\pi_1(X)$ and $\pi_1(Y)$ must act with a common
fixed point -- otherwise $\pi_1(Z)$ would fix an arc, contradicting smallness.

To see (iii), consider first a solid-torus component $S$ of the JSJ
decomposition. The boundary of $S$ contains a  union of
essential parallel annuli $A = \boundary S \intersect \boundary M$. At most one
of these lies in the locus $P$, by the definition of a pared
manifold. If one component $A_0$ of $A$ lies in $P$, there is an
essential annulus from $P$ to each of the other components, which by
hypothesis (1) means that each component of $A\setminus P$ lies in
$\CF$. If no component of $A$ lies in $P$ then, by the same argument
we see that at most one of the components of $A$ (which we still call
$A_0$) can fail to lie in $\CF$. 

Any component of the JSJ decomposition that meets $S$ does so along an
annulus component of $B = \boundary S \setminus {\rm int}(A)$. If two
such components meet in successive components of $B$ separated by a
component $A'$ of $A$ other than $A_0$, then they are connected in $G$
through the component of $\CF$ containing $A'$. Hence all the
components of the JSJ decomposition meeting $S$ can be connected in
$G$.

If two non-solid-torus JSJ components meet along a single annulus,
that annulus must have at least one boundary in $\CF$, again by
hypothesis (1), so we can again connect these components in $G$.

Since $M$ is connected, it follows immediately that $G$ is connected.

\medskip

It remains to consider the case when $(M,P)$ 
has compressible free sides. We may decompose $(M,P)$ into a (possibly
disconnected) incompressible core and a union of relative
compression bodies. The outer boundary of each compression body is in $\CF$ by hypothesis
(1), and hence its fundamental group acts with a fixed point by
hypothesis (2). 
The previous argument applies taking the graph $G$ to
be the union of graphs of the components of the incompressible core,
together with one new vertex for each relative compression body,
joined by an edge to a previous vertex for each subsurface 
of its inner boundary along which it is attached to the corresponding
component in $M$. 
\end{proof}

\end{proof}

%%%%%%%%%%%%%%%%%%%%%%%%%%%%%

\subsection{Hyperbolic structures}\label{subsec: hyperbolic structures}

Given an orientable compact atoroidal 3-manifold (such as
a decorated manifold), a {\em hyperbolic structure} on $M$ is a 
complete hyperbolic metric on the interior of $M$ or more precisely, a
3-manifold $N$ equipped with a complete Riemannian metric with all
sectional curvatures equal to $-1$, and an embedding 
$M\hookrightarrow N$ with the property that the complement of the
image of the embedding is a trivial interval bundle over $\D M$.
The image of this embedding is called a {\em standard compact 
core for $N$.} Two hyperbolic structures on $M$ are equivalent if 
there is an isometry between them which induces a homeomorphism
{\em isotopic} to the identity map on $M$. Note that a hyperbolic 
structure on $M$ includes a preferred isotopy class of an embedding 
of $M$ as a compact core.

In particular, writing $N=\BH^3/\Gamma$ where $\Gamma$ is a
discrete subgroup of $\Isom(\BH^3)=\PSL_2(\BC)$, the embedding of
$M$ in $N$ induces a representation $\pi_1(M)\to\Gamma\le\PSL_2(\BC)$,
and therefore a discrete and faithful element of 
$\chi(M)=\chi(\pi_1(M))$. A {\em cusp} in a hyperbolic 3-manifold
is a subset isometric to the quotient of a horoball in $\BH^3$ by 
a rank 1 or rank 2 discrete parabolic subgroup preserving the 
horoball. We call this a {\em rank 1} or {\em rank 2 cusp} 
accordingly. It follows that a toroidal boundary of a compact core
is isotopic to the boundary of a rank 2 cusp.

It is worth emphasizing that not every discrete faithful element of
$\chi(M)$ corresponds to a hyperbolic structure on $M$. Thanks to the
Tameness Theorem \cite{Ag, CG}, we know every discrete and faithful
element of $\chi(M)$ corresponds to a hyperbolic structure on a
compact atoroidal manifold $M'$ and obviously $M$ and $M'$ are
homotopy equivalent, but it is possible for $M$ and $M'$ not to be
homeomorphic. In what follows we appeal to results of \cite{NS12} to
guarantee that a limit representation is in fact a hyperbolic
structure on $M$.

\subsection{Convex cocompact and simply degenerate ends}
\label{subsec: geometry of ends}
Suppose $N$ is a hyperbolic structure on $M$. The end of $N$ 
associated to a component $E\subset\D_0 M$ is {\em convex 
cocompact} if it has a neighborhood which is disjoint from the 
convex core of $N$.  By Ahlfors-Bers theory, the
convex cocompact ends associated to $E$ are parametrized by $\CT(E)$,
the Teichm\"uller space of $E$, via the {\em conformal structure 
at infinity} given by a component of the domain of discontinuity
associated to $E$.

The end $\CE$ of $N$ associated to $E$ is
{\em simply degenerate} if it is not convex cocompact and it does not
contain a cusp. It is a consequence of the work of Thurston
\cite{Thu79} and Canary \cite{Can93b} that associated to every such
end there is an {\em ending lamination} $\lambda_E$ which is an element of
$\EL(E)$ and belongs to the Masur domain. 
This in particular implies that if $(\gamma_n)$ is a sequence of 
essential simple loops on $E$, which
converges to $\lambda_E$ in the Gromov boundary of $\CC(E)$, then the
geodesic representatives of $(\gamma_n)$ in $N$ {\em exit} $\CE$, i.e. 
given every compact subset $K$ of $\CE$, the geodesic representative
of $\gamma_n$ is in $\CE\setminus K$ for $n$ sufficiently large. 

In the opposite direction, assume $\rho$ is a discrete and faithful
representation of $\pi_1(M)$, $N_\rho=\BH^3/\rho(\pi_1(M))$ is a 
hyperbolic structure on the compact manifold $M'$, and $(\gamma_n)$ 
is a sequence of essential simple loops on $E$ that converges to 
the filling Masur domain lamination $\lambda_E$. 
It follows from simple properties of the Masur domain that 
$(\gamma_n)$ represent infinitely many free homotopy classes in $M$.
(See \cite{Ota88}.) If in addition the lengths $\ell(\rho(\gamma_n))$ 
are bounded uniformly, then 
	\begin{enumerate}
		\item there is a homotopy equivalence $\phi:M\to M'$ in the 
			homotopy class of $\rho$, which restricts to a homeomorphism 
			from $E$ to a component $E'$ of $\D_0 M'$, 
		\item the end associated to $E'$ is simply degenerate with 
			ending lamination $\phi(\lambda_E)$, and 
		\item the geodesic representatives of $\rho(\gamma_n)$ exit this 
			end. 
	\end{enumerate}
This follows from basic properties of ending laminations when $E$ is 
incompressible and from \cite[Thm. 1.4]{NS12} in the general case.

%%%%%%%%%%%%%%%%%%%%%%%%%%%%%

\subsection{Homeomorphism type of the algebraic limit}
\label{subsec: homeo type of limit}

It is not a priori clear that in the situation of theorem
\ref{eventually faithful convergence} the algebraic limit gives a
hyperbolic structure on $M$ or that its ends have the expected
geometry.  When $\ep_n$ have $R$-bounded combinatorics with respect to
a decorated manifold $M$, with heights tending to infinity, and
$\ep_n$ nonempty for every component of $\D_0 M$, we can say more via
an argument similar to that of \cite[Thm 8.1]{NS12}.

\begin{theorem}\label{eventually faithful convergence and 
bounded combinatorics}
	Let $M$ be a decorated manifold, and assume
  \begin{enumerate}
    \item $(\rho_n)$ is a sequence of discrete, eventually 
		faithful elements of $\chi(M)$,
    \item $(\ep_n)$ is a sequence of markings with $R$-bounded 
		combinatorics for a fixed $R$, and for every component $E$ 
		of $\D_0 M$, $\ep_n|_E$ is nonempty with heights tending to 
		infinity, and
    \item the translation lengths $\ell(\rho_n(\ep_n))$ are bounded. 
  \end{enumerate}
  Then a subsequence of $(\rho_n)$ converges to some $\rho\in\chi(M)$ 
	which is induced by a hyperbolic structure on $M$, and the end of 
	$N_\rho=\BH^3/\rho(\pi_1(M))$ associated to each component of 
	$\D_0 M$ is simply degenerate. 
\end{theorem}

\begin{proof}
By lemma \ref{limits of bounded combinatorics bind}, we can pass to a
subsequence and assume that $(\ep_n)$ converges to a binding
lamination $\lambda$. In addition the restriction of $\lambda$ to
each component $E$ of $\D_0 M$ is a nonempty filling lamination. 

By Theorem \ref{eventually faithful convergence}, after conjugation
and passing to a subsequence, we can assume $\rho_n$ converges to
$\rho$.  By the Tameness Theorem \cite{Ag, CG}, $N_\rho$ is a
hyperbolic structure on a compact 3-manifold $M'$. The bounds on
$\ell(\rho_n(\ep_n))$ imply, as explained in \S \ref{subsec: geometry
  of ends}, there is a homotopy equivalence $\phi:M\to M'$ which
restricts to a homeomorphism on $\D_0 M$, and for every component $E$
of $\D_0 M$, the end of $N_\rho$ associated to $\phi(E)$ is simply
degenerate with ending lamination $\phi(\lambda|_E)$. Also in the
special case that $M$ is a trivial $I$-bundle, the binding condition
implies the $\phi$-image of the components of $\D M$ are distinct
components of $\D M'$. Since $\phi$ is a homotopy equivalence we can
also assume that it is a homeomorphism restricted to every toroidal
component of $\D M$. It is now a consequence of a generalized version
of a theorem of Waldhausen \cite{Wa} (cf. \cite{Jaco, Tu}) that $\phi$
is homotopic to a homeomorphism with a homotopy that does not move any
point on $\D M$.
\end{proof}

%%%%%%%%%%%%%%%%%%%%%%%%%%%%%

\subsection{Geometric and algebraic limits}\label{subsec: geometric limit}
We say a sequence of pointed metric spaces $(X_n,x_n)$ converges
{\em geometrically} to the pointed metric space $(X,x)$ if 
the sequence $(X_n,x_n)$ converges to $(X,x)$ in the 
Gromov-Hausdorff topology. It is due to Gromov \cite{Gromov-book} that if for every $n$,
$(X_n,x_n)$ is a pointed Riemannian $d$-manifold with an upper bound
on the sectional curvatures and a lower bound for the injectivity radius
at $x_n$ independently of $n$, then $X$ is a smooth manifold
endowed with a $\CC^{1,1}$-Riemannian metric and $(X_n,x_n)$ converges
to $(X,x)$ in the $\CC^{1,\alpha}$-topology for all $\alpha<1$. 

Recall that a sequence $(X_n,x_n)$ of pointed Riemannian manifolds
converges in the $\CC^{1,\alpha}$-topology to a pointed Riemannian 
manifold $(X,x)$ if for all $D>0$, there is a sequence of {\em approximating
maps}
\[ \kappa_n:(B_{X}(x,D),x) \to (X_n,x_n) \]
on the ball $B_X(x,D)$ of radius $D$ centered at $x$ in $X$, so that the pulled-back metrics converge in the $\CC^{1,\alpha}$-topology
on tensors on $B_X(x,D)$ to the restriction of the metric of $X$.
The $\CC^{1,1}$-Riemannian metric on a manifold is one whose first
derivatives are Lipschitz continuous. See \cite{Petersen} for more details.

An algebraically convergent sequence of discrete, faithful representations 
$(\rho_n:G\to\PSL_2(\BC))$ 
converges {\em geometrically} to a hyperbolic manifold $N_G$ if there
are choices of base points $x_n\in N_n=\BH^3/\rho_n(G)$ 
and $x_G\in N_G$ 
so that for each $g \in G$ the translation distance
of $\rho_n(g)$ at $x_n$ is uniformly bounded,
and the sequence of pointed manifolds $(N_n,x_n)$ 
converges geometrically to $(N_G,x_G)$. In this
case, we can assume the convergence is smooth on compact sets 
\cite{BP}. It is 
standard that if $(\rho_n)$ converges algebraically to $\rho$, one 
can pass to a further subsequence and assume $(\rho_n)$ also converges 
geometrically to a hyperbolic manifold $N_G$; then the image of 
$\rho$ is naturally conjugate to a subgroup of $\pi_1(N_G)$, or 
equivalently $N_\rho=\BH^3/\rho(G)$ covers $N_G$. We say 
$(\rho_n)$ converges {\em strongly} to $\rho$ if the covering 
map is a homeomorphism. 

In the setting of the above theorem, we can use Thurston 's 
Covering Theorem \cite{Thu79, Can96} and show:

\begin{corollary}\label{algebraic limit is a finite cover}
  With the same hypothesis as in Theorem \ref{eventually faithful
  convergence and bounded combinatorics} and after passing to 
  a further subsequence we may conclude that the
  sequence of hyperbolic manifolds $\BH^3/\rho_n(\pi_1(M))$ converges
  in the Gromov-Hausdorff topology to a hyperbolic manifold $N_G$ and
  the covering map $N_\rho\to N_G$ is finite-to-one. 
\end{corollary}

\begin{proof}
By the above theorem, $N_\rho$ is a hyperbolic structure on $M$ with
simply degenerate ends and in particular has no rank 1 cusps. It follows 
from the Covering Theorem of Thurston that 
either the covering map $N_\rho\to N_G$ is finite-to-one or the covering
map $N_\rho\to N_G$ factors through the infinite cyclic covering of
a fibered finite cover of $N_G$. In that case, $N_\rho$ and 
therefore $M$ is an $I$-bundle over a closed surface and $N_G$ is closed.
Then $N_n=\BH^3/\rho_n(\pi_1(M))$ is isometric to
$N_G$ for $n$ sufficiently large, and the homomorphism
$\rho_n:\pi_1(M)\to \pi_1(N_n)$ is induced by a map that factors
through the fiber of a finite cover of $N_n=N_G$. But this is not 
possible since $\rho_n$ surjects onto $\pi_1(N_n)$. 
\end{proof}

%%%%%%%%%%%%%%%%%%%%%%%%%%

\section{Uniform Immersions} \label{sec:uniform immersions}

In this section we prove immersion theorems for decorated
manifolds within the hyperbolic manifolds associated to
representations which satisfy a number of conditions, and are in particular sufficiently close to injective. Applications
of these results will be of two types: for convex-cocompact
representations, and for representations induced from maps of decorated manifolds into hyperbolic 3-manifolds with control on the injectivity and the image of the boundary.

First in Theorem \ref{thick lifted embedding}, we show how to obtain
embeddings from the model core $\BM_M$ of a decorated manifold $M$. We
show this in the presence of a sufficiently injective homomorphism
$\pi_1(M)\to\pi_1(N)$ for  a hyperbolic manifold $N$ and length bound
in $N$ for a marking on $\D_0 M$ with bounded combinatorics and large
heights. In general, we cannot find an embedding of $\BM_M$ but what
we call a {\em lifted embedding}, which factors as a covering map
followed by an embedding.  In the case when $M$ is an $I$-bundle,
we also need to exclude a possibility that the homomorphism is induced from
a virtual fibration which is not a fibration. We consider the case of
trivial $I$-bundles in Theorems \ref{thick lifted embeddings of
  I-bundles}, and \ref{interpolations}, where we also show how
lifted embeddings for different 
decorations relate to each other and can be connected via an immersed
interpolation.

The main result of this section is Theorem \ref{bilipschitz embedding
  of models}, which gives conditions under which for a hyperbolic
structure $\tau$ on $\D_0 M$, the model $\BM_M[\tau]$
admits a locally bilipschitz immersion $\BM_M[\tau]\to N$ in a given
homotopy class. The conditions involves existence of a suitable map from
$\tau$ into $N$, almost-injectivity of
$\pi_1(M)$, as well as almost-injectivity of the boundary groups of
$M$ with respect to markings that arise naturally in the model.

%%%%%%%%%%%%%

\subsection{Approximately injective}\label{subsec: approximately injective}
Fixing generators for a group $G$ and an associated word-length
$|\cdot|$, we say a representation
$\rho:G\to H$ is $L$-injective if every nontrivial $g\in G$ with
$|g|\le L$ has  nontrivial image.
For each decorated manifold $M$ we fix a generating set for
$\pi_1(M)$, and therefore define $L$-injective
representations of $\pi_1(M)$. 

In the case of decorated manifolds
$I_F[\mu]$ for a closed surface $F$ and complete marking $\mu$ on $F$,
we assume the choice of this generating set is invariant under the
natural action of the mapping class group of $F$. Also when $\sigma\in\CT(F)$
is a hyperbolic structure on $F$, we assume the choice of the generating
set of $\pi_1(F)$ is the same as the choice of the generating set 
for $I_F[\nu_\sigma]$, where $\nu_\sigma$ is the complete marking of 
smallest length on $\sigma$ chosen in \S \ref{subsec: teichmuller};
therefore makes sense to speak of a representation of $\pi_1(F)$ which is
$L$-injective with respect to $\sigma$.

%%%%%%%%

\subsection{Lifted embeddings}\label{subsec: lifted embeddings}
We say a map $f: M\to N$ is a {\em lifted embedding} if $f$ factors as
$f\equiv g\circ \Pi$, with $g$ an embedding, $\Pi$ a covering map.
In the presence of metrics on $M$ and $N$, we say $f$ is a {\em lifted
  $K$-bilipschitz embedding} if there is a metric on the intermediate
space such that $\Pi$ is a local isometry and $g$ is a $K$-bilipschitz embedding.
We always assume $g$ is {\em collared}, i.e. it extends to an
embedding of an additional collar of width 1.

Our first theorem gives uniform conditions under which a homomorphism
$\rho:\pi_1(M)\to\pi_1(N)$, where $M$ is a decorated manifold and $N$
a hyperbolic 3-manifold, is induced by a lifted embedding $\BM_M\to N$.

\begin{theorem}\label{thick lifted embedding}
  Let $M$ be a decorated manifold which is not an $I$-bundle, and $R>0$. There exist $K, L$
        and $D$ such that if 
        \begin{enumerate}
           \item $\ep$ is a marking on $\D_0 M$ with
        $R$-bounded combinatorics, heights at least $D$ and $\ep|_E$ nonempty for every component $E$ of $\D_0 M$,
          \item $\rho:\pi_1(M)\to \pi_1(N)$ is an $L$-injective
        homomorphism for a hyperbolic 3-manifold
        $N$ and 
          \item $\ell_\rho(\ep) < R$, 
  \end{enumerate}
  then $\BM_M$ admits a
        lifted $K$-bilipschitz embedding $f:\BM_M\to N$
        in the homotopy class determined by $\rho$.
			Moreover, restricted to every toroidal
                        component of $\BM_M$, the map $f$ is a
                        covering of the boundary of a component of the thin part of $N$.
\end{theorem}
\begin{proof}
We argue by contradiction. If the theorem is false, there is a
sequence $(\rho_n:\pi_1(M)\to\pi_1(N_n))$ of homomorphisms
satisfying the hypothesis of the theorem with fixed $R$,
markings $\ep(n)$ with $L(n)$, $D(n),$ and $K(n)$ all tending to 
infinity such that there is no lifted $K(n)$-bilipschitz embedding 
from $\BM_M$ to $N_n$ in the homotopy class determined by $\rho_n$.  

The sequence $(\rho_n)$ is  eventually injective
(as representations to $\PSL_2(\BC)$), 
and $\ep(n)$ satisfy bounded combinatorics and
heights going to $\infty$, so we may apply
Theorem \ref{eventually faithful convergence and bounded
  combinatorics} to conclude that,
after possibly conjugating and passing to a subsequence,
$(\rho_n)$ converges to a representation
$\rho:\pi_1(M)\to\PSL_2(\BC)$ induced by
a hyperbolic structure on $M$ with the property that the end
associated to every component of $\D_0 M$ is simply
degenerate. Moreover by  Thurston's Covering Theorem
and the assumption
that $M$ is not an $I$-bundle, we can also assume the sequence of
hyperbolic manifolds $(N_n)$ converges in the Gromov-Hausdorff
topology to a hyperbolic 3-manifold $N_G$ which is finitely covered by
$N_\rho$ (compare Corollary \ref{algebraic limit is a finite cover},
where the hypothesis that $M$ is not an $I$-bundle is unnecessary
because $\rho_n$ are taken to be surjective).
Then a standard compact core $C$ of $N_G$ lifts to a
standard compact core $\cover C$ of $N_\rho$. Since $N_\rho$ is a
hyperbolic structure on $M$, there is an embedding
$\BM_M\hookrightarrow N_\rho$, in the homotopy class determined by
$\rho$, whose image is $\cover C$; we can assume this embedding is
bilipschitz and the image of every toroidal component of $\BM_M$ 
is the boundary of a rank 2 cusp of $N$. The post-composition of this 
embedding with the
projection $N_\rho\to N_G$ and the approximating map $N_G\to N_n$
provides lifted embeddings which factor through a local homeomorphism
from $\BM_M$ to $C$. The lifted embeddings satisfy a uniform
bilipschitz bound and every toroidal component of $\D\BM_M$
covers boundary components of thin parts of the approximates. 
This contradicts the choice of the sequence and completes the proof.
\end{proof}

%%%%%%%%

\subsection{Virtual fibration}\label{subsec: virtual fibration}

In the case of $I$-bundles, the results of theorem
\ref{thick lifted embedding} can be extended, as the next theorem shows. 
However we need to exclude homomorphisms $\pi_1(F)\to\pi_1(N)$ which
come from a virtual fibration which is {\em not} a fibration.  

Recall that a closed orientable $N^3$ is {\em fibered with fiber $F$}
if $F$ has a two-sided embedding in $N$
so that $N\backslash\backslash  F$ (the complement of a regular neighborhood of $F$) is a
union of (one or two) $I$-bundles. Note that if there is one $I$-bundle
it is $F\times[0,1]$ and $N$ fibers over $S^1$, and if there are two
they are twisted and $N$ fibers over $[0,1]$ considered as a 1-orbifold.
We say that a map $f:F\to N$ is a {\em virtual fiber} if
it can be factored as $f=p \circ \hhat f$
where $\hhat f:F\to \hhat N$ embeds $F$ as a fiber of $\hhat N$ and
$p:\hhat N \to N$ is a finite cover. We say it {\em covers a fiber} if
it can be factored as $f=g \circ q$ where $q:F\to F'$ is a covering
map and $g$ embeds $F'$ as a fiber of $N$. Note that the latter of
these is a special case of the former.

The following lemma observes that a virtual fiber which covers an
embedding in fact covers a fiber. It will be used in the proof of Theorem
\ref{bilipschitz models for bounded type manifolds}.

\begin{lemma}\label{embedding implies fibration}
   Suppose $M$ is a compact 3-manifold and $F$ is a closed surface. If
   a map $\pi_1(F)\to\pi_1(M)$ is induced by a virtual fiber
   and also by a composition
   $F\to F'\hookrightarrow M$ where $F\to F'$ is a finite-to-one
   covering map and $F'\to M$ is an embedding, then $M$ is fibered
   with fiber $F'$. 
\end{lemma}

\begin{proof}
After choosing maps appropriately, 
we can assume this diagram of homomorphisms is induced
by an inclusion $F\subset \hhat M$ that embeds $F$ as a fiber, and a finite
covering $p:\hhat M \to M$ that restricts to a covering from $F$ to
$F'\subset M$. 

We then have that  $\hhat M\backslash\backslash F$ is a union of (one or two)
$I$-bundles. The preimage $p^{-1}(F')$ in $\hhat M$ is a disjoint union of
incompressible surfaces one of which is $F$, and hence $p^{-1}(F')$
cuts $\hhat M$ into a union of $I$-bundles. 
Since a 3-manifold covered by an $I$-bundle is itself an $I$-bundle,
this means that every component of $M \backslash\backslash F'$ is an
$I$-bundle. It follows that $F'$ is a fiber of $M$. 
\end{proof}

%%%%%%%%%%%%%%%%%%%%

\subsection{Lifted embeddings of $I$-bundles}\label{subsec: lifted embeddings of I-bundles}

Now we state the version of theorem \ref{thick lifted embedding} for $I$-bundles. In this case, we need only one of the heights of the ending data $\ep$ to be large.

\begin{theorem}\label{thick lifted embeddings of I-bundles}
Given $M$ a decorated trivial $I$-bundle and $R>0$, there exist $K,L,D$ so that
the following holds. Suppose $\rho:\pi_1(M)\to\pi_1(N)$ is an
$L$-injective homomorphism for a hyperbolic 3-manifold $N$, $\ep$ is a
marking on $\D M$ with $R$-bounded combinatorics, one height
at least $D$, and the $\rho$-length of $\ep$ is bounded by $R$. Then
either $\rho$ is induced by a virtual fiber which does not cover
a fiber, or there is a lifted $K$-bilipschitz embedding $\BM_M\to N$
in the homotopy class of $\rho$. 
\end{theorem}

\begin{proof}
   Assume $M=I_F[\mu_0,\mu_1]$ is a decorated trivial $I$-bundle. We proceed as in the proof of theorem \ref{thick lifted embedding}. Suppose $L(n), D(n)$ are constants tending to infinity as $n\to\infty$, $\rho_n:\pi_1(M)\to \pi_1(N_n)$ is a sequence of representation and $\ep(n)$ is a sequence of markings on $\D M$ satisfying the hypothesis of the theorem for the constants $L(n)$ and $D(n)$. Suppose $\ep_0(n)$ and $\ep_1(n)$ are the restrictions of $\ep(n)$ to the components of $\D M$ respectively.  

If $d_{\CC(F)}(\ep_i(n),\mu_i) \to \infty$ for both $i=0$ and $i=1$ then
we can directly apply Theorem \ref{eventually faithful convergence and bounded combinatorics} to
conclude that $(\rho_n)$ has a subsequence
that converges to a discrete and faithful representation
$\rho$, with two simply degenerate ends. If, say,
$d_{\CC(F)}(\ep_0(n),\mu_0)$ remains bounded on a subsequence, then
the representations remain in a compact set and again we have a
convergent subsequence, which as in the proof of Theorem \ref{eventually faithful convergence},
must have a discrete and faithful limit $\rho$. Moreover since
 $d_{\CC(F)}(\ep_1(n),\mu_1)\to\infty$, it follows from Thurston's theory of degenerate ends of surface groups that one end of $\rho$ must be
 simply degenerate. As usual, we also assume $N_G$ is the Gromov-Hausdorff limit of $(N_n)$ and is covered by $N_\rho$.

It is a consequence of Thurston's Covering Theorem that either $N_G$
has a simply degenerate end which is finitely covered by a simply
degenerate end of $N_\rho$, or that the covering map $N_\rho\to N_G$
is virtually induced by a fibration. If $N_G$ has a simply degenerate
end, associated to a surface $F'$ which is finitely covered by $F$, we
can find an embedding $F'\times[0,1]\hookrightarrow N_G$ which lifts
to an embedding $\BM_M\hookrightarrow N_\rho$ in the homotopy class
determined by $\rho$. As in the proof of Theorem \ref{thick lifted
  embedding}, post-composing these maps with the approximating maps
$N_G\to N_n$ provides lifted embeddings $\BM_M\to N_n$ in the homotopy
class determined by $\rho_n$ which factor through embeddings
$F'\times[0,1]\hookrightarrow N_n$. We may select the embedding
$\BM_M\hookrightarrow N_\rho$ to be bilipschitz, and it follows that
the maps $\BM_M\to N_n$ are lifted uniformly bilipschitz embeddings. 

If $\rho$ is induced by a virtual fibration, the cover factors through
$N_\rho\to \hhat N_G \to N_G$ where $\hhat N_G$ is fibered by an
embedding of $F$. If $N_G$ itself is fibered with fiber $F'$, which is
covered by $F$, we can find an embedding
$F'\times[0,1]\hookrightarrow N_G$, which lifts to an embedding
$\BM_M\hookrightarrow N_\rho$, in the homotopy class of $\rho$. 
Then the same argument as above applies. So for $n$ sufficiently large, we
find lifted $K$-bilipschitz embeddings of $\BM_M$ in the homotopy
class of $\rho$, with $K$ independent of $n$. This shows the given
sequence $(\rho_n)$ cannot give counterexamples to the conclusion of
the theorem, unless $\rho$ is
induced by a virtual fibration which is {\em not} a fibration. 
But in that case $N_G$ is closed, therefore $N_n$ is
isometric to $N_G$ for $n$ sufficiently large and $\rho_n=\rho$ is
also induced by a virtual fibration which does not cover a fiber. 
\end{proof}

%%%%%%%%%%%%%%%%%%%%%%%%

\subsection{Interpolations of lifted embeddings}\label{subsec: interpolations}

The next theorem shows how lifted bilipschitz embeddings of decorated
$I$-bundles of the form $I_F[\mu]$ can be connected via immersed
locally bilipschitz interpolations. Before stating the theorem,
we recall some facts about the theory of {\em surface groups}.
The theory refers to a long line of work by Thurston, Bonahon,
Minsky, Brock-Canary-Minsky among others.
We restrict to when $F$ is a closed orientable surface of genus
at least 2 and by a {\em surface group}, we mean a discrete faithful
representation $\rho:\pi_1(F)\to\PSL_2(\BC)$. We state the
following corollary of the theory. The first part is a consequence of
the work of Minsky in \cite{ELC0} and more generally of the construction
of combinatorial models for general surface groups in \cite{ELC1,ELC2}.
The second part is really a corollary of an argument of Thurston
\cite{Thu79} about limits of surface groups. 

\begin{lemma}\label{lemma from surface groups}
	Given a closed surface $F$ of genus at least 2, $\ep_0$ fixed, 
and $k_0>0$, there exists $d_0$ so that the following holds.
Assume $\sigma_0,\sigma_1,\sigma_2$ are points ordered on a 
Teichm\"uller geodesic segment in the $\ep_0$-thick part of the surface
$F$ and with successive distances in $[d_0,2d_0]$. Also let $N$
be a hyperbolic manifold associated to a discrete and faithful
representation of $\pi_1(F)$. Then
	\begin{enumerate}
		\item If $f_i:\sigma_i\to N, i=1,2,$ are
$(k_0+1)$-bilipschitz embeddings in the preferred homotopy class, then there
exists a bilipschitz embedding $\BM_F[\sigma_0,\sigma_1]\to N$
that extends $f_1$ and $f_2$, with bilipschitz constant depending
on $F, \ep_0$ and $k_0$.
		\item If $f_i:\sigma_i\to N, i=0,1,2$, are $(k_0+1)$-lipschitz embeddings
in the preferred homotopy class, then the image of $f_1$ separates
the images of the other two.
	\end{enumerate}
\end{lemma}

We use this lemma for sufficiently injective lifted bilipschitz 
embeddings of $F$ in a general hyperbolic manifold. 

\begin{theorem}\label{interpolations}
  Suppose $F$ is a closed surface of genus $>1$ and $R, k_0>1$ are
  given, then $K,L,D$ and $d_0$ exist so that the following
  holds. Assume $\ep_0$ and $\ep_1$ are complete markings on $F$ with
  $R$-bounded combinatorics and $d_F(\ep_0,\ep_1)\ge D$. Let
  $N$ be a hyperbolic 3-manifold admitting a homomorphism
  $\rho:\pi_1(F)\to \pi_1(N)$ such that $\ell_\rho(\ep_0),
  \ell_\rho(\ep_1) \le R$; also assume $\sigma_0$ is a point
	on the Teichm\"uller geodesic $[\sigma_{\ep_0},\sigma_{\ep_1}]$
	with the property that $\rho$ is $L$-injective with respect 
	to $\sigma_0$. Then
  \begin{enumerate}
    \item Either $\rho$ is induced by a virtual fibration which is not a 
		 fibration or there is a lifted
    $K$-bilipschitz embedding $f:\sigma_0\to N$ in the homotopy
    class determined by $\rho$.
    \item Given $\sigma_1\in [\sigma_{\ep_0},\sigma_{\ep_1}]$ with $d_0\le
      d_{\CT(F)}(\sigma_0,\sigma_1)\le 2d_0$ and lifted $k_0$-bilipschitz embeddings $f_{i}:\sigma_i\to N$ in the homotopy class
      determined by $\rho$, $i=0,1$, there is a locally $K$-bilipschitz
      immersion
    \[ \BM_F[\sigma_0,\sigma_1]\to N\] 
    whose restrictions to the boundary components are $f_{0}$ and $f_{1}$.
    \item Given $\sigma_1,\sigma_2 \in [\sigma_{\ep_0},\sigma_{\ep_1}]$ with
      $\sigma_0<\sigma_1<\sigma_2$, in the order given by choosing an
		orientation on the Teichm\"uller geodesic $[\sigma_{\ep_0},\sigma_{\ep_1}]$,
		with successive $\CT(F)$ distances in $[d_0,2d_0]$, and lifted $k_0$-bilipschitz embeddings $f_i:\sigma_{i}\to N$ in the homotopy class determined by $\rho$, $i=0,1,2$, 
      every pair
      of locally $K$-Lipschitz immersions
    \[ \BM_F[\sigma_0,\sigma_1], \BM_F[\sigma_1,\sigma_2]\to N \] 
    that extend $f_0,f_1,$ and $f_2$, must have consistent orientations.
  \end{enumerate}
\end{theorem}

\begin{proof}
The first part is a simple consequence of theorem 
\ref{thick lifted embeddings of I-bundles}. Let $M=I_F[\nu_0]$ 
with $\nu_0=\nu_{\sigma_0}$ and $\ep=\ep_0\cup\ep_1$. 
Our discussion in \S \ref{subsec: teichmuller}
shows that $\ep$ has $R'$-bounded combinatorics as a marking on the
boundary of the decorated $I$-bundle $M$ with $R'$ depending on $R$ 
and the topology of $F$. Then it follows that if at least one of
the distances $d_F(\nu_0,\ep_0)$ or $d_F(\nu_0,\ep_1)$ is large, and 
$\rho$ is not induced by a virtual fibration which is not a fibration, 
there is a lifted bilipschitz embedding of $\BM_M$ in the homotopy class
of $\rho$ and then obviously there is a lifted bilipschitz embedding
of $\sigma_0$. Of course, the height requirement is always satisfied
if $d_F(\ep_0,\ep_1)$ is large.
The only additional thing we need to show is that
the height requirement and the bilipschitz constant depend only on $F$ 
(and not on the choice of the markings). This follows from the 
observation that under the action of self-homeomorphisms of $F$, there
are only finitely many decorated manifolds of the type $I_F[\mu]$; even 
more the choice of generating sets for $\pi_1(F)$ with respect to 
$I_F[\mu]$ is made invariant under this action. 

The other two parts are proved again by using geometric and
algebraic limits. 
Recall from \S \ref{subsec: teichmuller} that because of
the $R$-bounded combinatorics, the Teichm\"uller geodesic
stays in a uniform thick thick part of $\CT(F)$.
Use lemma \ref{lemma from surface groups} to find a constant
$d_0$ and assume the conclusion of part (2) or (3) of 
the theorem do not hold for a sequence of counterexamples 
$(\rho_n:\pi_1(F)\to\pi_1(N_n))$ with markings $\ep_0(n),\ep_1(n),$
constants $D(n), L(n),$ and $K(n)$ tending to infinity, 
and $\sigma_0(n)<\sigma_1(n)<\sigma_2(n)$ ordered along 
$[\sigma_{\ep_0(n)},\sigma_{\ep_1(n)}]$ with successive $\CT(F)$ 
distances in $[d_0,2d_0]$. After precomposing each $\rho_n$ with 
an automorphism of $\pi_1(F)$ and passing to a subsequence, 
we can assume $\nu_0=\nu_{\sigma_0(n)}$ is independent of $n$. 
Then an argument similar to the above shows that 
$\ep(n)=\epsilon_0(n) \cup \epsilon_1(n)$ has $R'$-bounded 
combinatorics as a marking on the boundary of $I_F[\nu_0]$ and 
at least one height of $\epsilon(n)$ tends to infinity, as $n\to\infty$. 
By theorem \ref{thick lifted embeddings of I-bundles} 
after passing to a subsequence, we assume $(\rho_n)$, as a 
sequence of representations of $\pi_1(I_F[\nu_0])$, converges 
algebraically to a discrete and faithful representation $\rho$, 
where $N_\rho=\BH^3/\rho(\pi_1(F))$ is a surface group with at least one 
simply degenerate end. We also pass to a subsequence and 
assume a choice of base point for $N_n$ is made so that the 
sequence $(N_n)$ converges, in the Gromov-Hausdorff topology, 
to a hyperbolic 3-manifold $N_G$ covered by $N_\rho$. 

Suppose the lifted $k_0$-bilipschitz embeddings $f_i(n):\sigma_i(n)\to
N_n$, $i=0,1,2$, are given. When $L(n)$ is large, the $L(n)$-injectivity
of $\rho(n)$ with respect ot $\sigma_0(n)$ implies that the $\rho(n)$-image
of $\pi_1(F)$ is non-elementary.
Since $\nu_0$ has bounded length on $\sigma_i(n)$
independently of $n$, it follows that the image of $f_i(n)$ stays within
a bounded distance from the base point of $N_n$ with a bound independent of $n$.
Therefore for $n$ sufficiently large and $i=0,1,2$, 
the composition of the map $f_i(n)$ and the
approximating maps $N_n\to N_G$, provides a lifted $(k_0+1)$-bilipschitz embedding 
$g_i(n): \sigma_i(n)\to N_G$. The map 
$g_i(n)$ is in the same homotopy class as the projection $N_\rho\to N_G$ 
and therefore $g_i(n)$ lifts to a $(k_0+1)$-bilipschitz {\em embedding} 
$\cover {g_i(n)} : \sigma_i(n)\to N_\rho$ in the homotopy class determined 
by $\rho$. Our assumption about $d_0$ shows for every $n$ large, the images 
of $\cover{g_0(n)},\cover{g_1(n)},$ and $\cover{g_2(n)}$ are disjoint and the 
image of $\cover{g_1(n)}$ separates the other two. Even more there is a bilipschitz 
embedding $\BM_F[\sigma_0(n),\sigma_1(n)]\to N_\rho$ which extends the embeddings 
$\cover{g_0(n)}$ and $\cover{g_2(n)}$, and the bilipschitz constant is independent 
of $n$. Postcomposing this map with the projection $N_\rho\to N_G$ and the approximation 
map $N_G\to N_n$, we obtain a locally uniformly bilipschitz embedding of 
$\BM_F[\sigma_0(n),\sigma_1(n)]$ which extends $f_0(n)$ and $f_1(n)$. This 
construction and the usual argument using contradiction proves the second part of 
the theorem and provide the constant $K$.

On the other hand, if 
\[ H_1(n):\BM_F[\sigma_0(n),\sigma_1(n)]\to N_n \ 
{\rm and} \ H_2(n):\BM_F[\sigma_1(n),\sigma_2(n)]\to N_n\] 
also extend 
$f_0(n),f_1(n),f_2(n)$ and are locally $K$-bilipschitz immersions, for 
$n$ sufficiently large, we postcompose those with the approximating maps 
$N_n\to N_G$ and lift to obtain immersions 
\[ \BM_F[\sigma_0(n),\sigma_1(n)]\to N_\rho \quad {\rm and} \quad  
\BM_F[\sigma_1(n),\sigma_2(n)]\to N_\rho \]
that extend $\cover{g_0(n)}, \cover{g_1(n)}, \cover{g_2(n)}$. But we 
knew that the image of $\cover{g_1(n)}$ separates the images of 
$\cover{g_0(n)}$ and $\cover{g_2(n)}$; therefore the immersions into 
$N_\rho$ have to have consistent orientation. This implies that $H_1(n)$ 
and $H_2(n)$ have to have consistent orientation. This proves the third 
part of the theorem.
\end{proof}

%%%%%%%%%%%%%%%%%%%%%%%%%

\subsection{Extending lifted embeddings to immersions}\label{subsec: extending to immersion}

The final theorem of the section allows us, again in the presence of bounded
combinatorics and almost-injectivity, to extend
lifted embeddings on the boundaries of a model manifold
$\BM_M[\tau]$, for a decorated manifold $M$ and 
hyperbolic structure $\tau$ on $\D_0 M$, to
immersions of the entire model. Recall from \S \ref{subsec: models} 
that given conformal structure $\tau$ on $\D_0 M$ with  
$R$-bounded combinatorics, $\BM_M[\tau]$ is the model metric obtained via 
gluing $\BM_M$ and $I$-bundles $\BM_F[\sigma_{\mu_F},\tau|_F]$ 
along identical boundary components for every component
$F$ of $\D_0 M$. To simplify the notations, we use the shorter
notation $\BM_F[\mu,\tau]$ instead of $\BM_F[\sigma_{\mu_F},\tau|_F]$.

\begin{theorem}\label{bilipschitz embedding of models}
  Given a decorated 3-manifold $M$, $\forall R>0 \ \exists L'
        \ \forall d \ \exists L, K, D$ such that the following
        holds.  Suppose
\begin{enumerate}
  \item $\tau$ is a hyperbolic structure on $\D_0 M$ with $R$-bounded combinatorics and all heights at least $D$,
  \item $\rho:\pi_1(M)\to\pi_1(N)$ is an $L$-injective homomorphism for a hyperbolic $3$-manifold $N$,
  \item for every component $F\subset \D_0 M$, and every point $\sigma$ 
  of the Teichm\"uller geodesic connecting $\sigma_{\mu_F}$ and $\tau|_F$  whose distance from $\sigma_{\mu_F}$ is at
    least $d$, the induced representation 
$$\rho_{F} = \rho \circ (F\hookrightarrow M)_*$$
is
    $L'$-injective, with respect to $\sigma$, and 
  \item there is a map $g_\tau:\D_0 M\to N$ whose restriction to each
    component $F$ of $\D_0 M$ is a lifted $R$-bilipschitz embedding of
    $\tau|_F$ in the homotopy class determined by $\rho$, 
\end{enumerate}
    then there exists
    a locally $K$-bilipschitz immersion $f:\BM_M[\tau] \to N$ which
	restricts on $\D_0\BM_M[\tau]$ to $g_\tau$ and on every
	toroidal component of boundary of $\BM_M[\tau]$, it is a covering map onto 
	the boundary of a component of the thin part of $N$.
\end{theorem} 

\begin{proof}
Fix $M$ and $R>0$ for the remainder of the proof. 

\subsection*{Lipschitz maps and homotopies.} Using Theorem
\ref{thick lifted embedding} and by letting 
 \[ \ep=\nu_\tau = \bigcup_{F\subset \D_0 M}\nu_{\tau|F},\]
we can choose $L_1, K_1$ such that if
$\rho$ is $L_1$-injective there is a $K_1$-Lipschitz map of 
$\BM_M$ in the homotopy class of $\rho$. 
(That theorem provides a stronger conclusion that the map 
is a lifted bilipschitz embedding, but at this stage we 
do not need that.) In particular this gives upper bounds 
$R_1\ge R$ on the lengths $\ell_\rho(\mu(M))$ in $N$.

\subsection*{Models of the ends.} 
Let $F$ be a component of $\D_0 M$. Since restriction of $g$
to $F$ is a lifted embedding, we can use lemma \ref{embedding implies fibration}
and see that if $\rho\circ(F\hookrightarrow M)_*$ is a 
virtual fibration then it must be induced by a map
that covers a fiberation. With this information, we apply Theorem
\ref{interpolations} part (1) to $\rho|_{\pi_1(F)} $ with markings
$\mu_F$ and $\ep|_F$, the length bound $R_1$ yields $K_2$, $L_2$ and
$D_2$ such that, assuming $D$ is at least $D_2$, 
for each $\sigma\in[\sigma_{\mu_F},\tau|_F]$ such that 
$\rho\circ(F\hookrightarrow M)_*$ is
$L_2$-injective with respect to $\sigma$, we have a lifted
$K_2$-bilipschitz embedding $f_\sigma: \sigma\to N$ in the
homotopy class determined by $\rho$. We may assume that $K_2>R$, 
and we choose $f_{\tau|F} = g_\tau|_F$.
From now on we will assume that $L' \ge L_2$.

Now let $K_3,L_3, D_3$ and $d_0$ be the 
constants given by Theorem \ref{interpolations} when the length bound
is $R_1$ as above, and the bilipschitz bound $k_0$ is equal to $K_2$.
We assume $L'\ge L_3\ge L_2$, and fix any $d>0$. Now as explained above
hypothesis (3) implies via Theorem \ref{interpolations} part (1) 
that a lifted $K_2$-bilipschitz map $f_\sigma$ exists for
each $\sigma\in[\sigma_{\mu_F},\tau|_F]$ whose $\CT(F)$-distance 
from $\sigma_{\mu_F}$ is at least $d$. 

Choose a sequence of points $\sigma_1, \sigma_2 , \ldots, \sigma_m =
\tau|_F$ in $[\sigma_{\mu_F},\tau|_F]$ with 
$\sigma_{\mu_F}<\sigma_1<\sigma_2<\cdots<\sigma_m$,
in the order given by the Teichm\"uller geodesic, whose successive distances in
$\CT(F)$ lie in $[d_0,2d_0]$, with $d_\CT(\sigma_{\mu_F},\sigma_1)=d$.
Then part (2) of Theorem \ref{interpolations} implies that,
if $D\ge D_3$ and
for every $i=1,\ldots, m-1$, there is a locally $K_3$-bilipschitz immersion
$\BM_F[\sigma_i,\sigma_{i+1}] \to N$ that extends $f_{\sigma_i}$ and
$f_{\sigma_{i+1}}$. Furthermore, part (3) of the same theorem shows
these immersions have consistent orientations. Thus they assemble
to an immersion of $\BM_F[\sigma_1,\sigma_m]$ into $N$.

The distance between $\sigma_1$ and $\sigma_{\mu_F}$ 
is $d$ and $\sigma_m=\tau|_F$, so there is 
a bilipschitz map $\Psi$
between $\BM_F[\mu,\tau]$ and $\BM_F[\sigma_1,\sigma_m]$. 
Thus composing our immersion with $\Psi$ and
for every component $F$ of $\D_0 M$, we get a
locally $K_4$-bilipschitz immersion 
  \[ \Phi_F: \BM_F[\mu, \tau] \to N \]
in the homotopy class determined by $\rho$ whose restriction to the
$\tau$-side of $\BM_F[\mu,\tau]$ is identical
to $g_\tau|_F$, with $K_4$ depending on $M, R$ and the constant $d$. 

\subsection*{Extension to model core}
To finish the theorem, we claim there exist $D_5,L_5$ and $K_5$
such that, provided hypothesis (3) holds for the same $L'$ as above,
(1) and (4) hold with $D\ge D_5$, and (2) holds with $L\ge L_5$, we
have a locally $K_5$-bilipschitz immersion $\BM_M[\tau]\to N$
in the homotopy class of $\rho$ and extending the maps $\Phi_F$.

Suppose, by contradiction, that for a sequence $\rho_n:\pi_1(M)\to\pi_1(N_n)$, hyperbolic structures $\tau(n)$ on $\D_0 M$ and
$K_5(n),L_5(n)$ and $D_5(n)$ all going to $\infty$, the
hypothesis of the theorem holds and the conclusion fails. 

Theorem \ref{eventually faithful convergence and bounded combinatorics} tells us that (up to subsequence and
conjugation) $(\rho_n)$ converges to a discrete and faithful
$\rho:\pi_1(M)\to \PSL_2(\BC)$, where $N_\rho = \Hyp^3/\rho(\pi_1(M))$
is a hyperbolic structure on $M$ with simply degenerate ends. Passing
to a further subsequence, we can assume the sequence of manifolds
$(N_n)$ converges to a geometric limit $N_G$ which is covered by
$N_\rho$. 
 
When $n$ is large enough so that the constants 
$K_5(n),L_5(n)$ and $D_5(n)$ are bigger than the 
corresponding constants in the previous steps, for 
every component $F$ of $\D_0 M$, we have a locally 
$K_4$-bilipschitz immersion
\[ \Phi_{F,n}: \BM_F[\mu,\tau(n)] \to N_n \]
in the homotopy class determined by $\rho$, whose 
restriction to the $\tau(n)$-boundary is identical 
to the corresponding restriction of $g_{\tau(n)}$.

Passing to further subsequences, we assume the sequence of maps
$(\Phi_{F,n})$ converges to a locally $K_4$-bilipschitz immersion 
\[ \Phi_{F,G}:\CE_F \to N_G, \]
where $\CE_F$ is a geometric
limit of the manifolds $\CE_{F,n} = \BM_F[\mu,\tau(n)]$ as
  $n\to\infty$.  
In fact, this geometric limit is identical to the subset of the 
universal curve over an infinite Teichm\"uller ray $[\sigma_{\mu_F},\lambda|_F)$
where $\lambda|_F$ is a projective measured lamination on $F$. 
Note that by lemma \ref{limits of bounded combinatorics bind}, 
since $(\tau(n))$ have $R$-bounded combinatorics and heights
tending to infinity, after passing to a subsequence, the corresponding 
markings $\nu_n=\nu_{\tau(n)}$ converge to a full binding lamination
with the same support as $\lambda = \bigcup_F\lambda|_F$. 

 For every component $F$ of $\D_0 M$, the immersion $\Phi_{F,G}$
 lifts to 
\[ \Phi_{F,\infty}:\CE_F \to N_\rho,\]
in the homotopy class determined by the restriction of $\rho$ to
$\pi_1(F)$. We claim that $\Phi_{F,\infty}$ is proper. To see this,
consider any sequence of points $p_n$ going to infinity
in $\CE_F$. Each $p_n$ is in a fiber of the
  surface bundle whose geometry is determined by a point in the
  Teichm\"uller ray with endpoint $\lambda|_F$. Hence each such
  fiber contains a curve $\gamma_n$ passing through $p_n$, such that the lengths of
  $\gamma_n$ is bounded and as curves in $F$ they converge to the
  lamination $\lambda|_F$. 
Since $\lambda$ is binding, $\lambda|_F$ is a filling Masur domain 
  lamination and as explained in \S \ref{subsec: geometry of ends}, 
  the sequence $(\gamma_n)$ must correspond to infinitely many
  homotopy classes in $M$; 
  the length bound for $(\gamma_n)$ in $\CE_F$ gives a length
  bound for their $\Phi_{F,\infty}$-images in $N_\rho$ and
  therefore these images cannot remain in a compact subset of 
  $N_\rho$. This proves the claim, and moreover that $\lambda_F$ is the
  ending lamination of the end $\CE'_F$ of $N_\rho$ associated to $F$. (When
  $M$ is an $I$-bundle there are two possible such ends of $N_\rho$ -- we shall
  return to this point later.)

We further claim that $\Phi_{F,\infty}$
restricted to a neighborhood of the end of
$\CE_F$ is a homeomorphism to its image, a neighborhood of the end
$\CE'_F$. This is an immediate
consequence of the following lemma:

\begin{lemma}\label{immersion}
Let $f:U \to V$ be a proper immersion, where both $U$ and $V$ are
copies of $F\times[0,\infty)$. Then there are compact subsets of $U$
  and of $V$ whose complements are homeomorphic via $f$
(In particular $f$ must be a homotopy equivalence). 
\end{lemma}

\begin{proof}
For a subset $J$ of $[0,\infty)$ let $U_J$ denote $F\times J$ with
respect to the product structure of $U$, and let $U_t =
U_{\{t\}}$. Define $V_J$ and $V_t$ similarly.

Since $f$ is a proper immersion, it is a covering map over $V\setminus
f(\boundary U)$. Let $V'$ be the unbounded component of 
$V\setminus f(\boundary U)$. Since $f$ is proper, $f^{-1}(V')$ is
nonempty, and hence $V'\subset f(U)$. Choose $t$ such that
$V_{[t,\infty)}\subset V'$. Now $V_t$ is covered by $H_t =
  f^{-1}(V_t)$, so $H_t$ is incompressible in $U$: 
an essential compressing disk would map under $f$ to
  an essential compression of $V_t$.
It follows, since $U\approx F\times[0,\infty)$, that $H_t$ is a union of copies of $F$,
each isotopic to a level surface $U_s$. Hence
$f^{-1}(V_{[t,\infty)})$, being a cover of $V_{[t,\infty)}$, is a disjoint union of homeomorphic copies
    of it. All of them are unbounded in $U$ since $f$ is
      continuous, and it follows that there can only be one of
      them. We conclude that $f$ is a homeomorphism from
      $f^{-1}(V_{[t,\infty)})$ to $V_{[t,\infty)}$, as desired. 
\end{proof}

Let $\BM_M[\lambda]$ be obtained from gluing $\CE_F$ to $\BM_M$
along the identical boundary components for every component $F$ of
$\D_0 M$. (This is just a simple generalization of the construction of
the models $\BM_M[\ep]$ in \S\ref{subsec: models}.) 
Since the maps $\Phi_{F,\infty}$, for all components $F$ of $\D_0 M$,
are in the homotopy class determined by $\rho$, we may extend them 
across $\BM_M$ to a map
\[ H_\infty: \BM_M[\lambda] \to N_\rho \]
in this homotopy class.
Furthermore, we may
assume that, on every toroidal component of $\D \BM_M$, $H_\infty$ is a
homeomorphism onto the boundary of the associated cusp of $N_\rho$.

We claim now that there is a compact core $C_M\subset
\BM_M[\lambda]$ such that $H_\infty$ maps $C_M$ by a homotopy
equivalence to its image $C_N$, and is a homeomorphism in its
complement. When $M$ is not a product $F\times I$, this is a consequence of
the fact that $\Phi_{F,\infty}$ is a homeomorphism on an end of each
$\CE_F$, since the map must take distinct ends to distinct ends (being
a proper homotopy equivalence). When
$M$ is $F\times I$ there is no topological obstruction to mapping both
ends to a single end. However, we know that the two components of $\lambda$
are distinct filling laminations (using the bounded combinatorics
condition on the sequence $\tau(n)$). These laminations map to the
ending laminations of the ends of $N_\rho$, which must therefore be
distinct ends. 

It then follows, applying a generalization of a theorem of
Waldhausen \cite{Wa, Jaco, Tu},
that $H_\infty|_{C_M}$ is homotopic to a homeomorphism
$C_M\to C_N$, by a 
homotopy that does not move any point on the boundary. As a result
$H_\infty$ is homotopic to a homeomorphism $\Phi_\infty$ onto the
complement of cusps of $N_\rho$ with a homotopy supported on the compact
subset $C_M$. We may further assume $\Phi_\infty$ is
$K_5$-bilipschitz using the fact that it is already $K_4$-bilipschitz
on the ends.

To complete the proof we need an interpolation argument to patch
together the homeomorphism $\Phi_\infty$ in the limit with the partial
immersions $\Phi_{F,n}$ in the sequence to obtain uniform global
immersions, and thus contradict our choice of sequence.

For brevity denote $\BM_n = \BM_M[\tau(n)] $ and
$\BM_\infty = \BM_M[\lambda]$. The geometric convergence of
$\BM_n\to\BM_\infty$ and $N_n\to N_G$ gives us comparison maps
$\sigma_n$ and $\kappa_n$ which fit into the following diagram:
\begin{equation*}
\xymatrix{
\BM_n  \ar[dd]^{Q_n} & \BM_\infty \ar[dr]^{\Phi_\infty}
\ar[l]^{\sigma_n} & \\
&            &     N_\rho \ar[dl]^{\Pi} \\
N_n & N_G \ar[l]^{\kappa_n} & 
}
\end{equation*}
and which converge in $\CC^{1,\alpha}$, for every $\alpha<1$, 
to isometries on larger and larger
compacta in $\BM_\infty$ and $N_G$, respectively. Recall from \S
\ref{subsec: models} that the models $\BM_n$ are equipped with 
Riemannian metrics on the gluing regions and therefore the limit
$\BM_\infty$ is also equipped with a Riemannian metric outside of
$C_M$. The maps $Q_n$ are to
be defined, and will be our desired immersions for $n$ large enough.

The map $\Phi_\infty$ restricted to each component $\CE_F$ of
$\BM_\infty\setminus \BM_M$ is, by construction, the lift of the 
limit (in $\CC^1$, on compact sets) of
$\kappa_n^{-1} \circ \Phi_{F,n}\circ \sigma_n$ -- equivalently
$\Pi\circ\Phi_\infty|_{\CE_F}$ is the limit of 
$\kappa_n^{-1} \circ \Phi_{F,n}\circ \sigma_n$. 

On the overlap regions $C_M \intersect \CE_F$ in $\BM_\infty$, we can
therefore use a partition of unity to interpolate between the maps
$\kappa_n^{-1}\circ \Phi_{F,n}\circ\sigma_n$ and $\Pi\circ \Phi_\infty$, to obtain maps $\Psi_n$
which converge in $\CC^1$ to $\Pi\circ\Phi_\infty$ on  compact sets, but
which are eventually equal to
$\kappa_n^{-1}\circ\Phi_{F,n}\circ\sigma_n$ on any bounded subset of 
the exterior of $C_M$.
Since the property of being an immersion is open in $\CC^1$, we may
conclude that $\Psi_n$ is an immersion on each compact set in
$\BM_\infty$,  for sufficiently large $n$.

Now we can define $Q_n$ by letting it take the values $\Phi_{F,n}$ on
each component of the exterior of $\sigma_n(C_M)$ in $\BM_n$, and
letting $Q_n =   \kappa_n \circ  \Psi_n \circ
\sigma_n^{-1}$ on $\sigma_n(C_M)$. This is an immersion in the
homotopy class determined by $\rho_n$, it satisfies
uniform bilipschitz bounds,  restricts to $g_{\tau(n)}$ on $\D_0\BM_M[\tau(n)]$,
and restricts to a covering of boundary of a component of the thin
part of $N$ on every toroidal component of $\D\BM_M[\tau(n)]$. This
contradicts the assumption and proves the theorem. 
\end{proof}

%%%%%%%%%%%%%%%%%%%%%%%%%%

\subsection{Consistent orientation of immersions}\label{subsec: consistent orientation}
Suppose $X$ is an $(\CM,R)$ gluing with the collection $\Xi$ of pieces
and gluing involution $\Psi$. For a component $F\subset\D_0M$ of a piece
$M$ in $\Xi$, recall that $\mu_F$ is the decoration on $F$ and let $\nu_F$
denote the $\Psi$-image of the decoration of the piece $M'$ adjacent along
$F$. On $F$ we choose the hyperbolic structure $\tau_F$ which is the
midpoint of the Teichm\"uller geodesic $[\sigma_{\mu_F},\sigma_{\nu_F}]$.
The the union of all these hyperbolic structures on $\D_0 \Xi$ is
invariant under $\Psi$.

In addition assume $\rho:\pi_1(X)\to\pi_1(N)$ is a
homomorphism with $N$ a hyperbolic 3-manifold and 
there is a map $g_\tau:\D_0\Xi \to N$ which on
every component $F$ restricts to a lifted $R$-bilipschitz
embedding $g_\tau|_F:\tau_F\to N$ in the homotopy class
of $\rho$, and has the property that $g_\tau \equiv g_\tau\circ \Psi$.
Finally assume that with respect to each of the pieces
$M\subset\Xi$,
hyperbolic structure $\tau_M = 
\bigcup_{F\in\D_0 M} \tau_F$, and map $g_\tau|_{\D_0 M}$, the induced
homomorphism $\rho_M = \rho \circ (M\hookrightarrow X)_*$ satisfies the 
hypothesis and therefore conclusion of Theorem 
\ref{bilipschitz embedding of models} for uniform 
constants $R, L', d, L, K$ and $D$. 

Then we claim that the obtained 
$K$-bilipschitz immersions $f_M:\BM_M[\tau_M] \to N$ will have 
consistent orientation. This is a consequence of part (3) of 
Theorem \ref{interpolations}. More precisely, using the
notation in the above proof, in the process of the construction 
of the map $f_M$, we chose a point $\sigma_{m-1}$ in the
Teichm\"uller geodesic $[\sigma_{\mu_F},\tau_F]$ whose 
$\CT(F)$-distance to $\tau_F$ is in $[d_0,2d_0]$. Then we have 
a lifted $K_2$-bilipschitz embedding $\sigma_{m-1}\to N$
in the homotopy class of $\rho$ and the global map $f_M$
restricts to an immersion 
\[ \BM_F[\sigma_{m-1},\tau_F]\to N \]
which extends $g_\tau|_E$ and the lifted embedding of $\sigma_{m-1}$. 

If $M$ and $M'$ are adjacent along $E$, then similarly
(and after translating via $\Psi$) we have a point
$\sigma_{m-1}'$ in the Teichm\"uller geodesic 
$[\tau_F, \sigma_{\nu_F}]$, and we have a lifted 
$K_2$-bilipschitz embedding $\sigma_{m-1}'\to N$
in the homotopy class of $\rho$. Also $f_{M'}$
(after an appropriate translation with $\Psi$) restricts to 
an immersion
\[ \BM_F[\tau_F,\sigma_{m-1}']\to N\]
that extends $g_\tau|_E$ and the lifted embedding of $\sigma_{m-1}'$.
Now we apply part (3) of Theorem \ref{interpolations}
to see these immersions have to have consistent orientations
and therefore $f_M$ and $f_{M'}$ have consistent orientations. 

 As a
consequence of this and the fact that $\BM_X$ is isometric to
the union of models $\BM_M[\tau_M]$ for $M\subset\Xi$ 
with boundary identifications given by $\Psi$ and
the assumption that the lifted embeddings of the midpoint
surfaces for adjacent pieces are identified in $N$, 
i.e, $g_\tau \equiv g_\tau\circ\Psi$, we obtain a 
locally bilipschitz immersion $\BM_X\to N$ with a bilipschitz
constant that depends only on $\CM$ and the given constants.

%%%%%%%%%%%%%%%%%%%%%%%%%%%%%%%%

\section{Bilipschitz models for geometrically finite manifolds}\label{sec:bilipschitz models for convex cocompact}

In this section we state and prove that our model metrics 
are uniformly good for describing geometrically finite 
structures with {\em $R$-bounded combinatorics}. The 
theorem can be generalized to a bigger class of hyperbolic
structures but for our application, we need it only for 
certain cases described below. 

Let $M$ be a decorated manifold. Recall from Ahlfors-Bers
Theory that given a hyperbolic structure $\tau$ on $\D_0 M$
there exists a unique hyperbolic structure on $M$, so that
for every component $E$ of $\D_0 M$, the end associated to $E$
is convex cocompact and the conformal structure at infinity
is given by $\tau|_E$, the restriction of $\tau$ to $E$. 
We denote this hyperbolic structure by $Q(M,\tau)$. These 
provide all {\em geometrically finite} hyperbolic structures
on $M$ with no rank 1 cusp. We use the definition of bounded
combinatorics for hyperbolic structures in \S 
\ref{subsec: teichmuller} and we say $Q(M,\tau)$ has 
$R$-bounded combinatorics if $\tau$ does.
\begin{theorem}\label{models for decorated manifolds}
  	Suppose $M$ is a decorated manifold, and $R>0$. There exist
  	$K, D$ such that if $Q(M,\tau)$ has $R$-bounded combinatorics
  	and heights $\ge D$, there exists a $K$-bilipschitz 
	embedding $\BM_M[\tau] \to Q(M,\tau)$, in the preferred
	isotopy	class of the embeddings of $M$ in $Q(M,\tau)$, and
	the image of the embedding is the complement of the cusps
	in the convex core of $Q(M,\tau)$.
\end{theorem}

\begin{proof}
By definition of $R$-bounded combinatorics, $\tau$ is in the 
$\ep_0$-thick part of the Teichm\"uller space of $\D_0 M$. By a 
theorem of Bridgeman-Canary \cite{BC03}, there exists a 
bilipschitz parametrization $g_\tau:\tau \to Q(M,\tau)$ of
the boundary of the convex core of $Q(M,\tau)$ in the preferred
isotopy class (restricted to $\D_0 M$). (This map is induced
from the extension to the boundary at infinity of the nearest-point 
projection map to the convex core.) The bilipschitz constant 
depends on $\ep_0$, and enlarging $R$ if necessary we may assume 
it is bounded by $R$. When heights of $\tau$ are sufficiently
large, it is easy to see that the hypothesis (1), (2) and (4) of 
Theorem \ref{bilipschitz embedding of models} for the induced 
homomorphism $\rho:\pi_1(M)\to\pi_1(Q(M,\tau))$ and $g_\tau$
are satisfied. Obviously $\rho$ is injective and $g_\tau$ is
an embedding. We will show in the proposition \ref{L-injectivity 
for the boundary} that given $L'$, there exists $d$ such that 
for every component $F$ of $\D_0 M$ and $\sigma\in 
[\sigma_{\mu_F},\tau|_F]$ with $\CT(F)$-distance at least $d$ 
from $\sigma_{\mu_F}$, the homomorphism $F\hookrightarrow M$ is 
$L'$-injective with respect to $\sigma$. Hence the hypothesis
(3) of Theorem \ref{bilipschitz embedding of models} also holds
when the heights are sufficiently large. Therefore there exists
$D, K$ so that if heights of $\tau$ are at least $D$, then there
is a locally $K$-bilipschitz immersion $f:\BM_M[\tau] \to Q(M,\tau)$
which extends $g_\tau$ and restricted to every toroidal boundary
component of $\BM_M[\tau]$ is a covering map on the boundary of a 
component of the thin part of $Q(M,\tau)$. But $f$ is in the homotopy
class of $\rho$ and therefore is a homotopy equivalence, and 
also $g_\tau$ is an embedding. This
implies $f$ is an embedding and can be modified on the toroidal
boundary components so its image is the complement of
the rank 2 cusps in the convex core of $Q(M,\tau)$. Since 
restricted to the boundary, $f$ is in the preferred isotopy
class, it is a consequence of work of Waldhausen \cite{Wa} 
that $f$ is in the preferred isotopy class. This finishes
the proof of the theorem modulo the proof of the next proposition.
\end{proof}

\begin{proposition}\label{L-injectivity for the boundary}
  Given a decorated 3-manifold $M$ and $R,L'>0$, there exists $d$
  depending only on $R,L'$ and
  the topology of $\D M$, such that if
  $\sigma$ is a hyperbolic metric on a component $F$ of $\D_0 M$ with
  $R$-bounded combinatorics relative to $M$ and
  $d_{\CT(F)}(\sigma,\sigma_{\mu_F}) > d$, then the map
  $\iota:\pi_1(F)\to\pi_1(M)$ induced by the inclusion $F\hookrightarrow M$
  is $L'$-injective, with respect to $\sigma$.
\end{proposition}

\begin{proof}
If $\alpha$ is an essential simple loop in $F$ of $\sigma$-length at most $L'$ then
its $\CC(F)$-distance from the associated marking $\nu_\sigma$ is
bounded by some $d_0$, a function of $L'$. The $R$-bounded
combinatorics condition on $\sigma$ implies
(see (\ref{eqn: mu nearest}) in
\S \ref{subsec: bounded combinatorics for decorated manifolds})
that 
$d_F(\nu_\sigma,\Delta(M)) > d_F(\nu_\sigma,\mu_F) - R,$
and hence
$$
d_F(\alpha,\Delta(M)) > d_F(\nu_\sigma,\mu_F) - R - d_0.
$$
Moreover the $R$-bounded combinatorics property for the pair
$\mu_F,\nu_\sigma$ implies 
(see \S \ref{subsec: teichmuller}) that 
$d_{\CT(F)}(\sigma,\sigma_{\mu_F})$ and 
$d_{F}(\nu_\sigma,\mu_F)$ 
are uniformly comparable. Putting these together,  there exists $d$
such that $d_{\CT(F)}(\sigma,\sigma_{\mu_F})>d$ implies that
$d_F(\alpha,\Delta(M)) > 1$, and in particular $\alpha \notin
\Delta(M)$, so that $\iota(\alpha) \ne 1$.

Now consider an arbitrary element $\beta$ of $\pi_1(M)$ of $\sigma$-length
bounded by $L'$. 
It follows from the Loop Theorem that there are a finite number of simple
closed curves $\alpha_1, \ldots, \alpha_k$, obtained by surgery on (the
$\sigma$-geodesic representative of) $\beta$, so that if
$\iota(\beta)=1$ then
$\iota(\alpha_i)=1$
 for some $i\in\{1,\ldots,k\}$. On the other hand each $\alpha_i$ has
 $\sigma$-length bounded by $L'$, so this is impossible by the previous
 paragraph. Hence the kernel of $\iota$ contains no curves of length
 less than $L'$, and the proof is complete.
\end{proof}

For future reference we state the following theorem for limits of the
hyperbolic structures discussed above. Recall from \S 
\ref{subsec: geometry of ends} that a sequence $(\rho_n)$ of representations 
converges {\em strongly} to a representation $\rho$ if the sequence 
converges both algebraically and geometrically to $\rho$. 

\begin{theorem}\label{strong limits of convex cocompacts}
	Suppose $M$ is a decorated manifold and $R>0$. Given a sequence 
	$(\tau_n)$ of hyperbolic structures on $\D_0 M$ with $R$-bounded 
	combinatorics and heights tending to infinity as $n\to\infty$, 
	after passing to a subsequence, the sequence of representations 
	induced by $Q(M,\tau_n)$ converges strongly to a discrete and 
	faithful representation $\rho$ of $\pi_1(M)$, where $N_\rho=
	\BH^3/\rho(\pi_1(M))$ is a hyperbolic structure on $M$ and 
	the end of $N_\rho$ associated to every component of $\D_0 M$ 
	is simply degenerate.
\end{theorem}

\begin{proof}
Suppose $(\tau_n)$ is as in the hypothesis and $\rho_n:\pi_1(M)\to\PSL_2(\BC)$ 
is the representation induced by $Q(M,\tau_n)$. 
By Theorem \ref{eventually faithful convergence and bounded combinatorics},
after passing to a subsequence, we can assume
$(\rho_n)$ converges algebraically to a representation $\rho$, where
$N_\rho$ is a hyperbolic structure on $M$ with the ends associated to
$\D_0 M$ simply degenerate. By Corollary \ref{algebraic limit is a finite cover}
we can assume the sequence
$Q(M,\tau_n)$ converges geometrically to $N_G$ which is finitely
covered by $N_\rho$. So it only remains to show that the covering map
$N_\rho\to N_G$ is one-to-one. We use an argument of
Jorgensen-Marden \cite{jorgensen-marden:convergence}:
Let $\gamma$ be an element of
$\pi_1(N_G)$; then there exists $k$, so that $\gamma^k=\rho(\alpha)$ for
some $\alpha\in\pi_1(M)$. Since $N_G$ is the geometric limit of
$(N_n)$, we can find $\beta_n\in\pi_1(M)$ with $(\rho_n(\beta_n))$
converging to $\gamma$. Since the sequences $(\rho_n(\beta_n^k))_n$
and $(\rho_n(\alpha))_n$ both converge to $\rho(\alpha)$ and $\rho_n$
is discrete and faithful, by the Margulis lemma, we must have
$\alpha=\beta_n^k$ for $n$ sufficiently large. But $\alpha$ has only
finitely many roots and therefore we can pass to a subsequence and
assume $\beta=\beta_n$ is independent of $n$. Then $\gamma = \lim
\rho_n(\beta)$ is in $\rho(\pi_1(M))$. This proves
$\pi_1(N_G)=\rho(\pi_1(M))$ and the covering map $N_\rho\to N_G$ is a
homeomorphism.
\end{proof}
%%%%%%%%%%%%%%%%%%%%%%%%%%%%%%%%%%%

\section{Nearly hyperbolic gluings}\label{sec:nearly hyperbolic gluings}

In Theorem \ref{existence of nearly hyperbolic metrics}
we construct negatively curved metrics on a gluing
with $R$-bounded combinatorics. These negatively curved metrics will
have curvatures as close as required to $-1$ when the heights are
sufficiently large. In fact, it turns out that the height requirements
for every boundary identification in the gluing only depends on the
two adjacent pieces and on $R$.  

The existence of negatively curved metrics obviously provides some
immediate consequences, in particular that the fundamental groups are
hyperbolic. In Corollary \ref{eventual injectivity} and
Theorem \ref{incompressibility of boundary of a gluing}, we
utilize the explicit properties of the constructed metrics to obtain
finer topological data about  approximate  injectivity,
and incompressibility, for the inclusion of a piece into such a gluing.
This will be used essentially in the forthcoming sections
to control the geometry of the true hyperbolic metric on a gluing, and
in fact to show that for sufficiently large height it is not very
different from the negatively curved metric that is produced here.  

The construction of the negatively curved metrics utilizes the
bilipschitz models of the previous section for the convex cocomapct
structures, as well as a gluing argument similar to the one used in
\cite{Na05, NS09}. 

Assume $X$ is an $\CM$-gluing possibly with
boundary and $\lambda$ is a complete marking on $\D_0 X$.
Here we do not restrict to a finite set of decorated manifolds and 
basically we assume $\CM$ is the set of all decorated manifolds. 

\begin{theorem}\label{existence of nearly hyperbolic metrics}
  Given $R, \eta>0$, for every decorated manifold $M$ there exist
  constants $h_\eta(M)$ and $K(M)$ so that the following
  holds. Suppose $X$ is a gluing possibly with boundary and $\lambda$
  is a complete marking on $\D_0 X$ which satisfy the following: 
  \begin{enumerate}
    \item $(X,\lambda)$ has $R$-bounded combinatorics,
    \item for every  adjacent pair $M,M'$ of pieces of $X$, the height for identifying a component of $\D_0 M$ and a component of $\D_0 M'$ in $X$ is at least $h_\eta(M)+h_\eta(M')$, and
    \item for every component $E$ of $\D_0 X$ which is in
      $\D_0 M$ for a piece $M$, the height of $\lambda|_E$ is at least $h_\eta(M)$.
  \end{enumerate}
  Then $X$ admits a $(-1-\eta,-1+\eta)$-pinched negatively curved metric $\theta$ and there is an embedding $\BM_X[\lambda]\to (X,\theta)$ in the isotopy class of the identity map, whose restriction to the pre-image of every piece $M$ of $X$ is $K(M)$-bilipschitz, and its image is a convex core of $(X,\theta)$.
\end{theorem}

It follows from Theorem 
\ref{models for decorated manifolds} that if $\tau$ is a
hyperbolic structure on $\D_0 M$ with $R$-bounded combinatorics
and heights sufficiently large, there is a bilipschitz map
from $\BM_M[\tau]$ onto the complement of the cusps in the 
convex core of $Q(M,\tau)$ in the preferred isotopy class.
The bilipschitz constant depends on $M$ and $R$ and for a 
component $F$ of $\D_0 M$, this bilipschitz map 
restricts to a map $\BM_F[\sigma_F, \tau|_F] \to Q(M,\tau)$ 
which is bilipschitz onto a neighborhood of the boundary 
of the convex core of $Q(M,\tau)$ associated to $F$,
where we use the notation $\sigma_F$ to denote $\sigma_{\mu_F}$.

One can also consider the quasi-fuchsian manifold 
\[ B_{M,\tau,F} \equiv Q(F,\sigma_F,\tau|_F),\] 
which is the convex cocompact structure on
$F\times[0,1]$ with the conformal structure at infinity associated to
$F\times\{1\}$ is $\tau|_F$ and the one associated to
$F\times\{0\}$ is $\sigma_F$, where the latter is considered
with the opposite orientation. 
Again we have a bilipschitz map from
$\BM_F[\sigma_F,\tau|_F]$ to the convex core of $B_{M,\tau,F}$
and the bilipschitz constant depends on $R$ and $F$. This can be
implied from Theorem \ref{models for decorated manifolds} or from the
existing theory of surface groups. 
We fix from now on a choice of base point for $B_{M,\tau,F}$
which is equidistant from the two components of the convex core
boundaries.

Note  that there is a preferred isotopy class of embeddings
$B_{M,\tau,F}\to Q(M,\tau)$ such that the composition with the core embedding
$F\hookrightarrow B_{M,\tau,F}$ is isotopic to the inclusion $F\hookrightarrow
M\to Q(M,\tau)$, given by the preferred inclusion of $M$ as a compact core.

\begin{proposition}\label{close to a quasifuchsian}
  Given $\eta>0$ and $R$, and a decorated manifold $M$, there exists a
  constant $D$ such that if $\tau$ is a hyperbolic structure with
  $R$-bounded combinatorics on $\D_0 M$ and heights at least $D$, and
  $F$ is a component of $\D_0 M$, there is an embedding
  \[ B_{M,\tau,F} \to Q(M,\tau), \] 
in the preferred isotopy class, 
  which is $\eta$-close to an isometry in the $C^\infty$-topology on
  the neighborhood of radius $1/\eta$ centered at the base point of
$B_{M,\tau,F}$.
\end{proposition}

\begin{proof}
  For a decorated manifold $M$, a component $F$ of $\D_0 M$ and fixed
  $\eta>0$ and $R$, suppose $(\tau_n)$ is a sequence of hyperbolic
	structures on $\D_0 M$ with $R$-bounded combinatorics and heights tending to
  infinity. We claim that (restricting to a subsequence) there is a limit  of $Q(M,\tau_n)$, in the
  Gromov-Hausdorff topology, which is isometric to a limit of
  $B_{M,\tau_n,F}$.

Recall that we have embeddings 
  \[ \Phi_n: \BM_F[\sigma_F,\tau_n|_F] \to B_{M,\tau_n,F}\]
  and 
  \[ \Psi_n: \BM_F[\sigma_F,\tau_n|_F] \to Q(M,\tau_n),\]
in the preferred isotopy classes, with uniform bilipschitz
constants. If $x_n$ is the base point of $B_{M,\tau_n,F}$ then
the distance from $o_n = \Phi_n^{-1}(x_n)$ to the boundary components
of $ \BM_F[\sigma_F,\tau_n|_F]$ tends to infinity with $n$ and
therefore after passing to a subsequence we can assume
$(\BM_F[\sigma_F,\tau_n|_F], o_n)$ converge in the Gromov-Hausdorff
topology to a complete manifold without boundary $(\BM_\infty,
o_\infty)$, which in fact will be a subset of the universal
Teichm\"uller curve over a bi-infinite geodesic that 
connects filling intersecting laminations $\lambda^+$ and 
$\lambda^-$. Furthermore and after passing to a further 
subsequence, we can assume the manifolds
$B_{M,\tau_n,F}$ also converge in the Gromov-Hausdorff topology
and the maps $\Phi_n$ converge to provide a bilipschitz
homeomorphism between $(\BM_\infty,o_\infty)$ and this limit. Note
that in particular, there is a lower and upper bound for the
injectivity radius for each of these limits. Passing to 
a further subsequence we may assume the sequence of manifolds
$Q(M,\tau_n)$ with base point $y_n = \Psi_n(o_n)$ converges in the
Gromov-Hausdorff topology, and there is a bilipschitz homeomorphism
from $(\BM_\infty, o_\infty)$ to this limit too. So these limits of
$B_{M,\tau_n,F}$ and $Q(M,\tau_n)$ are bilipschitz and it follows
from McMullen's rigidity theorem for hyperbolic 3-manifolds with
injectivity radius bounds \cite{McM96} that
these two must be isometric. 

Passing this isometry back to the sequence we obtain maps
$B_{M,\tau_n,F}\to Q(M,\tau_n)$  in the preferred isotopy class arbitrarily
close (in the $C^\infty$-topology) to isometries on large neighborhoods of the base points. The
usual argument by contradiction then implies the statement of the theorem.
\end{proof}

\begin{proof}[Proof of Theorem \ref{existence of nearly hyperbolic metrics}]
Let $(X,\lambda)$ be a gluing with $R$-bounded combinatorics, obtained from a
collection $\Xi$ of copies of decorated manifolds via the gluing involution
$\Psi$. Let $\nu=\nu_{(X,\lambda)}=\Psi(\mu_X) \cup \lambda$ be the marking 
on $\D_0 \Xi$ defined in \S\ref{subsec: gluings with bounded combinatorics} 
which is the $\Psi$-image of the decorations on buried components of $\D_0\Xi$ 
and is $\lambda$ on the unburied components.
Recall that for every piece $M$ of $X$, $\nu(M)$, the restriction of 
$\nu$ to $\D_0 M$, has $R$-bounded combinatorics with respect to $M$. 
We define $\tau(M)=\sigma_{\nu(M)}$ to be the corresponding conformal
structure on $\D_0 M$ and denote its restriction to a component 
$F\subset\D_0 M$ by $\tau_F$.
Also assume for every pair of pieces $M,M'$ of $\Xi$, the height of boundary identifications between a component of $\D_0 M$ and $\D_0 M'$ is at least $h_\eta(M)+h_\eta(M')$, where the quantities $h_\eta(M)$ and $h_\eta(M')$ will be determined in what follows.
We consider the convex cocompact structure $Q(M,\tau(M))$ for
every piece of $\Xi$ and construct the negatively curved metric by
gluing these structures. 

Let $M$ and $M'$ be adjacent along
$F\subset \D_0M$ and $\Psi(F) = F'\subset \D_0M'$.
By Proposition \ref{close to a quasifuchsian}
we have embeddings
  \[ T_F: B_{M,\tau(M),F} \to Q(M,\tau(M))\] 
  and 
  \[ T_{F'}: B_{M',\tau(M'),F'} \to Q(M',\tau(M')),\]
in the preferred isotopy classes,
which are arbitrarily close to an isometry on balls of arbitrarily
large radii centered at the base points, provided that the height 
of the boundary identification is sufficiently large, depending 
on $M$ and $M'$.

Recalling that $B_{M,\tau(M),F}$ is the quasi-Fuchsian manifold
$Q(F,\sigma_F,\tau_F)$ and that the restriction $\Psi|_F$ is an 
orientation reversing homeomorphism that takes the ordered pair 
$(\sigma_F,\tau_F)$ to $(\tau_{F'},\sigma_{F'})$, we see that 
$\Psi|_F$ induces an orientation {\em preserving} isometry 
\[ B_{M,\tau(M),F} \to B_{M',\tau(M'),F'}. \]
Composing this with $T_{F'}$ we obtain an embedding
  \[ \hat T_{F}:B_{M,\tau(M),F} \to Q(M',\tau(M'))\]
which is still arbitrarily close to an isometry on a ball of
arbitrarily large radius centered at the base point of
$B_{M,\tau(M),F}$.
We use $T_F$ and $\hat T_F$ to glue $Q(M,\tau(M))$ 
and $Q(M',\tau(M'))$ and obtain a nearly hyperbolic
metric on $M\cup_{\Psi|_F} M'$. 

Using the structure of the models $\BM_F[\sigma_F,\tau_F]$,
the $R$-bounded combinatorics assumption, and the uniform bilipschitz
embedding  
$$\BM_F[\sigma_F,\tau_F]\to B_{M,\tau(M),F},$$
one obtains a lower bound on the injectivity radii in
$B_{M,\tau(M),F}$, and sees
that a neighborhood of sufficiently 
large radius (depending only on $R$ and the genus of $F$) centered at
the base point of $B_{M,\tau(M),F}$
contains a subset $V$ homeomorphic to
$F\times[0,1]$ such that $V\hookrightarrow B_{M,\tau(M),F}$
is a homotopy equivalence and the distance between the boundary
components $\D_-V$ and $\D_+ V$ of $V$ is at least 1. Moreover, the 
injectivity radius lower bound implies the existence of a smooth
bump function $\alpha: V\to [0,1]$ where $\alpha|_{\D_- V} \equiv 0$
and $\alpha|_{\D_+ V} \equiv 1$, satisfying 
uniform bounds on the norms of its
first and second derivatives 
(depending only on $R$ and the genus of $F$).
  
Consider the homeomorphic image $T_F(V)$ of $V$ in $Q(M,\tau(M))$
(respectively $\hat T_{F}(V)$ in $Q(M',\tau(M'))$): it
separates a neighborhood $U$ (resp $U'$) of the end associated to $F$
(resp $F'$) from a core of the manifold. We cut off this neighborhood and identify the remainder
via $\hat T_{F}\circ T_{F}^{-1}$, i.e. 
  \[ X_{(M,M')} = \left(Q(M,\tau(M)) \setminus U\right) \cup_{\hat T_{F}\circ
    T_F^{-1}} \left(Q(M',\tau(M')) \setminus U'\right), \] 
 obviously $X_{(M,M')}$ is a 3-manifold homeomorphic to the interior
 of $M\cup_{\Psi|_F} M'$. Moreover the fact that $T_F$ and $T_{F'}$ 
are in the preferred isotopy classes translates to imply that the
induced bilipschitz map
\[ \BM_M[\tau(M)]\to  \left( Q(M,\tau(M)) \setminus U \right) \hookrightarrow X_{(M,M')} \to M\cup_{\Psi|_F} M'\]
is in the isotopy class of the embedding $M\hookrightarrow M\cup_{\Psi|_F} M'$.

There is a natural embedding of $V$ from
 the construction, whose image we continue to call $V$. The
 complement of $V$ consists of a component $W\subset Q(M,\tau(M))$ and
 a component $W'\subset Q(M',\tau(M'))$. Extend the bump function $\theta$
to be constant on the complement of $V$, and use it
 to produce a Riemannian metric on $X$ which is a convex
 combination of the metrics $\theta_M$ of $Q(M,\tau(M))$ and $\theta_{M'}$ of  $Q(M',\tau(M'))$.
 More precisely, the new metric is 
   \[ \theta(x) = \theta_M(x) + (1-\alpha(x)) \cdot \theta_{M'}(x) \]
for any $x\in V$, and is equal to $\theta_M$ on $W$ and $\theta_{M'}$ on $W'$.  This metric is smooth, and obviously hyperbolic on
the complement of $V$. Moreover when the height of the boundary identification between $F$ and $F'$ is sufficiently large and
$\hat T_{F}\circ T_F^{-1}$ is close enough to an isometry, we can
guarantee that all the sectional curvatures are pinched in the
interval $(-1-\eta,-1+\eta)$.

Theorem \ref{existence of nearly hyperbolic metrics} follows, when we
perform such a gluing for every boundary identification in $X$. All
other requirements follow directly from the construction and the
existence of bilipschitz embeddings $\BM_M[\tau(M)]\to Q(M,\tau(M))$ for
every piece $M$.
\end{proof}

Letting $\eta=1/2$ in Theorem \ref{existence of nearly hyperbolic metrics}, it follows that if $(X,\lambda)$ is a gluing with heights satisfying the hypothesis of the theorem, then it admits a Riemannian metric $\theta_0$ with curvatures pinched in $(-3/2,-1/2)$. Moreover there is a homeomorphism $\BM_X[\lambda]\to (X,\theta_0)$ in the isotopy class of the identity map, whose restriction to the pre-image of every piece $M$ is $K(M)$-bilipschitz. Recall that the model $\BM_X[\lambda]$ divides into vertex pieces, which are copies of the pieces of $X$ and edge pieces, which are interval bundles whose diameter increase as functions of the corresponding heights. Then the image of a vertex piece associated to $M$, is a subset of $(X,\theta_0)$ with diameter bounded depending on $M$. Also if $M$ and $M'$ are adjacent, the distance between the images of their corresponding vertex pieces tends to infinity as a function of the height of the gluing between them. Standard properties of the negative curvature metric then guarantee that given $L$, if the heights of the gluings adjacent to $M$ all are assumed to be large, the homomorphism induced by the inclusion $M\hookrightarrow X$ is $L$-injective. We summarize in the following corollary which resembles results of \cite{Na05} and \cite{NS09}.

\begin{corollary}\label{eventual injectivity}
  Given $L, R>0$, for every decorated manifold $M$ there is a
  constant $h'_L(M)$ so that the following holds. Suppose $X$ is a
  gluing with $R$-bounded combinatorics, with
  the property that if $M$ and $M'$ are adjacent pieces of $X$, the height for
  the identification of a component of $\D_0 M$ and a component of
  $\D_0 M'$ is at least $h'_L(M)+h'_L(M')$. Then for every piece $M$
  of $X$, the homomorphism $\pi_1(M)\to\pi_1(X)$ induced by the
  inclusion $M\hookrightarrow X$ is $L$-injective.  
\end{corollary}

As another corollary of the construction of the negatively curved
metrics, assume $E$ is an incompressible component of $\D_0
M$ with $\pi_1(E)$ infinite idnex in $\pi_1(M)$. We show in the next
theorem that if $M$ is a piece of a gluing $X$ with
$R$-bounded combinatorics and sufficiently large heights,
and $E$ is unburied in $X$, then $E$ is incompressible in $X$.

\begin{theorem}\label{incompressibility of boundary of a gluing}
	Suppose $\CM$ is a finite collection of decorated manifolds and
	$R>0$. There exists $D$ so that if $X$ is an $(\CM,R)$-gluing
	with heights at least $D$, $M$ is a piece of $X$ with
	an incompressible component $E\subset \D_0M$ which is
	unburied in $X$, and $\pi_1(E)$ has infinite index in
	$\pi_1(M)$, then $E$ is incompressible in $X$.
\end{theorem}

\begin{proof}
Consider a sequence $(X_n)$ of $(\CM,R)$-gluings with
heights tending to infinity, assume $M$ is a piece of
each $X_n$, $E\subset\D_0 M$ is incompressible,
$\pi_1(E)$ has infinite index in $\pi_1(M)$, and 
$E$ is unburied in $X_n$ for every $n$.
Equip $\D_0 X_n$ with a complete marking
$\lambda_n$, so that $(X_n,\lambda_n)$ has $R$-bounded combinatorics, 
and heights of $\lambda_n$ tend to infinity
with $n$. 

As before and for every $n$, we consider $\nu_n=\nu_{(X_n,\lambda_n)}$
to be the collection of markings that are obtained either
as restrictions of $\lambda$ or as images of decorations
of pieces of $X_n$ under the gluing involution. In particular
the restriciton of $\nu_n$ to $\D_0 M$ provides a complete
marking $\nu_n(M)$ with $R$-bounded combinatorics and heights
tending to infinity as $n\to\infty$. Similar to the proof of theorem
\ref{existence of nearly hyperbolic metrics}, we define 
$\tau_n(M)=\sigma_{\nu_n(M)}$ the corresponding conformal structure on 
$\D_0 M$.

We may use theorem \ref{existence of nearly hyperbolic metrics}
for the pair $(X_n,\lambda_n)$, with $n$ sufficiently large, to equip
$X_n$ with a metric whose curvatures are pinched in $(-1-\eta_n,
-1+\eta_n)$ with $\eta_n\to 0$ as $n\to\infty$. 
By Theorem \ref{strong limits of convex cocompacts}, there is a
choice of base point $x_n$ for the convex cocompact manifold 
$Q(M,\tau_n(M))$, with respect to which and after passing to
a subsequence, the sequence $Q(M,\tau_n(M))$ converges in
the Gromov-Hausdorff topology to a hyperbolic structure $N_\infty$
on $M$ with simply degenerate ends. The construction of the 
negatively curved metrics in the proof of Theorem \ref{existence
of nearly hyperbolic metrics} guarantee that for large enough $n$,
$X_n$ contains an isometric embedding of the $r_n$-neighborhood of
$x_n$ in $Q(M,\tau_n(M))$ in the preferred isotopy class, where 
$r_n\to\infty$ as $n\to\infty$. Letting $p_n$ be the image of 
$x_n$ we thus have that the pointed manifolds $(X_n,p_n)$ also 
converge in the Gromov-Hausdorff sense to $N_\infty$.

Let $N_E$ be the cover of $N_\infty$ associated to the inclusion of
$\pi_1(E)$ in $\pi_1(M)$ and hence in $\pi_1(N_\infty)$. This is a
Kleinian surface group, with a degenerate end which is an isometric
lift of the corresponding end of $N_\infty$. The infinite-index
assumption, together with Thurston's Covering Theorem implies 
that the other end of $N_E$ is convex cocompact.

Then an argument similar to Namazi-Souto \cite{NS09} shows that
$\pi_1(E)$ must inject in $\pi_1(X_n)$ for $n$ sufficiently
large. We explain the argument briefly. We can choose a compact core
$K_E$ of $N_E$ which is the complement of a product
neighborhood of the degenerate end in 1-neighborhood of the 
convex core of $N_E$. Note that $K_E$ can be identified with 
$E\times[0,1]$, and chosen so that $E\times\{0\}$ is the boundary 
of the $1$-neighborhood of the convex core of $N_E$.
Let $\psi_n:K_E\to X_n$ be the composition of the covering
$N_E\to N_\infty$ and approximating maps $N_\infty\to X_n$.
We can further assume, for $n$ large, that
$\psi_n|_{E\times\{1\}}$ is an embedding whose image 
separates
a product neighborhood $P_n$ of the end of $X_n$ associated to $E$. When $n$
is sufficiently large, the restriction of the approximating maps to
the image of $K_E$ in $N_\infty$ is nearly a local isometry. 
Hence if we pull back via $\psi_n$ the metric of $X_n$ to $K_E$, we
obtain a hyperbolic metric for which the boundary component
$E\times\{0\}$ is still strictly convex. We can attach $P_n$ to
$E\times\{1\}$ via the map $\psi_n$, and the result is a complete
hyperbolic manifold $Y_n \approx  E\times[0,\infty)$ with convex
  boundary, equipped with a locally isometric immersion
  $\psi'_n:Y_n\to X_n$. Now any nontrivial element of $\pi_1(Y_n) =
  \pi_1(E)$ has a geodesic representative in the interior of $Y_n$ which maps to
  a geodesic in $X_n$. 
 It follows that $\psi'_n$ is $\pi_1$-injective, which implies that
 $E$ is an incompressible boundary component of $X_n$.

Now the theorem follows by a typical contradiction argument. If the
theorem were to fail we would have a sequence $X_n$ with $M_n$ and
$E_n$ as in the statement, such that heights go to infinity but $E_n$ is
never incompressible. Since $\MM$ is finite we may extract a
subsequence in which $M_n$ and $E_n$ are copies of a fixed decorated
manifold $M$ with boundary $E$, and then the argument we have given
shows that $E$ is incompressible in $X_n$ after all, for large
$n$. This contradiction proves the theorem. 
\end{proof}

Note that the assumption that the image of $\pi_1(E)$ has infinite
index in $\pi_1(M)$ just rules out the possibility that $M$ is an
interval bundle. 
In fact, we can start by a finitely generated subgroup $\Gamma$ of $\pi_1(M)$
which (up to conjugacy) does not include a finite index subgroup of
a buried peripheral subgroup; then the above proof can be modified to
prove that if $X$ has sufficiently large heights then the 
homomorphism $\pi_1(M)\to \pi_1(X)$ induced by inclusion of $M$ as 
a piece, is injective on $\Gamma$. 
%%%%%%%%%%%%%%%%%%%%%%%%%%%%%%%%%%%
\section{Stability of JSJ decompositions and a priori bounds}\label{sec: stability of JSJ decomposition}

The goal of this section is to prove Theorem \ref{window frames don't break in
  a gluing} below, which gives us some a-priori control of the
geometry of the pieces of a general gluing with $R$-bounded
combinatorics and large heights. %This generalizes

Given a set $\CM$ of decorated manifolds and $R>0$, we define
$\CA_\CM(R,D)$ to be the set of $\CM$-gluings $X$ possibly with
boundary with the property that the boundary identifications adjacent
to compressible boundary components have $R$-bounded combinatorics and
heights at least $D$, and such that all unburied boundary components
of a piece $M$ of $X$ are incompressible in $M$. 
Note that this last property is vacuous for our application of the following
theorem where $\D_0 X$ is empty. 

Given a manifold $M$, we use the notation $AH(M)$ to denote the subset
of the character variety of $M$ that consists of the discrete and
faithful representations of $\pi_1(M)$.  
If $M$ is a decorated  manifold with incompressible
boundary, 
and $U\subset M$ is a submanifold,
we have for each inclusion of $M$ as a piece of a gluing $X$ a restriction map
$$
r_{X,U} : AH(X) \to \XX(U),
$$ 
where $\XX(U)$ denotes the character variety of $\pi_1(U)$, induced by 
the inclusion $U\subset M \subset X$. Note that the representations of
$\pi_1(U)$ obtained in this way need not be injective.

\begin{theorem}\label{window frames don't break in a gluing}
  Suppose $\CM$ is a finite set of decorated manifolds not including
  interval bundles, and $R>0$. Then there exists $D$ so that given
  $X\in \CA_\CM(R,D)$ and $M\in\CM$ identified with a piece of $X$, the following
  holds. 
  \begin{enumerate}[\qquad (i)]
    \item If $C_E$ is a nontrivial relative compression body of $M$, the image of $r_{X,C_E}$ is contained in a compact subset of $\XX(C_E)$ independent of $X$.
    \item If $(U,P_U)$ is an acylindrical component of the JSJ decomposition of the incompressible core of $M$, the image of $r_{X,U}$ is contained in a compact subset of $\XX(U)$ independent of $X$.
    \item There is a uniform upper bound for the $\rho$-length of window frames of the incompressible core of $M$ where $\rho$ is in the image of $r_{X,M}$.
  \end{enumerate}
\end{theorem}

The proof of the theorem essentially uses the existence of
the negatively curved metrics constructed in \S 
\ref{sec:nearly hyperbolic gluings}, as well as Thurston's Only Windows Break Theorem. In fact, we prove a generalization of Thurston's Theorem, Theorem \ref{only transparent windows break} below, and we refer to the appendix for a new proof of Thurston's Theorem with an emphasis on the fact that the constants in the conclusion depend only on the topology of the boundary.

\subsection{Stability of JSJ decompositions}\label{subsec: stability of jsj}
We begin with a theorem  which allows
us to control the JSJ decompositions of bounded-combinatorics,
large-height compressions of a decorated manifold $M$ with incompressible
boundary, in terms of the JSJ decomposition of $M$. 

Suppose $M$ has incompressible boundary and $X$ is a {\em compression} of
$M$ (see \S\ref{subsec: compressions} for definitions). Let $(W,P)$ be
an $I$-bundle or solid torus component of the JSJ decomposition of $M$ (see
\S\ref{subsec: jsj decomposition}). We say that $(W,P)$ is {\em
  transparent} in $X$ if there is an essential annulus or M\"obius
band of $(W,P)$ all of whose boundary lies on unburied components of
$\boundary_0 M$ (see \S\ref{subsec: gluings}). 

If $(W,P)$ is an interval bundle, being transparent in $X$ means its
entire free boundary lies in the unburied part of $\D_0 M$. If $(W,P)$
is a solid torus, transparency means at least two components of $\boundary W
\setminus P$ do so. In the solid torus case, let $(W',P')$ be obtained by
pushing $W$ slightly away from the buried components of $\D_0M$, and
letting $P'$ be the closure of $\boundary W' \setminus \boundary_0
M$. After this adjustment we obtain a collection of $I$-bundles and
pared solid tori properly embedded in $M$ rel its unburied boundary,
and it may be that some 
solid torus components are now isotopic into other
components. Removing such redundancies we obtain what we call the
``induced characteristic submanifold'' of $X$. 
The following theorem states that, assuming
bounded combinatorics and large heights, this is indeed the 
characteristic submanifold of the compression, and hence determines
its JSJ decomposition:

\begin{theorem}\label{stability of JSJ}
  Suppose $M$ is a decorated manifold with incompressible
        boundary which is not an $I$-bundle. Given $R>0$ there
        exists $D$ so that, if $X$ is a compression of $M$
        with $R$-bounded combinatorics and heights at least $D$, then
        $X$ has incompressible boundary and the induced characteristic submanifold
        is the characteristic submanifold of $X$.
\end{theorem}

\begin{proof}
It is the result of Theorem \ref{incompressibility of boundary of a
  gluing} that when $X$ has sufficiently large heights then $\D_0 X$ is
incompressible. 

Partition $\boundary_0
M$ into the subset $B$ of buried components and $U$ of unburied
components. 
Consider first the case that no I-bundle or solid torus of the JSJ
decomposition of $M$ is transparent in $X$. In other words, every
essential annulus or M\"obius band in $M$ has at least one boundary component in $B$.
Then the double $D_U M$ of $M$ along $U$ must be
atoroidal, since an essential torus in minimal position must decompose
into essential annuli whose boundaries only meet $U$. Thus
$D_U M$, together with the doubled decorations of $M$ on $B$,  is a decorated
manifold, and $D_U X$ is obtained as a compression of $D_U M$ using two copies of
the compressions of $M$. We can therefore apply 
Theorem \ref{existence of nearly hyperbolic metrics} to conclude that,
fixing $R$ and supposing the heights of the gluings in $X$ are
sufficiently large, $D_U X$ admits a negatively curved metric. 

The existence of this negatively curved metric implies that $D_U X$ is
atoroidal, and this in turn means that $X$ is acylindrical. 
In particular the JSJ decomposition of $X$ has no
I-bundles or solid tori, which proves the theorem in this case.

In the general case, let $(M',P)$ denote the pared manifold obtained
as the closure of the complement in $M$ of the transparent I-bundles and
(adjusted) transparent solid tori of the JSJ decomposition of $M$, where
$P$ denotes the annuli in $\boundary M'$
along which $M'$ was attached to those components. 
Now letting 
$U'= U\intersect M'$, we have just as before that
$D_{U'} M'$ is atoroidal, except that now it is a manifold with
peripheral tori, namely the doubles of annuli in $P$.
Note that none of
the deleted pieces meet $B$ (this is what the adjustment of the
solid tori accomplishes), and therefore $\boundary P$ is
contained in $U$. The compressions along $B$ are therefore 
undisturbed, and
Theorem \ref{existence of nearly hyperbolic metrics} 
applies. We conclude for large enough gluing heights that
$D_{U'} X'$ admits a complete negatively 
curved metric, where $X'$ is the union of $M'$ with the compression
bodies of $X$ along $B$.

It follows that $(X',P)$ is an acylindrical pared manifold. 
Reattaching the transparent JSJ components to $X'$ along $P$, we obtain
the JSJ decomposition of $X$. 
\end{proof}

{\bf Remark:} The appearance of toroidal boundary components in the
proof above, even when the original decorated manifolds are without
toroidal boundary, has been our motivation for 
choosing this level of generality for decorated manifolds. 

\subsection{Only transparent windows break}\label{subsec: only transparent windows break}
Using Theorem \ref{stability of JSJ} we prove the following, which can be thought
as a generalization of Thurston's Only Windows Break Theorem
to a class of representations of $\pi_1(M)$ that now also include some unfaithful
representations. We should point the reader to examples of
Biringer-Souto \cite{BS10} that show a more obvious generalization of
the theorem is false.

\begin{theorem}[Only Transparent Windows Break]\label{only transparent windows break}
Let $M$ be a decorated manifold with incompressible boundary which is not an $I$-bundle.
Given $R>0$ there exist $D, L$ and a compact subset $\KK_U$ of $\XX(U)$
for each component $(U,P_U)$ of the JSJ decomposition of
$M$, such that,
if $X$ is a compression of $M$ with $R$-bounded combinatorics
and heights at least $D$, then for any $\rho\in AH(X)$ the lengths of
window frames are bounded by $L$, and if $(U,P_U)$ is not a transparent window or solid torus, we have
$$
r_{X,U}(AH(X)) \subset \KK_U.
$$
\end{theorem}

\begin{proof}
Consider a sequence of compressions
$X_n$ of $M$ with $R$-bounded combinatorics and heights going to
$\infty$. For large enough heights, we have by
Theorem \ref{stability of JSJ} that $\D X_n$ is incompressible and the JSJ decomposition of $X_n$ is given by the
transparent $I$-bundles and (adjusted) transparent solid tori of $M$. Passing to a subsequence, we assume this holds for all $n$, and moreover
that the subset $U_n$ of unburied components of $\boundary_0 M$ in $X_n$
is a constant subsurface $U$. The set of transparent $I$-bundles and solid
tori is therefore constant as well. 

Let $(\rho_n)\subset AH(X_n)$ be any sequence. 
It follows from Thurston's Only Windows Break Theorem that the $\rho_n$-lengths of the window frames of $X_n$ stay bounded by a constant that depends only on the
topology of $\D X_n$. Since $\D X_n$ is just $U$, a subsurface of
$\boundary_0 M$, the upper bound does not depend on $X_n$. Note that
these window frames include the frames of transparent $I$-bundles of $M$.

Let $(M',Q)$ denote the pared manifold obtained as the closure of the complement in $M$ of the
transparent $I$-bundles and adjusted transparent solid tori, as in the proof of Theorem
\ref{stability of JSJ}. 
The upper bound for the $\rho_n$-lengths of window frames of $X_n$ in
particular implies that the $\rho_n$-lengths of all components of $Q$
remain bounded. We now wish to show that $\rho_n$ restricted to each
component of $(M',Q)$ remains bounded in its character variety.
Suppose for simplicity that $M'$ is connected.

The maps $\pi_1(M)\to\pi_1(X_n)$ are eventually injective by 
Corollary \ref{eventual injectivity}, and since
$\pi_1(M')\to\pi_1(M)$ is injective, the representations
$\rho_n|_{\pi_1(M')}$ are eventually injective as well.

For each buried component $E$ of $\D_0 M$, 
the decorations $\nu_n(M,E)$ of the
adjacent pieces in $X_n$ have $R$-bounded combinatorics and heights going to
$\infty$ with respect to the decoration $\mu(M,E)$. Hence by Lemma
\ref{limits of bounded combinatorics fill}, they converge to a filling
lamination $\lambda_E$. Let $\ep_n|_E$ denote the image under
gluing of a meridian of the adjacent compression body to $E$ in $X_n$,
which we may select so that $d_E(\ep_n,\nu_n(M,E))$ is bounded. 
Then $\ep_n|_E \to \lambda_E$ as well. 

The lamination $\lambda = \bigcup\lambda|_E$ is binding in $(M',Q)$,
because every essential annulus
in $M'$ must have at least one boundary component on
a buried component $E$ of $\D_0 M$ (since we removed the transparent
windows and solid tori), and $\lambda|_E$ is filling.
The lengths of the components of  $\ep_n$ are bounded in $\rho_n$ --
in fact they are zero.
Hence by theorem
\ref{eventually faithful convergence} the sequence $(\rho_n|_{\pi_1(M')})$ is
contained in a compact subset of the character variety of
$\pi_1(M')$. 
The case where $M'$ is disconnected is handled
in the same way, component by component.

Every component $(U,P_U)$ of the JSJ decomposition of  $M$, which is
not a transparent $I$-bundle or solid torus,  
is contained in a component of $M'$, and every window frame is either
contained in $M'$, or is the frame of a transparent window. In all these
cases we have shown that $\rho_n$ restricted to such components is
bounded in its character variety. Since this is true for an arbitrary
sequence of $X_n$ and $\rho_n$ with heights going to infinity, the
usual argument by contradiction yields a uniform bound for all
sufficiently large heights.
\end{proof}

\subsection{Boundedness in general gluings}\label{subsec: boundedness in general}
Recall that compressions of a decorated manifold $M$ are special
examples of gluings. We will use Theorem \ref{only transparent windows break}
to prove Theorem \ref{window frames don't break in a gluing},  which
establishes a similar conclusion
for representations of $\pi_1(M)$ induced by elements of $AH(X)$ with $X\in\CA_\CM(R,D)$.

\begin{proof}[Proof of Theorem \ref{window frames don't break in a gluing}.]
Given $\CM$ as in the hypothesis of the theorem, we begin by describing
a new set $\closure\CM$ of decorated manifolds, every element of which
is either a nontrivial compression body or a manifold with
incompressible boundary which is not an interval bundle.

Recall from \S \ref{subsec: compression bodies}, the unique decomposition 
of each element $M$ of $\CM$ into its incompressible core and a
union of nontrivial relative compression bodies.
We ignore components of the incompressible core which are interval
bundles. Define decorations on each remaining component as follows: each
boundary component which is also a boundary component of $M$ inherits
its decoration from $M$. For each remaining boundary component we fix
some arbitrary selection of a decoration (note that this is a finite
number of choices). Let $\closure\CM$ denote the set of all decorated
manifolds obtained in this way from the elements of $\CM$.

Each $\CM$-gluing $X$ can therefore be further decomposed into an
$\closure\CM$-gluing, and we will call this the {\em full decomposition}
of $X$. Note that
in the full decomposition, whenever two boundaries are identified and
one or both of them are compressible (i.e. exterior boundaries of
compression bodies), the decorations on {\it both} boundaries are
inherited from decorations of elements of $\CM$ containing them. 
Moreover the disk sets on these boundaries are unchanged.
Hence the height and bounded combinatorics conditions for
boundary identifications adjacent to compressible boundary components
remain the same and in particular if $X$ is in $\CA_\CM(R,D)$, then it
is also in $\CA_{\closure\CM}(R,D)$. 
  
Given an $\CM$-gluing $X$, we reorganize pieces of the full
decomposition by performing only the boundary identifications adjacent
to exterior boundaries of compression bodies. (Recall that these are
the only compressible boundary components of elements of
$\closure\CM$.) At the end of this process, every component $Y$ of the
new decomposition either contains exactly one piece $M'\in\closure\CM$
with incompressible boundary and $Y$ is a compression of $M'$ with
respect to the compression bodies in $\closure\CM$, or $Y$ is a union
of compression bodies where the exterior boundaries of exactly two
``central" ones are glued together and the exterior boundary of every
other compression body is glued to the interior boundary of some 
compression body. The latter type can also be viewed as a compression
of either of the two central compression bodies. Obviously every such
component $Y$ is a finite union of elements of $\closure\CM$. It will
turn out that the height requirement in the statement of the theorem
for the boundary identifications within $Y$ only depends on the pieces
of $Y$. We say this is the {\em maximal-compression decomposition} of
$X$ and obviously every component $Y$ of this decomposition also
belongs to $\CA_{\closure\CM}(R,D)$ if $X\in \CA_\CM(R,D)$. 
   
It follows from theorem \ref{incompressibility of boundary of a gluing} that there exists a height $h_Y$ depending on $R$ and $\CM$, so that if the boundary identifications within $Y$ have $R$-bounded combinatorics and heights at least $h_Y$, then $\D Y$ is incompressible. So we assume $D$ is larger than this height and as a result $\pi_1(Y)$ injects into $\pi_1(X)$; so, an element of $AH(X)$ gives a discrete and faithful representation of $\pi_1(Y)$, i.e. an element of $AH(Y)$.
  
Suppose $\rho\in AH(X)$ is given and $Y$ is piece of the
maximal-compression decomposition of $X$. If $Y$ is a compression of
a component $M'$ of the incompressible core of an element $M\in\CM$,
we can apply the Only Transparent
Windows Break Theorem \ref{only transparent windows break} to conclude that,
if $D$ is sufficiently large, the window frames of $M'$ will have
$\rho$-length bounded by a constant, and the induced representations
of acylindrical components of the JSJ decomposition of $M'$ stay in
compact subsets of the associated character varieties. These bounds
depend on $R$ and the subset of $\closure\CM$ whose elements are used
as pieces of $Y$. In particular we have proved claims {\em (ii)} and
{\em (iii)} for the
components of the incompressible core of pieces of $X$ and window
frames of those pieces. It  remains to prove claim {\em (i)}, that the induced
representations of the relative compression bodies stay in compact
subsets of the associated character varieties. 
  
To prove this, we need to show that for every compression body $C$ in
$\closure \CM$, the representations induced by hyperbolic structures
on elements of $\CA_\CM(R,D)$ stay in a compact subset of $\XX(C)$. We
prove this by backward induction on the genus of the exterior boundary
of $C$. Assume we already know this claim for compression bodies whose
exterior boundary has genus bigger than the genus of $\D_e C$ (note
that this holds vacuously in the base case, when the genus of $\D_e C$
is maximal in $\closure\CM)$. Suppose a sequence $X_n$ is given with
$X_n \in \CA_\CM(R,D_n)$ where $D_n\to\infty$ as $n\to\infty$ such that
$C$ is a piece of the full decomposition of $X_n$ for every $n$, and
let $\rho_n\in AH(X_n)$ also be given.  Let $Y_n$ denote the component
of the maximal-compression decomposition of $X_n$ which contains
$C$. As we explained above, for $n$ sufficiently large, say all $n$,
$\rho_n|_{\pi_1(Y_n)}$ is an element of $AH(Y_n)$. After passing to a
subsequence we can assume the exterior boundary of $C$ is glued to the
boundary $E$ of a piece $M$ of $Y_n$ with $M\in\closure\CM$ fixed, via
a boundary identification $\psi_n:\D_e C\to E$. Recall that $\D_e C$
and $E$ have inherited decorations from decorations of elements of
$\CM$, which we denote by $\mu_C$ and $\mu_E$ respectively, such that
$\mu_C$ and $\psi_n^{-1}(\mu_E)$ have $R$-bounded combinatorics and
their distance in $\CC(\D_e C)$ is at least $D_n$.
We have three
possibilities for $M$:
  \begin{enumerate}
    \item $M$ is a component of the incompressible core of an element of $\CM$,
    \item $M$ is a nontrivial relative compression body of an element of $\CM$ and $E$ is an interior boundary component of $M$ or
    \item $M$ is a nontrivial relative compression body of an element of $\CM$ and $E$ is the exterior boundary of $M$.
  \end{enumerate}
  In the first case, note that since $M$ is not an interval bundle,
the boundary component $E$ is not entirely inside the window part
of $M$. As a consequence, there is an essential simple loop $\gamma_E$
on $E$ which either can be homotoped into an acylindrical component of
the JSJ decomposition of $M$ or is homotopic to a window frame. Then it 
follows from the proof of claims {\em (ii)} and {\em (iii)} in the
conclusion of the theorem that the $\rho_n$-length of $\gamma_E$ stays bounded in
  $X_n$ independently of $n$.  
  
  In the second case, the genus of the exterior boundary of $M$ is bigger than that of $C$ and therefore the induction hypothesis shows that the induced representations of $\pi_1(M)$ stay bounded. So if we choose a fixed curve $\gamma_E$ on $E$ (say a component of the decoration on $E$) then the $\rho_n$-length of $\gamma_E$ in $X_n$ is bounded independently of $n$.
  
  Finally in the third case, $E$ is compressible in $M$ and therefore we can choose a fixed meridian $\gamma_E$. The $\rho_n$-length of $\gamma_E$ in $X_n$ is zero and is obviously bounded independently of $n$.
  
  Hence in any of the cases above we have made a choice for $\gamma_E$ whose $\rho_n$-length is bounded independently of $n$.
  The $R$-bounded combinatorics and increasing heights show, by lemma \ref{limits of bounded combinatorics bind}, that the sequence $(\psi_n^{-1}(\mu_E))_n$ converges in $\PML(\D_e C)$ to a filling Masur domain lamination $\lambda_C$ on $\D_e C$. In particular $\lambda_C$ is binding on $\D_0 C$.
  The curve $\gamma_E$ is chosen only depending on the piece $M$ and therefore has bounded curve complex distance from $\mu(E)$; hence $(\psi_n^{-1}(\gamma_E))_n$ also converges to the same lamination $\lambda_C$ and we have already seen that $\rho_n(\gamma_n)$ is bounded independently of $n$. So by theorem \ref{eventually faithful convergence}, the induced representations of $\pi_1(C)$ stay in a compact subset of $\XX(C)$.
\end{proof}

Our discussion in the above proof yields a fact which we record here
for use in the proof of the main theorem:

\begin{corollary}\label{existence of peripheral bounded length curves}
    Suppose $\CM$ is a set of decorated manifolds that does not
    include $I$-bundles, and $R>0$. There exist constants $D, L$, and
    for every component $E$ of $\D_0 M$ with $M\in\CM$, there is an
    essential simple loop $\gamma_E$, so that if $M$ is a piece of
    $X\in \CA_\CM(R,D)$ and $\rho\in AH(X)$, then the $\rho$-length of
    $\gamma_E$ is bounded by $L$. 
\end{corollary}

%%%%%%%%%%%%%%%%%%%%%%%%%%%%%%%%%%%
\section{Bilipschitz models for bounded type manifolds}\label{sec:bilipschitz models for bounded type manifolds}

In this section we prove the main theorem. 
Note that in the statement of the theorem and throughout this section
we always assume $X$ is a gluing with
$\D_0 X$ empty, i.e. $X$ has only toroidal boundary. It is easy to
generalize the theorem to the case when $\D_0 X$ is nonempty and
equipped with an assignment of end invariants, but for simplicity of 
presentation we avoid this generality.  

\begin{theorem}\label{bilipschitz models for bounded type manifolds}
  Let $\CM$ be a finite collection of decorated manifolds and
  $R>0$. There exist $D$ and $K$ such that, for any $(\CM,R)$-gluing
  $X$ with all heights greater than $D$, the interior of $X$ admits a
  unique hyperbolic metric $\sigma$. Moreover there exists a $K$-bilipschitz
  embedding $\BM_X\to (X,\sigma)$, in the correct isotopy class,
  whose image is the complement of the rank 2 cusps of $X$. 
\end{theorem}

\begin{proof}\mbox{}
This will be a cumulation of our results in the previous sections. 
Assume $X$ is an $(\CM,R)$-gluing with $\D_0 X$ empty obtained as 
the identification space $\Xi/\Psi$ for a collection $\Xi$ of 
copies of elements of $\CM$ and the gluing involution 
$\Psi:\D_0\Xi\to\D_0\Xi$, all satisfying the properties of 
gluings mentioned in \S \ref{subsec: gluings}.

%%%%%%%%%%%%%%%%%%%%%%%%

\subsection*{Existence of the nearly hyperbolic metrics}
By Theorem \ref{existence of nearly hyperbolic metrics} given
finite $\CM$ and $R>0$, there are $D_0$ and $K_0$ so that if all heights
of $X$ are greater than $D_0$, then $X$ admits a $(-3/2,-1/2)$-pinched
negatively curved metric $\theta_0$ and there is a $K_0$-bilipschitz 
embedding $f_0$ from $\BM_X$ onto the complement of the rank 2 cusps 
of $X$ equipped with $\theta_0$. The map $f_0$ is in the isotopy class
of the identity map, but we will only use the fact that it is
homotopic to the identity map. 

Assume from now on that heights of $X$ are at 
least $D_0$. Corollary \ref{eventual injectivity} implies that 
given $L$, we can choose the heights sufficiently large, so that 
the embeddings of each piece of $X$ into $X$ induces an $L$-injective 
homomorphism on the level of fundamental groups. 

As a consequence of the existence of the negatively curved metrics, when 
$X$ is compact, i.e. $\Xi$ has finitely many pieces, by Perelman's 
Geometrization Theorem the interior of $X$ admits a finite volume 
hyperbolic metric, which by Mostow rigidity is unique up to homotopy.

The existence of the hyperbolic metric and the uniqueness in the general 
case, when $\Xi$ has infinitely many components, will be a consequence 
of proving the theorem for the compact 
case and will be explained at the end. So we continue with the assumption 
that $X$ (whether compact or non-compact) admits a complete hyperbolic metric.

%%%%%%%%%%%%%%%%%%%%%

\subsection*{Reduction to the case without $I$-bundles}
Recall that results of \S \ref{sec: stability of JSJ decomposition} 
apply when there are no 
$I$-bundle pieces. In view of Lemma \ref{looking through} on
collapsing $I$-bundles, we may reduce to the case where
either $X$ (and $\CM$)  has no $I$-bundles, or that
$X$ is a fibered gluing (a gluing of a single $I$-bundle to
itself). In the latter case, the existence of the bilipschitz map from  
$\BM_{X}$ to $X$ with its hyperbolic metric is a consequence of work 
of Minsky in \cite{Min01}. 
(It can also be proved as a consequence of theorem \ref{interpolations}.) 

Thus, for the rest of the proof of 
Theorem \ref{bilipschitz models for bounded type manifolds} 
we may assume that 
$X$ and $\CM$ have no $I$-bundle pieces.

%%%%%%%%%%%%%%%%%%%%%

\subsection*{$L$-injectivity along the gluing regions}
Assume $E$ is a component of $\D_0 M$ for a piece $M$
of $X$. Recall that $\mu_E$ is the decoration
on $E$ and $\nu_E$ is the $\Psi$-image of the 
decoration of the piece adjacent along $E$. We use
$\sigma_E=\sigma_{\mu_E}$ and $\zeta_E=\sigma_{\nu_E}$
to denote the corresponding conformal structures on
$E$, then $\BM_E[\mu_E,\nu_E]=\BM_E[\sigma_E,\zeta_E]$
is the subset of $\BM_X$ associated to $E$. So if $\tau$
is a hyperbolic structure on $E$ associated to a point 
of the Teichm\"uller geodesic $[\sigma_E,\zeta_E]$,
the $K_0$-bilipschitz map $f_0$ from $\BM_X$ to $(X,\theta_0)$ 
restricts to a bilipschitz embedding of $\tau$ into 
$(X,\theta_0)$. Furthermore if the $\CT(E)$-distances 
between $\tau$ and both ends of the geodesic 
$[\sigma_E,\zeta_E]$ are at least $d$, then the image of 
this embedding is enclosed in a product region which gets 
larger as $d\to\infty$. For a given $L$, 
we can choose $d$ large enough so that the induced 
homomorphism $\pi_1(E)\to\pi_1(X)$ is $L$-injective with 
respect to $\tau$. We assume this choice of $d$ works for 
every $M\in\CM$ and every component of $\D_0 M$. 

%%%%%%%%%%%%%%%%%%%%%%%

\subsection*{Lifted embeddings of the middle surfaces}
Given $E\subset\D_0 M$ for a piece
$M$ of $X$, let $\tau_E$ denote
the midpoint of the the Teichm\"uller geodesic 
$[\sigma_E,\zeta_E]$. By the previous step given $L$, we can
choose $d$ so that if the height of the gluing is at least 
$2d$, then the homomorphism $\pi_1(E)\to\pi_1(X)$ is 
$L$-injective with respect to $\tau_E$. (Obviously 
$\Psi(\tau_E) = \tau_{\Psi(E)}$.)

Let $M'$ be the piece of $X$ containing $E'=\Psi(E)$ in its boundary.
By corollary \ref{existence of peripheral bounded 
length curves}, we can choose simple essential loops $\gamma_E$ 
and $\gamma_{E'}$ on $E$ and $E'$ respectively, so that their 
lengths are bounded uniformly for every hyperbolic structure on 
an $(\CM,R)$-gluing that contains $M$ and $M'$ as pieces and 
has sufficiently large heights. So assuming the heights of $X$
are larger than this required height, lengths of $\gamma_E$
and $\gamma_{E'}$ are bounded in $X$. 

Let $\ep_E =\nu_{\tau_E}$ be the shortest marking on $\tau_E$,
chosen in \S \ref{subsec: teichmuller}; we construct a marking $\gamma$ on the 
boundary of the decorated $I$-bundle $I_E[\epsilon_E]$,
whose restriction to one boundary is $\gamma_E$ and to the other
boundary is $\Psi(\gamma_{E'})$. 
Obviously as the height of $\Psi|_E$ tends to infinity, the 
heights of $\gamma$ tend to infinity. Also it follows from the
discussion in \S \ref{subsec: teichmuller} that
there exists $R'$ depending on $R$ and topology of $E$, so that 
$\gamma$ has $R'$-bounded combinatorics as a marking on $I_E[\ep_E]$. 
Finally the homomorphism $\pi_1(E)\to\pi_1(X)$ is induced by the 
embedding $E\hookrightarrow M \hookrightarrow X$ and therefore by 
lemma \ref{embedding implies fibration}, either this homomorphism 
is not a virtual fibration or it is a fibration and the image of 
$E$ is the fiber. In either case the hypothesis of theorem 
\ref{thick lifted embeddings of I-bundles} holds for $I_E[\ep_E]$, 
the homomorphism $\pi_1(E)\to \pi_1(X)$ and the marking $\gamma$. 
We conclude there exist constants $D_1\ge D_0$ 
and $K_1$ so that, if heights of $X$ are at 
least $D_1$, there is a lifted $K_1$-bilipschitz embedding
$\tau_E \to X$ in the homotopy class determined by the 
inclusion $E\hookrightarrow M\hookrightarrow X$. (Note that 
Theorem \ref{thick lifted embeddings of I-bundles} provides a lifted 
bilipschitz embedding of $\BM_{I_E[\ep_E]} = \BM_E[\ep_E]$, 
but we know there is a uniformly bilipschitz map 
$\tau_E\to\sigma_{\ep_E}$ isotopic to the identity on $E$ and
therefore we also have a lifted bilipschitz embedding of $\tau_E$.)
As before we assume the constants $D_1$ and $K_1$ are chosen 
in a way that they work for every pair $M, M'$ of copies
of elements of $\CM$ which appear as pieces of a $\CM$-gluing
$X$.

Repeating this for every component of $\D_0 M$, we obtain a 
hyperbolic structure $\tau_M=\bigcup_{E\subset\D_0 M} \tau_E$ 
on $\D_0 M$ and a lifted $K_1$-bilipschitz embedding 
\[ g_\tau:\tau_M\to X\] whose restriction to 
every component $E$ of $\D_0 M$ is the above lifted $K_1$-bilipschitz embedding
$\tau_E\to X$. We can repeat this for every piece of $X$
and further assume that $g_\tau|_E\equiv g_\tau|_{E'}\circ\psi_E$,
when $\psi_E:E\to E'$ is a boundary identification in $X$.

%%%%%%%%%%%%%%%%%%%%%%%%%%

\subsection*{Applying the immersion theorem}
We next claim that with heights of $X$ sufficiently large, 
we can apply Theorem \ref{bilipschitz embedding of models} to
$M$, the homomorphism $\rho:\pi_1(M)\to\pi_1(X)$ induced by the 
inclusion $M\hookrightarrow X$, the constant 
$R''=\max\{R,K_1\}$, and the lifted $K_1$-bilipschitz embedding 
$g_\tau$ of $\tau_M$. In order to satisfy condition (3) in the
statement of Theorem \ref{bilipschitz embedding of models} for
$L'$, we choose $d$ with the property that for every
component $E$ of $\D_0 M$ and point $\sigma$ in the Teichm\"uller
geodesic $[\sigma_E, \tau_E]$, whose distance from 
$\sigma_E$ is at least $d$, the homomorphism $\pi_1(E)\to\pi_1(X)$ 
is $L'$-injective with respect to $\sigma$. (Recall that $\tau_E$ 
was the midpoint of the Teichm\"uller geodesic connecting $\sigma_E$
and $\zeta_E=\Psi(\sigma_{\Psi(E)})$. 

We can choose $D_2\ge \max\{D_1,2d\}$ large enough so that if
heights of $X$ are at least $D_2$, then by using the first step of
the proof above, the induced homomorphism $\pi_1(M)\to\pi_1(X)$ is
$L$-injective where $L$ is given in Theorem \ref{bilipschitz embedding 
of models}. Then the hypothesis of Theorem \ref{bilipschitz 
embedding of models} is satisfied for constants $R'',L',d$ and 
$L$ and we obtain a locally $K_2$-bilipschitz immersion 
$f_M:\BM_M[\tau_M]\to X$ in the homotopy class of the inclusion
$M\hookrightarrow X$ that extends $g_\tau$ and restricted 
to each toroidal component of $\BM_M[\tau_M]$ covers the boundary 
of a component of the thin part of $X$; the constant $K_2$ 
depends only on $M$ and $R''$. In fact we can assume it is 
chosen to work for every $M\in\CM$ and therefore ultimately 
it depends on the collection $\CM$ and on $R$.

%%%%%%%%%%%%%%%%%%%%%%%%%%

\subsection*{Patching pieces of the model} Once the locally
bilipschitz immersions $f_M:\BM_M[\tau_M]\to X$ are 
constructed, we are in the situation described 
in \S \ref{subsec: consistent orientation};
as explained there, it is a consequence of part (3) of 
Theorem \ref{interpolations} that the immersions $f_M$ for 
different pieces of $X$ have consistent orientation. Hence 
patching them together yields a locally $K_2$-bilipschitz 
immersion $f: \BM_X\to X$. Since each map $f_M$ is in the 
homotopy class determined by the inclusion $M\hookrightarrow X$ 
of $M$ as a piece of $X$, the map $f$ is a homotopy 
equivalence. Moreover restricted 
to every toroidal boundary component of $X$, the map $f$ 
is a covering map to the boundary 
of a component of the thin part of $X$. Since $f$ is homotopy 
equivalence, such a component of the thin part has to be a rank 2 
cusp. So $f$ induces a homotopy equivalence from 
$\BM_X$ to the complement of the rank 2 cusps of $X$. 
When $X$ is compact it immediately follows that $f$ is
in fact a homeomorphism and provides a $K_2$-bilipschitz 
embedding $\BM_X\to X$ whose image is the complement of the
rank 2 cusps of $X$. In the general case, the same conclusion
will follow after proving that $f$ is proper. 

%%%%%%%%%%%%%%%%%%%%%%%%%%%%%%%%%
\subsection*{The noncompact case}
When $X$ is compact and all heights are at least $D_2$, we have shown so
far that there is a hyperbolic metric on $X$ (using geometrization) which is 
unique up to homotopy by Mostow Rigidity Theorem. We
have established a $K_2$-bilipschitz embedding
$f$ from $\BM_X$ to $X$ equipped with this hyperbolic metric
and $f$ is in the homotopy class of the identity map.

When $X$ is noncompact (and heights are at least $D_2$), the two core
elements of our proof still apply: we obtain a
$(-3/2,-1/2)$-pinched metric $\theta_0$ on $X$ together
with a $K_0$-bilipschitz embedding $f_0:\BM_X\to (X,\theta_0)$
homotopic to the identity; and {\em provided} $X$
admits a hyperbolic structure $\sigma$, we have a locally bilipschitz
immersion $f:\BM_X \to (X,\sigma)$ homotopic to the identity.
Let us show that $f$ is proper.

The composition $g=f\circ f_0^{-1}$ is a locally bilipschitz immersion from
the complement of the rank 2 cusps of $(X,\theta_0)$ to $(X,\sigma)$,
which is also homotopic to the identity.
Assume $(x_n)$ is a sequence
of points of $X$ that exit every compact set. It is clear from 
the existence of the bilipschitz map $f_0$, that based at every point of
$(X,\theta_0)$ and in particular based at every $x_n$, there exists a 
homotopically nontrivial loop $\alpha_n$ of $\theta_0$-length bounded
depending only on the collection $\CM$ and therefore independent of $n$.
Moreover, the negative curvature of $\theta_0$ and existence of a lower 
bound for the injectivity radius in $(X,\theta_0)$ outside the cusps, implies $\alpha_n$
and $\alpha_m$ are non-homotopic when $x_n$ and $x_m$ are far from
each other; so after passing to a subsequence we can assume the
sequence of loops $(\alpha_n)$ are pairwise non-homotopic. Since
$g$ is homotpic to the identity map, $g(\alpha_n)$ is homotopic
to $\alpha_n$ and therefore is not homotopic to $g(\alpha_m)$ for
$m >> n$. Therefore at most finitely many loops $g(\alpha_n)$ 
can belong to a given compact subset of $X$. This implies $g(x_n)$ 
exits every compact subset of $X$ and therefore $g$ is 
proper. Thus $f$ must be proper as well, and as we mentioned 
earlier it follows that $f$ is an embedding
whose image is the complement of the rank 2 cusps of $X$.

We can now show that a
hyperbolic metric on $X$ is unique up to homotopy if it exists.
If there are two complete hyperbolic metrics 
$\theta_1,\theta_2$ on $X$, we can use the $K_2$-bilipschitz
embeddings $f_i:\BM_X\to (X,\theta_i), i=1,2,$ to find a 
bilipschitz homeomorphism between the complements of rank 2 
cusps of $\theta_1$ and $\theta_2$. This map can be extended
to a bilipschitz map $(X,\theta_1)\to(X,\theta_2)$. There
is a global upper bound for the injectivity radius of $\BM_X$
and therefore there is one for $\theta_1$ and $\theta_2$ outside the cusps.
Then McMullen's Rigidity Theorem \cite{McM96} applies to show 
the bilipschitz map $(X,\theta_1)\to(X,\theta_2)$
is homotopic to an isometry.

In fact the existence of a hyperbolic metric on $X$ is also
a consequence of the construction of the bilipschitz maps.
We realize $X$ as a limit of compact gluings $X_n$ with $\D_0 X_n$ empty, so that
the models $\BM_{X_n}$ converge geometrically to $\BM_X$. 
This can be done by taking compact approximations, which are 
$\CM$-gluings with nonempty nontoroidal boundary and extend them in an 
arbitrary way to $\CM$-gluings $X_n$ without non-toroidal boundary with 
the same bounded combinatorics and height properties. Being compact
the hyperbolic metric on each $X_n$ is unique and we have constructed
$K_2$-bilipschitz embeddings $\BM_{X_n}\to X_n$ onto the 
complement of the rank 2 cusps of $X_n$.

It is immediate from the construction of the models $\BM_{X_n}$ that 
with an appropriate choice of base points, they converge geometrically
to the model $\BM_X$. This implies immediately that after passing
to a subsequence and choosing base points in $X_n$, obtained from images
of base points of $\BM_{X_n}$, the sequence of hyperbolic manifolds $X_n$ 
converges geometrically to a hyperbolic manifold. We can also assume the 
bilipschitz embeddings $\BM_{X_n}\to X_n$ converge to an embedding of
$\BM_X$ to this limit. The limits of rank 2 cusps of $X_n$ will be rank
2 cusps of the limit and therefore the image of every toroidal boundary
component of $\BM_X$ is the boundary of a rank 2 cusp of the limit. 
This implies immediately that the limit is homeomorphic to $X$ and we 
have equipped $X$ with a complete hyperbolic structure. 

%%%%%%%%%%%%%%%%%%%%%%%%%%%%
\subsection*{The isotopy class}

It only remains to prove that 
the hyperbolic metric on $X$ is also unique up to {\em isotopy} and 
the bilipschitz map $\BM_X\to X$
is in the {\em isotopy} class of the identity map. These are both
consequences of the following lemma:

\begin{lemma}\label{homotopy implies isotopy}
	Given a collection $\CM$ of decorated manifolds and $R>0$, there
exists $D_3$, so that if $X$ is an $(\CM,R)$-gluing with heights
at least $D_3$, then every self-homeomorphism of $X$ homotopic to
the identity is isotopic to the identity. 
\end{lemma}

\begin{proof}
	When $X$ is compact and $D_3\ge D_2$, we have shown that $X$
admits a hyperbolic structure. Then the claim follows from the 
results of Gabai-Meyerhoff-Thurston \cite{GMT}. 

	So assume $X$ is noncompact. Similar to the arguments above,
we continue with the assumption that $\CM$ has no interval bundles.
Recall from the proof of Theorem
\ref{window frames don't break in a gluing} the construction of 
the maximal-compression decomposition of $X$. 
Also recall there exists $D_3>0$ 
depending on $\CM$ and $R$, so that if heights of $X$ are at least
$D_3$, then every component $Y$ of the maximal-compression decomposition
is a compact 3-manifold with incompressible boundary. 
Suppose the heights of $X$ are at least $D_3$ and $\phi:X\to X$
is a homeomorphism homotopic to the identity; then for every
component $E$ of $\D Y$, the restriction $\phi|_E$ is homotopic
to the inclusion $E\hookrightarrow X$. Since $E$ is incompressible,
by a theorem of Waldhausen \cite{Wa} $\phi|_E$ is 
isotopic to the inclusion $E\hookrightarrow X$. Therefore we can
change $\phi$ with an isotopy and assume $\phi|_E$ is the inclusion
$E\hookrightarrow X$ for every component $E$ of $\D Y$. Then the 
restriction $\phi|_Y$ is a self-homeomorphism of $Y$
which is the identity map on $\D Y$. Again using a classical
result of Waldhausen \cite{Wa} in 3-manifold topology, we
see that $\phi$ 
is isotopic to the identity map on $Y$. 
Arguing similarly for
every component of the maximal-compression decomposition proves the lemma.
\end{proof}

This concludes the proof of the theorem. 
It is worth a final remark that, as the proof of Lemma \ref{homotopy
  implies isotopy} shows,
even when $X$
is compact but the number of pieces of $X$ which are not interval bundles
is sufficiently large, depending on $\CM$, the maximal-compression decomposition
of $X$ has more than one component. Thus, assuming
heights of $X$ are at least $D_3$, we obviously see that $X$ is Haken.
In this case, then, Thurston's hyperbolization theorem applies to
establish the existence of the hyperbolic structure on $X$, and
Waldhausen's theorem gives us uniqueness up to isotopy. Hence with
this further assumption on the number of pieces we do not need
to use Perelman's proof of geometrization or Gabai-Meyeroff-Thurston.
\end{proof}

\subsection*{Boundedness of vertex representations}
We want to point out a consequence of the proof above which may be of independent interest. This basically states that for every element $M$ of a finite collection $\CM$ of decorated manifolds and $R>0$, the representations of $\pi_1(M)$, induced by hyperbolic structures on $(\CM,R)$-gluings, are contained in a compact subset of $\XX(M)$, the character variety of $\pi_1(M)$. In the case of $(\CM,R)$-gluings, this improves our result in Theorem \ref{window frames don't break in a gluing}.
  \begin{theorem}
    Given a finite collection $\CM$ and $R$ there exists $D$, so that for every $M\in\CM$, there is a compact subset of $\XX(M)$ which contains the representation of $\pi_1(M)$ induced by the hyperbolic structure on every $(\CM,R)$-gluing with heights at least $D$, which contains $M$ as a piece.
  \end{theorem}
    
\begin{proof}
	As in the proof above, we assume $X$ has no $I$-bundles, except possibly for $M$. To prove the claim by means of contradiction suppose $M\in\CM$ is given and $X_n$ is a sequence of $(\CM,R)$-gluings that contain $M$ as a piece and whose heights tend to infinity as $n\to\infty$ and each is equipped with a hyperbolic metric. Let $\rho_n:\pi_1(M)\to \PSL_2(\BC)$ denote the representation induced by $X_n$ and we let $\nu_n(M)=\nu_{X_n}(M)$ as in the definition of gluings to be the image of the decorations of the adjacent pieces under the gluing involution. Because of the assumption on the heights and bounded combinatorics and by using lemma \ref{limits of bounded combinatorics bind}, we can pass to a subsequence and assume $\nu_n(M)$ converges to a binding lamination $\lambda$ on $\D_0 M$. 
  
  By corollary \ref{existence of peripheral bounded length curves}
there exist constants $D,L$ so that for every element $M'\in\CM$,
which is not an $I$-bundle, and boundary component $E'\subset\D_0 M'$,
there exists an essential simple loop $\gamma_{E'}$, so that if $M'$
is a piece of an element $X\in \CA_{\CM}(R,D)$, then the length of
$\gamma_{E'}$ in a hyperbolic structure on $X$ is bounded by $L$. For
$n$ sufficiently large, say for all $n$, $X_n\in
\CA_{\CM}(R,D)$. Given a component $E$ of $\D_0 M$, assume in $X_n$,
$E$ is identified with a component $E'_n$ of $\D_0 M'_{E,n}$ where
$M'_{E,n}$ is also a piece of $X_n$. Then we can choose
$\gamma'_{E,n}=\gamma_{E'_n}$ on $E'_n$. The curve complex distance
between $\gamma'_{E,n}$ and $\nu_n(M,E)=\nu_n(M)|_E$ is bounded
independently of $n$, so the sequence $(\gamma'_{E,n})_n$ also
converges to $\lambda|_E$ in $\PML(E)$. Since the $\rho_n$-length of
$\gamma'_{E,n}$ is bounded independently of $n$ and the union of those
over all components of $\D_0 M$ converge to the binding lamination
$\lambda$ on $\D_0 M$, and $(\rho_n)$ is eventually faithful, we can
use theorem \ref{eventually faithful convergence} and conclude that
the sequence $(\rho_n)$ stays in a compact subset of
$\XX(M)$. Therefore $(X_n)$ could not have been a sequence of
counterexamples. This proves the theorem. 
\end{proof}

\section*{Appendix: Thurston's Only Windows Can Break Theorem}\label{sec: only windows can break}

We give a new proof of the following theorem of Thurston
\cite{ThuIII}. The arguments of this section are independent of the
other parts of the paper.
Thurston's original proof, which made an elegant use of
area growth rates for branched pleated surfaces,
was never published. 
The proof presented here uses somewhat more elementary ideas
which were also developed by Thurston and in some ways are closer to
ideas in his earlier work. We also emphasize in the statement
and proof that the constant in the conclusion only depends on the
topology of the boundary of $M$ (Thurston's proof also implies this,
though it was not explicitly stated). This dependence is crucial in 
section \ref{sec: stability of JSJ decomposition}.

\begin{theorem*}[Thurston's Only Windows Can Break Theorem]\label{only windows can break}
  Given a compact irreducible atoroidal 3-manifold $M$ with incompressible boundary,
  there exists a constant $C$ depending only on the topology of $\D_0
  M$, so that for every discrete faithful representation
  $\rho:\pi_1(M)\to \PSL_2(\BC)$ the $\rho$-length of the window frame
  of $M$ is at most $C$. 

  Moreover if $(U,Q)$ is a pared acylindrical
  component of the JSJ decomposition, then the induced representations
  of $\pi_1(U)$ stay in a compact subset of the character variety of
  $U$.
\end{theorem*}
Although the statement and proof can easily be generalized to the
pared setting, it will suffice for our needs to restrict to the case
that the pared locus of $M$ consists only
of toroidal boundary components of $M$.

\begin{proof}
Let $N_\rho = \BH^3/\rho(\pi_1(M))$ and choose a component $\gamma_1$
of the window frame (recall from \S\ref{subsec: jsj decomposition} the
details of the JSJ decomposition).
Note that we can extend $\gamma_1$ to a set $\gamma_1,\gamma_2,\ldots,
\gamma_k$ with $k\ge 2$, of simple closed curves on the boundary
which are pairwise freely homotopic in $M$ but not in $\D M$. Also
each $\gamma_i$ is either a component of the window frame or is
in a toroidal component of $\D M$. To prove the first part of the
theorem, we can assume $\gamma_1$ does not represent a parabolic
element of $\rho(\pi_1(M))$, otherwise the $\rho$-length of this
component would be zero. In particular, we assume $\gamma_i, 
1\le i\le k,$ is not on a toroidal component of $\D M$. Also
the representative of the free homotopy class of $\gamma_1$ in 
$N_\rho$ is a closed geodesic $\gamma^*$ of length $L>0$. 
We need to show $L$ is bounded from above by a constant depending 
only on the topology of $\D_0 M$.

This will follow from the following claim:
\begin{claim*}
If $L$ is bigger than a constant depending on the topology of $\D_0 M$, there exists a non-contractible closed curve $\alpha_1$ on $\D_0 M$ with $i(\alpha_1,\gamma_1)\neq 0$ so that
	\begin{itemize}
		\item[(i)] either $\alpha_1$ is a boundary component of an essential immersed annulus in $M$, or 
		\item[(ii)] $\alpha_1$ is non-primitive in $M$ while it is primitive in $\D_0 M$, where by a primitive loop we mean one that represents an indivisible conjugacy class.
	\end{itemize} 
\end{claim*}

Before proving this claim, we show how the first conclusion of the
theorem follows from it. It is a standard topological
construction that if $\alpha_1$ is primitive in $\D_0 M$ but
non-primitive in $M$, there exists an essential immersed annulus and
$\alpha_1$ can be homotoped into a regular neighborhood of this
annulus. Recall from \S \ref{subsec: jsj decomposition} that
non-boundary-parallel annuli can be pushed via a homotopy rel $\D_0 M$
into the characteristic submanifold of the JSJ decomposition of $M$,
and therefore can be made disjoint from the window frames. This
contradicts the assumption $i(\alpha_1,\gamma_1)\neq 0$ and proves $L$
is bounded above by a constant depending on the topology of $\D_0 M$.

\begin{proof}[Proof of the claim]

Recall that a {\em pleated surface} is a map $f:S\to N$, from a
surface $S$ equipped with a complete hyperbolic metric $\sigma_f$ to a
hyperbolic 3-manifold $N$, which preserves the length of paths and
such that every point $x$ is contained in at least one open geodesic segment of
$\sigma_f$ which is mapped by $f$ to a geodesic in $N$. (See
\cite{ThuI, CEG} for more on pleated surfaces.) We say $f$ {\em
  realizes} a multi-curve $\alpha$ if the image of every component of
the geodesic representative of $\alpha$ in $\sigma_f$ is mapped to a
closed geodesic in $N$. 

There is a homotopy equivalence $\iota: M\to N_\rho$ which  induces $\rho$
on the level of fundamental groups. We say a map from a
subset of $M$ to $N_\rho$ is {\em in the homotopy class of $\rho$} if
it is homotopic to the restriction of $\iota$.  Similar to
\cite{ThuIII}, one can construct a pleated surface $f:\D_0 M\to N$ in
the homotopy class of $\rho$, that realizes
$\gamma_1\cup\gamma_2\cup\cdots\cup\gamma_k$. We let $\Sigma$ denote
the complete hyperbolic surface obtained from $\D_0 M$ equipped with
$\sigma_f$.  Obviously $f$ identifies
$\gamma_1,\gamma_2,\ldots,\gamma_k$ with $\gamma^*$. Abusing notation slightly, we parametrize $\gamma^*$ by arclength as a map
$\gamma^*:[0,L]\to N_\rho$.
We extend $f$ to
a map $F:M\to N_\rho$ in the homotopy class of $\rho$. 

Let $A\subset
M$ be an essential annulus in $M$ with boundary components $\gamma_1$
and $\gamma_2$. 
We parametrize $A$ by a map $K:[0,L]\times[1,2]\to A$,
whose restriction to $[0,L]\times\{i\}, i=1,2$, is an arclength parametrization
of $\gamma_i$ (which we will also
denote by $\gamma_i$). After homotopy, we may assume 
$F\circ K$ is constant along vertical arcs $\{t\}\times[1,2]$, or equivalently
$F\circ K$ factors as
the composition $\gamma^*\circ q_1$, where
$q_1:[0,L]\times[1,2]\to[0,L]$ is the projection to the first
factor. This is possible using the fact that $N$ is atoroidal.
Note that because $M$ has incompressible boundary and $A$ is
essential, $A$ admits no boundary compressions; in particular given
$t\in[0,L]$, the properly embedded arc $K(\{t\}\times[1,2]$ is not
homotopic rel endpoints to an arc in $\D_0 M$.

For $\ep>0$ smaller than the
Margulis constant $\ep_M$ for $\Hyp^3$ and $X$ a hyperbolic 2- or
3-manifold, let $X^{\le\ep}$ denote the closure of the $\ep$-thin part of $N$.  
We argue first in the case that either $\gamma_1$ or $\gamma_2$
intersects $\Sigma^{<\ep'}$ for some $\ep'$ depending
only on the topology of $\D_0 M$. The following lemma is an
observation of Thurston \cite{ThuI}: 

\begin{lemma*}\label{thick goes to thick}
  Given a surface $S$ and $\ep>0$, there exists a positive constant
  $\ep'\le \ep$ so that if $f: S\to N$ is a $\pi_1$-injective pleated
  map from $S$ to a hyperbolic 3-manifold $N$, and $S$ is equipped
  with the induced hyperbolic metric $\sigma_f$, the image of the
  $\ep$-thick part of $S$ does not enter the $\ep'$-thin part of $N$. 
\end{lemma*}

Fixing $\ep<\ep_M$, use the above lemma with $S=\D_0 M$ to choose
$\epsilon'$. Suppose $\gamma^*$ 
intersects a component $T^{\epsilon'}$ of 
$N_\rho^{\le\ep'}$, so that
by the lemma the $f$-pre-images of a point $p$ of the intersection
gives points $p_1\in\gamma_1$ and $p_2\in\gamma_2$ which are contained
in components $V_1^\ep$ and $V_2^\ep$ of $\Sigma^{\le\epsilon}$. Hence 
there are essential simple loops $\alpha_1\subset V_1^\ep$ and
$\alpha_2\subset V_2^\ep$ based at $p_1$ and $p_2$ respectively, whose
lengths are at most $\ep$ and whose $f$-images are contained in the component
$T^\ep$ of $N_\rho^{\le\epsilon}$ which contains $T^{\epsilon'}$. 
We can assume $\ep$ is chosen small enough, so that on a hyperbolic
surface a simple geodesic that enters a component of the $\ep$-thin
part is either homotopic to the core, or has to cross it. If
$\gamma_1$ or $\gamma_2$ is homotopic to the core of a component of
$\Sigma^{\le\ep}$, then its length is at most $\ep$ and this bounds
the length of $\gamma^*$ by $\ep$. So we continue with the assumption
that $i(\gamma_1, \alpha_1)$ and $i(\gamma_2,\alpha_2)$ are
nonzero. If either one, say $\alpha_1$, is non-primitive in $M$ then
the conclusion of the above claim is satisfied; so we continue with
the assumption that both $\alpha_1$ and $\alpha_2$ are primitive in
$M$. 

If $T^\ep$ is a rank 2 cusp in $N_\rho$ then, since $\rho$ is an
isomorphism from $\pi_1(M)$ to $\pi_1(N_\rho)$, $T^\ep$ has to
correspond to a toroidal component $T$ of $\D M$. Since $F$ is a
homotopy equivalence, $\alpha_1$ can be homotoped into $T$ in $M$. So
there is an essential immersed annulus in $M$ that connects $\alpha_1$
to $T$ and we have established the claim. Therefore
we are allowed to assume $T^\ep$ is either a rank 1 cusp or is a
Margulis tube, and in any case has a cyclic fundamental group.  

Since $F$ is a homotopy equivalence and both $\alpha_1$ and $\alpha_2$ are primitive in $M$, their images $f(\alpha_1)$ and $f(\alpha_2)$ are also primitive in $N_\rho$; therefore $f(\alpha_1)$ and $f(\alpha_2)$ are freely homotopic to the generator of $\pi_1(T^\ep)$ and are freely homotopic to each other. Another appeal to the homotopy equivalence $F$ shows $\alpha_1$ and $\alpha_2$ are freely homotopic in $M$. Then either there is an immersed essential annulus providing the free homotopy between $\alpha_1$ and $\alpha_2$ and the conclusion of the claim holds, or $\alpha_1$ and $\alpha_2$ are freely homtopic in $\D_0 M$. Hence, we continue with the assumption that $\alpha_1$ and $\alpha_2$ are freely homotopic in $\D_0 M$ and in particular $V_1^\ep= V_2^\ep$. This also implies there exists an arc $\beta\subset V_1^\ep$ connecting $p_1$ and $p_2$ whose image in $T^\ep$ is a homotopically trivial loop based at $p$. To construct such an arc, start from just any arc $\beta'\subset V_1^\ep$ that connects $p_1$ and $p_2$. The $f$-image of $\beta'$ is a loop based at $p$ in $T^\ep$. 
Since $\pi_1(T^\epsilon,p)$ is generated by $f(\alpha_1)$ as a loop based at $p$, there is an interger $k$, so that the $f$-image of the concatenation $\beta=\beta'*(\alpha_1)^k$ is a homotopically trival loop based at $p$. 

Recall that the map $K:[0,L]\times[1,2]\to A$ parametrizes the annulus
$A$ in a way that $F\circ K$ is identical to the map $\gamma^*\circ
q_1$. We can assume $K(\{t\}\times[1,2])=p$ for some
$t\in[0,L]$, $K(t,1)=p_1$, and $K(t,2)=p_2$. Consider the loop
$\beta*K(\{t\}\times [1,2])$ based at $p_1$, whose $F$-image is the
same loop as the $f$-image of $\beta$ and is therefore homotopically
trivial. Because $F$ is a homotopy equivalence,
$\beta*K(\{t\}\times[1,2])$ is homotopically trivial in $M$. This will
imply however that $K(\{t\}\times[1,2])$ is homotopic (rel endpoints)
to an arc in $\D_0 M$ which we claimed is impossible. This
contradiction rules out the case that $V_1^\ep = V_2^\ep$. 

Thus from now on we may assume that $\gamma^*$ does not enter the $\ep'$-thin part
of $N_\rho$. Because the pleated map $f$ is $\pi_1$-injective and
preserves lengths of curves, this implies that both $\gamma_1$ and
$\gamma_2$ have to stay in the $\ep'$-thick part of $\Sigma$. 

The area of $\Sigma$ is $-2\pi\chi(\D M)$ and therefore the volume of
the product $\Sigma^{\ge\ep'}\times\Sigma^{\ge\ep'}$, equipped with
the product metric, is bounded from above by 
 \[ (-2\pi\chi(\D M))^2=4\pi^2\chi^2(\D M).\] 
Given $(x,y)\in
\Sigma^{\ge\ep'}\times\Sigma^{\ge\ep'}$, there are embedded disks
$B(x,\ep')$ and $B(y,\ep')$ of radius $\epsilon'$ centered at $x$ and
$y$ whose areas are equal to $4\pi \sinh^2(\epsilon'/2)$; so the
volume of the product of these disks is $16\pi^2\sinh^4(\epsilon'/2)$
and $\Sigma^{\ge\ep'}\times\Sigma^{\ge\ep'}$ contains at most
  \[ L'=\frac{4\pi^2 \chi^2(\D M)}{16\pi^2\sinh^4(\ep'^2)} = \frac{\chi^2(\D M)}{4\sinh^4(\ep'^2)} \]
pairwise disjoint products of the form $B(x,\ep')\times B(y,\ep')$. 
As above we write $\gamma_i=K(\cdot,i)$, and  in particular $f\circ\gamma_i \equiv \gamma^*$.
If $L\ge 5L'$, there are positive constants
$a_1,a_2,a_3,a_4,a_5\in[0,L]$, so that for every pair $j\neq k$,
$|a_j-a_k|\ge 1 (\mod L)$ and $B(\gamma_1(a_j),\ep')\times
B(\gamma_2(a_j),\ep')$ intersects $B(\gamma_1(a_k),\ep')\times
B(\gamma_2(a_k),\ep')$; equivalently $\gamma_i(a_j)$ and
$\gamma_i(a_k)$ are closer than $\ep'$ in $\Sigma$ for $i=1,2$. 

Since
$\gamma_1$ is embedded, this implies that for $j,k\in\{1,\ldots,5\}$,
short subsegments of $\gamma_1$ around $a_j$ and $a_k$ are
parallel. We can then find three of them, say $a_1,a_2,a_3$, so that
$\gamma_1$ crosses the
transversal geodesic arc of length $\le\ep'$, that connects
$\gamma_1(a_j)$ and $\gamma_1(a_k)$, in the same direction. Then by a
similar analysis for $\gamma_2$, we can choose two of them, say $a_1 <
a_2$, so that $\gamma_2$ also intersects the geodesic arc of length
$\le\ep'$, that connects $\gamma_2(a_1)$ and $\gamma_2(a_2)$, in the
same direction.  

Let $\mu_i$ denotes the transversal geodesic arc of length $\le\ep'$ that connects $\gamma_i(a_1)$ and $\gamma_i(a_2)$, $i=1,2$. Also let $\alpha_i = \gamma_i([a_1,a_2])*\mu_i$ be the closed curve on $\Sigma$ which is the concatenation of $\gamma_i[a_1,a_2]$ and the transversal geodesic arc $\mu_i$. When $\epsilon'$ is small and because $a_2-a_1\ge 1$, a standard argument using the geometry of the hyperbolic plane shows that every lift $\cover\alpha_i$ of $\alpha_i$ is a quasi-geodesic in $\cover\Sigma$, the universal cover of $\Sigma$. 
Moreover if we consider a lift $\cover\gamma_i$ of $\gamma_i$ which intersects $\cover \alpha_i$ in a lift of $\gamma_i([a_1,a_2])$, the geodesic $\cover\gamma_i$ separates the two limit points of $\cover\alpha_i$ in the boundary at infinity $\D_\infty \cover\Sigma$ of $\cover\Sigma$. Hence $i(\gamma_i, \alpha_i)\neq 0$. 

Since $f(\gamma_1(a_1)) = f(\gamma_2(a_1))$ and $f(\gamma_1(a_2)) = f(\gamma_2(a_2))$ are on $\gamma^*$, they belong to $N_\rho^{\ge\epsilon'}$ and therefore $f(\mu_1)$ and $f(\mu_2)$ are homotopic rel endpoints in $N_\rho$. As a result if $\delta$ is the concatenation 
$\mu_1* K(\{a_1\}\times [1,2]) * \mu_2 * K(\{a_2\}\times[1,2])$, then $F(\delta)$ is homotopically trivial in $N_\rho$. 
As before $F$ is a homotopy equivalence and therefore $\delta$ is
homotopically trivial in $M$. The concatenation of the disk in $M$
bounded by $\delta$ and the sub-rectangle $K([a_1,a_2]\times[1,2])$ of
$A$, which connects $\gamma_1([a_1,a_2])$ and $\gamma_2([a_1,a_2])$,
gives an immersed annulus $A'$ that is a free homotopy between
$\alpha_1$ and $\alpha_2$. 

We see that $A'$ is essential, because otherwise the vertical arc 
$K(\{a_1\}\times[1,2])$ would be homotopic into $\D_0M$ rel endpoints,
contradicting the fact that $A$ is essential. Since
$i(\alpha_1,\gamma_1)\ne 0$, conclusion (i) of the claim is satisfied,
provided $L\ge 5L'$. 
\end{proof}

The last statement in the conclusion of the theorem follows easily from the first and theorem \ref{eventually faithful convergence}. Suppose $(U,Q)$ is an acylindrical pared component of the JSJ decomposition of $M$ and $(\rho_n)$ is a sequence of discrete faithful representations of $\pi_1(M)$. Since $(U,Q)$ is acylindrical, the empty set is a binding lamination on $\D_0(U,Q)$. Also by the first part of the theorem, the $\rho_n$-lengths of window frames of $M$ are bounded uniformly, which implies that for every annular component of $Q$, we can select a core curve whose $\rho_n$-length is bounded uniformly for all $n$; so by theorem \ref{eventually faithful convergence} and after conjugation and passing to a subsequence, the restrictions $(\rho_n|\pi_1(U))$ of $(\rho_n)$ to $\pi_1(U)$ is convergent. This proves such restrictions stay in a compact subset of $\XX(\pi_1(U))$ as claimed.
\end{proof}

%%%%%%%%%%%%%%%%%%%%%%%%%%%%%%%%%%

\newcommand{\etalchar}[1]{$^{#1}$}
\ifx\undefined\bysame
\newcommand{\bysame}{\leavevmode\hbox to3em{\hrulefill}\,}
\fi

\end{document}